\documentclass[]{article}

\usepackage[utf8]{inputenc}

\usepackage{fullpage}

\usepackage{amsmath}
\usepackage{amssymb}
\usepackage{amsthm}
\usepackage{enumitem}
\usepackage{tikz}
\usetikzlibrary{calc,positioning,fit}
\usetikzlibrary{arrows,decorations.markings}
\usepackage{url}
\usepackage{comment}

\usepackage{boxedminipage}

\usepackage{listings}

\usepackage{ifpdf}
\usepackage{hyperref}

\usepackage{todonotes}

\newtheorem{remark}{Remark}
\newtheorem{theorem}{Theorem}
\newtheorem{lemma}[theorem]{Lemma}
\newtheorem{example}[theorem]{Example}

\newtheorem{definition}[theorem]{Definition}
\newtheorem{proposition}[theorem]{Proposition}
\newtheorem{observation}[theorem]{Observation}
\newtheorem{claim}[theorem]{Claim}

\newcommand{\N}{\mathbb N}

\newcommand{\Z}{\mathbb Z}

\newcommand{\balpha}{\boldsymbol{\alpha}}
\newcommand{\bbeta}{\boldsymbol{\beta}}
\def\a{\mathbf a}

\def\c{\mathbf c}

\newcommand{\beeq}{\begin{eqnarray*}}
	\newcommand{\eneq}{\end{eqnarray*}}

\newcommand{\be}{\begin{equation}}
\newcommand{\ee}{\end{equation}}

\theoremstyle{remark}
\newtheorem{rmk}[theorem]{Remark}
\newtheorem{obsalt}[theorem]{Observation}

\usepackage{xcolor}
\definecolor{darkred}{cmyk}{.3,.9,.80,.2}

\title{On the extremal families  for the  Kruskal--Katona theorem}
\author{Oriol Serra\thanks{Department of Mathematics, Universitat Polit\`ecnica de Catalunya, Barcelona.\texttt{oriol.serra@upc.edu}. Supported by the Spanish Research Agency under project MTM2017-82166-P}
\and Llu\'{i}s Vena
\thanks{Department of Mathematics, Universitat Polit\`ecnica de Catalunya, Barcelona. \texttt{lluis.vena@gmail.com}. Supported by the a Beatriu de Pin\'os Fellowship BP2018-0030 of the AGAUR, Horizon's 2020 program cofund.}
}

\date{\today}

\begin{document}

\maketitle
\begin{abstract}
	In  \cite[Serra, Vena, Extremal families for the Kruskal-Katona theorem]{sv21}, the authors have shown a characterization of the extremal families for the Kruskal-Katona Theorem. We further develop some of the arguments given in \cite{sv21} and give additional properties of these extremal families.
F\"uredi-Griggs/M\"ors theorem from 1986/85 \cite{furgri86,mors85} claims that, for some cardinalities, the initial segment of the colexicographical is the unique extremal family; we extend their result as follows: the number of (non-isomorphic) extremal families strictly grows with the gap between the last two coefficients of the $k$-binomial decomposition.

We also show that every family is an induced subfamily of an extremal family, and that, somewhat going in the opposite direction, every extremal family is close to being the inital segment of the colex order; namely, if the family is extremal, then after performing $t$ lower shadows, with $t=O(\log(\log n))$, we obtain the initial segment of the colexicographical order.
We also give a ``fast'' algorithm to determine whether, for a given $t$ and $m$, there exists an extremal family of size $m$ for which its $t$-th lower shadow is not yet the initial segment in the colexicographical order.

As a byproduct of these arguments, we give yet another characterization of the families of $k$-sets satisfying equality in the  Kruskal--Katona theorem. Such characterization is, at first glance, less appealing than the one in \cite{sv21}, since the additional information that it provides is indirect. However, the arguments used to prove such characterization provide additional insight on the structure of the extremal families themselves.
\end{abstract}

\tableofcontents


\section{Introduction}

The well--known Kruskal--Katona  Theorem \cite{kat09,krus63} on the minimum shadow of a family of $k$--subsets of   $[n]=\{1,2,\ldots ,n\}$ is a central result in Extremal Combinatorics with multiple applications, see e.g. \cite{Frankl1991}.
The \emph{shadow}  of a  family $S\subset \binom{[n]}{k}$ 
is the family $\Delta (S)\subset \binom{[n]}{k-1}$  of 
$(k-1)$--subsets  which are contained in some set in $S$. 
The Shadow Minimization Problem asks for the minimum cardinality of $\Delta (S)$ for families of $k$-sets $S$ with a given cardinality $m=|S|$. 
The answer given by the Kruskal--Katona theorem  can be stated in terms of  $k$--binomial decompositions. The 
\emph{$k$--binomial decomposition} 
of a positive integer $m$ is 
\begin{equation}
\label{eq:k-bin-seq}
m=\binom{a_0}{k}+\binom{a_{1}}{k-1}+\cdots +\binom{a_{t}}{k-t},
\end{equation}
where the coefficients satisfy  $a_0>a_1>\cdots >a_{t}\ge k-t\ge 1$, $t\in [0,k-1]$, are   uniquely determined by $m$ and $k$. $t+1$ is said to be the length of the sequence k-binomial decomposition $(a_0,a_1,\cdots,a_{t})$.

\begin{theorem}[Kruskal--Katona \cite{zbMATH03489128,krus63}, and Theorem~2.1 in \cite{furgri86}]\label{thm:kk}
	Let $S\subseteq \binom{[n]}{k}$ be a non-empty family of $k$--subsets of $[n]$  and let  	
	$$
	m=\binom{a_0}{k}+\binom{a_{1}}{k-1}+\cdots +\binom{a_{t}}{k-t}
	$$
	be the $k$--binomial decomposition of $m=|S|$.
	Then 
	\begin{equation}
	\label{eq:kk_ineq}
	|\Delta S|\ge \binom{a_0}{k-1}+\binom{a_{1}}{k-2}+\cdots +\binom{a_{t}}{k-t-1}.
	\end{equation}
	and, more generally
	\begin{equation} \label{eq:fur_grigs_21}
	|\Delta^i S|\ge \binom{a_0}{k-i}+\binom{a_{1}}{k-1-i}+\cdots +\binom{a_{t}}{k-t-i}.
	\end{equation}
	where $\Delta^i(S)$ is the $i$--th iterated shadow of $S$  defined recursively by $\Delta^i(S)=\Delta (\Delta^{i-1}(S))$, $1\le i\le k-1$, $\Delta^{0}(S)=S$.
	Moreover, if \eqref{eq:kk_ineq} is satisfied with equality, then \eqref{eq:fur_grigs_21} is also satisfied with equality, for all $i\geq 0$.
\end{theorem}
The  ``moreover'' part is \cite[Theorem~2.1]{furgri86}, the previous \eqref{eq:kk_ineq} is Kruskal--Katona's theorem properly \cite{zbMATH03489128,krus63}. What here is denoted by $\Delta^i$, in  \cite{furgri86} is denoted by $\Delta^{k-i}$. Note that Kruskal--Katona's theorem \cite{zbMATH03489128,krus63} is used in the proof of \cite[Theorem~2.1]{furgri86}.

In this context, we say that a family $S$ is \emph{extremal} 
if the cardinality of its lower shadow achieves the lower bound in \eqref{eq:kk_ineq} from Theorem \ref{thm:kk}. 
For every $m$, the initial segment of length $m$ in the \emph{colex order} is an extremal family. We recall that the colex order 
on the $k$--subsets of $[n]$ is defined by $x\le_{\text{colex}} y$ if and only if $\max((x\setminus y)\cup (y\setminus x))\in y$. We denote by  $I_{n,k}(m)$   the initial segment  of length $m$ in the colex order in $\binom{[n]}{k}$. In what follows, we refer to initial segments up to automorphisms of the Boolean lattice, induced by any permutation of $[n]=\{1,2,3,\ldots,n\}$. We recall that the shadow of an initial segment in the colex order in $\binom{[n]}{k}$ is again an initial segment in the colex order in $\binom{[n]}{k-1}$. This shows that \eqref{eq:kk_ineq} is tight.

F\"uredi and Griggs \cite{furgri86} (see also M\"ors \cite{mors85}) proved that, for cardinalities $m$ for which the $k$--binomial decomposition has length $t+1<k$ or $a_{k-1}=a_{0}-k+1$, these initial segments $I_{n,k}(m)$ are in fact the unique extremal families (up to automorphisms of the Boolean lattice);  the authors also give some examples which show that this is not the case when $t+1=k$ and $a_{k-1}<a_{0}-k+1$, thus characterizing  those cardinalities for which the initial segment in the colex order is the unique example of an extremal family (see Theorem~\ref{thm:card} below).

Let us introduce some notation that will be extensively used throughout the paper. $\a=(a_0,\ldots,a_t)$ denotes a sequence of integers, $\ell(\a)=t+1$ its length, while
\[
\binom{\a}{k}=\binom{a_0}{k}+\binom{a_1}{k-1}+\cdots+\binom{a_t}{k-t}
\]
For  positive integer $m$, we write $\binom{\a}{k}\stackrel{b}{=}m$ as the k-binomial decomposition of $m$; then $\a$ is the unique strictly decreasing sequence of $\leq k$ positive integers satisfying $\binom{\a}{k}=m$.

\begin{theorem}[\cite{furgri86},\cite{mors85}] \label{thm:card} Let $n\ge k> 0$ and $0<m\le \binom{n}{k}$. Let $\binom{\a}{k}\stackrel{b}{=} m$ be the $k$--binomial decomposition of $m$. 
	The initial segment of length $m$ in the colex order in $\binom{[n]}{k}$  is the unique  extremal set (up to isomorphism) 
	\begin{center}if and only if  \end{center}
	\[\ell (\a)<k \;\; \text{ or }\;\; m=\binom{q}{k}-1=\binom{q-1}{ k}+\cdots+\binom{q-k}{1}, \text{ for some positive integer $q$, with} k< q\leq n\]
\end{theorem}
We note that, for fixed $k$, the set of integers with $k$--binomial decomposition of length $k$ has upper asymptotic density one. Thus, the unicity of the extremal families can be ensured only on a thin set of cardinalities.

 Theorem~\ref{thm:card} prompted F\"uredi and Griggs to ask about the characterization of  the  extremal families. We answered this question in \cite[Theorem~3]{sv21}.
  The aim of this paper is to further study the relation between the binomial decompositions, the extremal families and their relation with the colex order.
 
 \subsection{Summary of the results}
 
  We begin by giving an additional numerical characterization of the extremal families, which is inspired by how in \eqref{eq:kk_ineq} the binomial decomposition of $|S|$ is transformed to give a bound on $|\Delta(S)|$.
   Our approach is based on the following variation on the $k$--binomial decomposition sequence of a positive integer that is uniquely determined by the family $S$ of $k$--subsets and not only on its size. 
  The \emph{shadow $k$-binomial decomposition} of $S\subseteq \binom{[n]}{k}$ is the sequence of (non necessarily positive) integers 
   $\beta_0>\beta_1>\ldots >\beta_{k-2}\ge \beta_{k-1}$ 
   such that, for each $0\le i \le k-1$,
   $$
   |\Delta^{i}(S)|=\binom{\beta_0}{k-i}+\binom{\beta_1}{k-i-1}+\cdots +\binom{\beta_{k-1}}{-i+1}.
   $$
   (We recall that, for all integers $n\in \Z$ and each positive integer $k$,  $\binom{n}{k}=\frac{n (n-1)\cdots (n-k+1)}{k!}$, while  $\binom{i}{0}=1$ for all $i\in \Z$ and $\binom{i}{k}=0$ for $k<0$.)
   The shadow $k$-binomial decomposition of a family $S$ of $k$--subsets can be alternatively defined recursively by
   \begin{equation}\label{eq.finding_b} 
   \beta_i=
   \begin{cases}
   |\Delta^{k-1}(S)| -1, &\text{$i=0$, if $k\geq 2$}\\
   |\Delta^{k-i-1}(S)|-\left(\binom{\beta_0}{i+1}+\cdots +\binom{\beta_{i-1}}{2}\right)-1, & \text{$1\le i< k-1$,}\\ 
   |S|-\left(\binom{\beta_0}{k}+\cdots +\binom{\beta_{k-2}}{2}\right) & \text{$i=k-1$}
   \end{cases}
   \end{equation}

\begin{theorem}[Characterization of extremal families]\label{thm:charac}
	$S\subset \binom{[n]}{k}$ is an extremal family (extremal for the Kruskal--Katona theorem, Theorem~\ref{thm:kk}) if and only if $\beta_{k-1}\ge 1$, where $\beta_{k-1}$ is the last element of the shadow $k$-binomial decomposition of $S$.
  \end{theorem}

The computation of the shadow binomial decomposition of a family $S$ can be obtained through the \emph{hypotenusal process} described in Section \ref{sec:hypotenusal} using an initial configuration derived from the Stanley-Reisner complex.  For each family $S$, the hypotenusal process provides  an infinite sequence $(w_1,w_2,\ldots)$ of nonnegative integers from which one can derive the shadow binomial decomposition of $S$ (see Proposition~\ref{prop.translation}~\ref{en:prop_trans_2} or Theorem~\ref{thm:hypotenusal} below).  The name \emph{hypotenusal process} arises from the sequence of so--called hypotenusal numbers introduced by Lucas \cite{lucas1891theorie} which can be obtained from the hypotenusal process using a particular initial condition.
The following theorem gives an alternative proof to Theorem~\ref{thm:kk}.

\begin{theorem}\label{thm:hypotenusal} Let $S\subset \binom{[n]}{k}$ be a family $k$--subsets of $[n]$. Let $(w_1,w_2,\ldots )$ be the sequence generated by the hypotenusal process of $S$. Then the shadow $k$--binomial decomposition of $S$ is
	$$
	(n-w_1-1,n-w_2-2,\ldots ,n-w_{k-1}-(k-1), n-w_k-(k-1)).
	$$ 
\end{theorem}

The properties of the sequence $(w_1,w_2,\ldots )$, and in particular its growth, provide a number of quantitative and structural results for extremal families, such as Theorem~\ref{thm:depth} below. 

From the definition of $\binom{\a}{k-1}$ when $\binom{\a}{k}\stackrel{b}{=}m$, it is clear that  all initial segments $I_{n,k}(m)$ with $m$ having a $k$--binomial decomposition sequence of the form  $(a_0,a_1,\ldots ,a_{k-2},i)$ for $1\le i< a_{k-1}$ have the same lower shadow. Let $m_0$ (resp. $m_1$) be the cardinality given by the $k$-binomial decomposition $(a_0,a_1,\ldots ,a_{k-2},1)$ (resp. $(a_0,a_1,\ldots ,a_{k-2},a_{k-1})$)  then (as the shadow should be smaller or equal, yet Kruskal--Katona's theorem ensures that the first would not be the case) every family of the form $I_{n,k}(m_1)\setminus J$ is extremal for every $J$ of size $\leq m_1-m_0$ . These are the examples given in \cite[Example~2.4]{furgri86}\footnote{In our case we are considering only the initial segment of the colex order as an extremal family, yet in \cite[Example~2.4]{furgri86} the authors consider the general case when the initial family is just extremal, not necessarily the initial segment in the colex order.} of extremal families $S$ different from initial segments $I_{n,k}(m)$; when considering the initial family to be the colex, then $\Delta (S)$ is an initial segment in the colex order and one can say that these examples are given by small perturbations of the initial segment in the colex order.

Since the $(k-1)$--iterated shadow $\Delta^{k-1}(S)$ of a family $S$ are the singletons, which are isomorphic to the initial segment in the colex order, it is reasonable to measure how far is an extremal family from an initial segment in the colex order by the smallest $t$ such that $\Delta^t(S)$ is an initial segment of the colex order. From this perspective, extremal families can not be very far away from initial segments in the colex order. 

\begin{theorem}\label{thm:depth} Let $S\subset \binom{[n]}{ k}$ be an extremal family with cardinality $m$ and $\binom{\a}{k}\stackrel{b}{=} m$ be its $k$-binomial decomposition. Then there exists $c$, a positive constant that depends only on $k$, $c=c(k)$, such that, 
	\begin{center}
		for every $t\ge c\log\log n$, we have
		$\Delta^t(S)\cong I_{n,k-t}\left(\binom{\a}{k-t}\right).$
	\end{center}
\end{theorem}  
The next result gives a more precise necessary numerical condition Theorem~\ref{thm:depth} than for the existence of an extremal family $S$ such that $\Delta^t(S)$ is not an initial segment in the colex order.
\begin{theorem}[Hypotenusal numbers]\label{thm:hn}
	Let $S\subset \binom{[n]}{k}$ be an extremal  family of $k$-sets.
	Let $t$ be such that $\Delta^t(S)$ is not an initial segment of the colex order and let	
\begin{equation}\label{eq:hypot_ref}
	|S|=\binom{a_0}{k}+\cdots+\binom{a_{k-1}}{1},
\end{equation}
	be the $k$--binomial decomposition of $|S|$.
	Then we have
	\begin{equation}\label{eq:hn}
	a_{k-2-t+i}-a_{k-1-t+i}\ge a[i-1]+1, \;\; 1\le i < t, \qquad \text{ and }\qquad a_{k-2}-a_{k-1}\ge a[t-1],
	\end{equation}
	where $\{a[n]\}_{n\geq -1}$, with $a[-1]=1$, is the sequence of hypotenusal numbers.\footnote{The hypotenusal numbers can be found in Lucas \cite{lucas1891theorie}, and Sylvester and Hammond \cite{sylvester1887x,sylvester1888iv}. It is the sequence A001660 in the Encyclopedia of Integer Sequences \cite{oeis}.} Moreover, there are some instances of $t$ and $(a_0,\ldots,a_{k-1})$ for which the inequalities of \eqref{eq:hn} are satisfied with equality.
\end{theorem}

Another result which illustrates  the robustness of the colex order is that, eventhough for most cardinalities the colex is not the unique extremal family with such cardinality, it holds that, for most cardinalities, all the examples have the property that their third shadow is already the initial colex segment.

\begin{theorem}\label{thm:depth3} Let $k<n$ be positive integers and $N=\binom{n}{k}$. Let $U\subset [N]$ be the set of integers for which all extremal families $S\subset \binom{[n]}{k}$ with cardinality $\binom{\a}{k}\stackrel{b}{=}m\in U$ satisfy that $\Delta^{4}(S)\cong I_{n,k-4}(\binom{\a}{k-4})$. Then
	$$
	\lim_{n\to\infty} \frac{|U|}{N}=1.
	$$
\end{theorem}

Even more, given an $m$ and a $t$, we can determine whether there exists and extremal family $S$ of size $m$ for which $\Delta^t(S)$ is not an initial segment in the colex order, but $\Delta^{t+1}(S)$ is, and if that is the case, construct one, in time $O(n\; \text{poly}(k))$.

\begin{theorem}[Algorithmic construction]\label{thm:alg}
	Let $k,n$, $2\leq t<k-1$ and $m$ be given. The existence of an extremal family $S$ with cardinality $m$,  support in $[1,n]$, such that $\Delta^t(S)$ is not an initial segment of the colex order, and $\Delta^{t+1}(S)$ is the initial segment of the colex order, can be decided in time $O(n k)$. If such a family exists, an instance of $S$ can be constructed with the same time complexity. 	
\end{theorem}

The next result, Theorem~\ref{thm:solid_part}, shows that, in a specific sense, every family is an \emph{induced} subfamily of an extremal family; this provides a tool to build extremal families which are not initial segments of the colex order. Therefore, in another sense, we can find extremal families that are as far away as the colex as desired (at least in a localized area of the support).  
For two families $A,B$ of subsets we write 
\begin{equation}\label{eq:vee_not}
A\vee B=\{a\cup b: a\in A, b\in B\}.
\end{equation}

\begin{theorem}\label{thm:solid_part} For each non-empty family  $S\subseteq \binom{[n]}{k}$ there is a positive integer $r_0=r_0(S)$ such that, for every $r\ge r_0$, the family $S'$ of $k$--subsets of $[n+r]$ defined as
\begin{equation} \label{eq:family_ext}
	S'=S'(r)=\bigcup_{i=1}^{k-1} \left(\Delta^{i}(S)\vee \binom{[n+1,n+r]}{i}\right),
\end{equation}
	is extremal. 
\end{theorem}

If we consider the families to be labelled, then two isomorphic families are considered to be different. We can prove the following relative result, which claims that \cite[Example~2.4]{furgri86} or \cite[Proposition~2.5-2.4]{furgri86} restricted to the initial segment in the colex order, give asymptotically all the examples:
\begin{theorem} \label{thm:count_1}
	Consider the families as labelled.\footnote{The hypergraphs are labelled hypergraphs, so there exists only one colex segment, for instance.}
	Let $\mathcal{E}$ be the set of extremal families in $\binom{[n]}{k}$, and let $\mathcal{E}_0$ be the set of extremal families in $\binom{[n]}{k}$ such that, for each $S\in \mathcal{E}_0$,  $\Delta(E)$ is the initial segment in the colex order. Then, for each fixed $k$,
	\[
	\lim_{n\to\infty} \frac{|\mathcal{E}|}{|\mathcal{E}_0|}=1
	\]
\end{theorem}
With Theorem~\ref{thm:count_1}, equation~\ref{eq:count_rem_colex} gives an asymptotical estimate on the number of extremal labelled families.
Further, Theorem~\ref{thm:count_1} can be refined in some instances.
\begin{theorem} \label{thm:count_3}
	Consider the families as labelled. Fix $k>0$ and let $t_0,\ldots,t_{k-1}$ be a strictly increasing sequence of positive integers.
	Let $\mathcal{E}(t_0,\ldots,t_{k-1})$ be the set of extremal families in $\binom{[n]}{k}$ that have $k$-binomial decomposition:
	\[
	\binom{n-t_0-1}{k}+\cdots+\binom{n-t_{k-2}-(k-1)}{2}+\binom{n-t_{k-1}-(k-1)}{1}
	\]
	Let $\mathcal{E}_0(t_0,\ldots,t_{k-1})$ be the set of extremal families in $\mathcal{E}(t_0,\ldots,t_{k-1})$ that are obtained from the initial segment in the colex order by removing some edges of size $k$. Then
	\[
	\lim_{n\to\infty} \frac{|\mathcal{E}(t_0,\ldots,t_{k-1})|}{|\mathcal{E}_0(t_0,\ldots,t_{k-1})|}=1
	\]
\end{theorem}
In Section~\ref{sec:count} we give a generalization of Theorem~\ref{thm:count_3} in the form of Theorem~\ref{thm:count_2}.

We just mention that in Section~\ref{sec:adding_and_substracting} we obtain some results regarding when the extremality is preserved when adding or substracting sets to the family.  These results allows for a discussion on maximal chains (in the poset of $k$-set families ordered by inclusion) of extremal families in Section~\ref{sec:maximal_chains}. However, we do not make the statements explicit in this introduction since they require notation that is introduced later in this paper.


\subsection{Organization of the paper}

In Section \ref{sec:shadow} we give the proof of Theorem \ref{thm:charac}.
Section~\ref{sec:sketch} give a sketch of the proof of Theorem~\ref{thm:main} using the ideas of translation invariance of binomial expressions from \cite{sv21}.
 In Section \ref{sec:hypergraph} we introduce a representation of a family of $k$--sets by an hypergraph (which is equivalent to the Stanley-Reisner complex) and discuss the first step towards the proof of Theorem \ref{thm:hypotenusal}. Section \ref{sec:hypotenusal} describes the hypotenusal process and completes the proof of Theorem \ref{thm:hypotenusal}. Theorems \ref{thm:depth} and \ref{thm:depth3} are consequences of Theorem \ref{thm:hypotenusal} and their proofs are given in Section \ref{sec:proof_prop_trans}, which also contains the algorithmic construction of extremal families with the proof of Theorem \ref{thm:alg}. Theorem \ref{thm:solid_part} is proved in Section \ref{sec:proof_prop_trans}.
With the hypothenusal process we are also able to show an alternative proof to Theorem~\ref{thm:card} in Section~\ref{sec:alt_to_mtfg}.

%
%

%
%
%
%
%

\section{Shadow sequences and proof of Theorem~\ref{thm:charac}}\label{sec:shadow}

In contrast with the $k$-binomial decomposition of the cardinality of a family of $k$-sets, the coefficients of the shadow $k$-binomial decomposition $\beta_i$ can be negative, and the vector $(\beta_0,\ldots,\beta_{k-1})$ has always $k$ elements. This second point implies that, as \eqref{eq.finding_b} shows, there is a discrepancy in the definition between the last coefficient $\beta_{k-1}$ of the shadow $k$-binomial decomposition and the rest of the $\beta_i$. Indeed, when computing $\beta_i$, the term $\binom{\beta_{i+1}}{0}=1$ should be considered, eventhough the value for $\beta_{i+1}$ is not yet known.\footnote{The fact that the coefficient is not known is not a problem since all the binomial coefficients with binomial denominator being zero always have $1$ as their numerical value, regardless of the binomial numerator.}
This discrepancy can be solved by framing the shadow $k$-binomial decomposition within an infinite sequence which we call the shadow $\infty$-binomial decomposition, in which all the coefficients are treated equally.
The results in Section~\ref{sec:hypergraph} involving the alternate representation of the families using hypergraphs, as well as the relation between the shadow k-binomial decompositions and the hypotenusal processes hinted in Theorem~\ref{thm:hypotenusal}, and that are further developed in Section~\ref{sec:hypotenusal}, helps to cement this approach.
Also, we define:
\begin{definition}[Full $k$-binomial decomposition] \label{def:full_k_bin}
	Let $m>0$ and $k>0$ be positive integers. Then the \emph{full $k$-binomial decomposition} of $m$ is the unique sequence of integers $\balpha=(\alpha_0,\ldots,\alpha_{k-1})$ with $\alpha_0>\alpha_1>\cdots>\alpha_{k-2}\geq \alpha_{k-1}\geq 1$ and such that
	$m=\binom{\balpha}{k}$. We use $\binom{\balpha}{k}\stackrel{fb}{=}m$ to denote that $\balpha$ is the full $k$-binomial decomposition of $m$.
\end{definition}
Observe that if the $k$-binomial decomposition is complete (so it has $k$ terms), then the full $k$-binomial decompositions equals the $k$-binomial decomposition, if it has $<k$ terms, then the last terms are obtained by 
\begin{equation}\label{eq:kbin_to_fullkbin}
\binom{a_t}{k-t}=\binom{a_t-1}{k-t}+\binom{a_t-1-1}{k-t-1}+\cdots+\binom{a_t-1-(k-t-2)}{k-t-(k-t-2)}+\binom{a_t-1-(k-t-2)}{k-t-(k-t-1)}
\end{equation}
Also observe that, if $a_t=k-t$ and $t<k-1$, then some terms of the previous equality are of the type $\binom{i}{i+1}$, with $i\geq 1$ so they have zero as their value, yet $\binom{i}{i+1-1}=1$.

\begin{remark} \label{rmk:k-binom_determines_shadow}
	By \cite[Theorem~2.1]{furgri86}, the cardinality of the extremal family fully determines the cardinalities of their shadows. Furthermore, in an extremal family the shadow $k$-binomial decomposition equals the full $k$-binomial decomposition.
\end{remark}

Our main technical result used in the proof of Theorem~\ref{thm:charac} is the following property of the shadow $k$-binomial decomposition sequence of a family.

\begin{theorem}[Monotonicity of the shadow $k$-binomial decomposition]\label{thm:main}
If $(\beta_0,\ldots,\beta_{k-1})$ is the shadow $k$-binomial decomposition of a non-empty family $S\subseteq \binom{[n]}{k}$, then $\beta_0>\beta_1>\ldots >\beta_{k-2}\ge \beta_{k-1}$.
\end{theorem}

We proceed to show Theorem~\ref{thm:charac} assuming Theorem~\ref{thm:main}, which is proved in Section~\ref{sec:proof_prop_trans}.

\begin{proof}[Proof of Theorem~\ref{thm:charac}]
	First suppose that  all the elements of the shadow $k$-binomial decomposition $\bbeta=(\beta_0,\ldots,\beta_{k-1})$ of $S$ are strictly positive. Then, by their respective unicities and Theorem~\ref{thm:main}, this decomposition coincides with the full $k$-binomial decomposition of $|S|$, which is strictly positive. By Theorem \ref{thm:kk} with $i=1$ and the relation from the full $k$-binomial decomposition and the $k$-binomial decomposition of $|S|$ from \eqref{eq:kbin_to_fullkbin}, and the definition of shadow $k$-binomial decomposition, we have $|\Delta(S)|=\binom{\bbeta}{k-1}$ and $S$ is extremal. 
	
	For the other implication,
	suppose now that
	there are some non-positive elements in the shadow $k$--decomposition of $m$ and we lets show that then $S$ is not extremal.
	Let $s\geq 0$ be the largest nonnegative integer such that 
	\[
	|\Delta^s(S)|=\binom{\beta_0}{k-s}+\cdots+\binom{\beta_{k-s-3}}{3}+\binom{\beta_{k-s-2}}{2}+\binom{x}{1},\text{ where } x=
	\begin{cases}
	\beta_{k-1} \leq 0 &\text{ if }s=0 \\
	\beta_{k-s-1}+1\leq 0& \text{ if }s>0
	\end{cases}
	\]
	with the property that $x\le 0$ and $\beta_{k-s-2}\geq 0$. Let 
	\[
	|\Delta^s(S)|\stackrel{\text{fb}}{=}\binom{\c}{k-s}
	\]
	be the full $(k-s)$--binomial decomposition of $|\Delta^s(S)|$. As $x<c_{k-s-1}$, then
	\begin{equation}
	\left(\binom{c_0}{k-s}+\cdots+\binom{c_{k-s-2}}{2}\right)-\left(\binom{\beta_0}{k-s
	}+\cdots+\binom{\beta_{k-s-3}}{3}+\binom{\beta_{k-s-2}}{2}\right)<0, \label{eq.conc1}
	\end{equation}
	which implies that $(c_0,\ldots,c_{k-s-2})<_{lex} (\beta_0,\ldots,\beta_{k-s-2})$ (as both sequences are decreasing and non-negative), and therefore,
	\begin{equation}
	\left(\binom{c_0}{k-s-1}+\cdots+\binom{c_{k-s-2}}{2-1}\right)-\left(\binom{\beta_0}{k-s-1
	}+\cdots+\binom{\beta_{k-s-3}}{3-1}+\binom{\beta_{k-s-2}}{2-1}\right)<0, \label{eq.conc2}
	\end{equation}
	Then we conclude that
	\begin{align}
	\binom{\c}{k-s-1}&=\binom{c_0}{k-s-1}+\cdots+\binom{c_{k-s-2}}{1}+\binom{c_{k-s-1}}{0}&& \text{definition of $\binom{\c}{k-s-1}$}\nonumber \\
	&<\binom{\beta_0}{k-s-1}+\cdots+\binom{\beta_{k-s-2}}{1}+\binom{x}{0}&& \text{by \eqref{eq.conc2} }\nonumber \\
	&=|\Delta^{s+1}(S)|,&&\text{by definition of $\beta_i$}\nonumber
	\end{align}
	and thus $\Delta^{s}(S)$ is not extremal (both the $k$-binomial decomposition and the full $k$-binomial decomposition give the same expression when the binomial denominator is decreased by $1$). By \cite[Theorem~2.1]{furgri86} (recalled in this work as the ``moreover'' part of Theorem~\ref{thm:kk}), if $\Delta^s(S)$ is not extremal for some $s\ge 0$ then $S$ is  not extremal either. This completes the proof. 
\end{proof}





\section{Strategy to show the monotonicity result: Theorem~\ref{thm:main}} \label{sec:sketch}

An identity $\sum_{i}\alpha_i \binom{x_i}{y_i}=\sum_j \alpha_j'\binom{x'_j}{y'_j}$ of two  sums of binomial coefficients (with $x_i$,$y_i$,$x_i'$,$y_i'$ integers, and $\alpha_i$,$\alpha_i'$ real numbers) is {\it translation invariant} if, for every pair of integers $r,s$ we have
$$
\sum_{i}\alpha_i\binom{x_i+r}{y_i+s}=\sum_j \alpha_{j}' \binom{x'_j+r}{ y'_j+s}.
$$
We use the notation $\doteq$ as in  
$
\sum_{i}\alpha_i\binom{x_i}{y_i}\doteq \sum_j \alpha_j'\binom{x'_j}{ y'_j}
$
to indicate that the identity is translation invariant.
This notion was used in \cite{sv21} in order to show the characterization result of the extremal families of the Kruskal-Katona theorem; in
\cite[Proposition~14]{sv21} it was shown that the only way to obtain translation invariant identities is to transfrom one expression into the other by the use of the \emph{binomial recurrence} $\binom{n}{k}\doteq\binom{n-1}{k}+\binom{n-1}{k-1}$ and by possibly adding/removing identical terms. ($\binom{n}{0}=\binom{n-1}{0}$ yet they are not identical binomial coefficients much like $\frac{10}{5}=\frac{2}{1}$, yet $\frac{10+1}{5+1}\neq \frac{2+1}{1+1}$.)

Our first goal to prove to Theorem~\ref{thm:main} is to find a finite multiset of pairs of positive integers $\{(m_j,d_j)\}_{j\in I}$, for some indexing set $I$, depending on $S$, in such a way that:
\begin{equation}\label{eq:1}
\text{for all $i$,}\qquad |\Delta^i(S)|=\binom{n}{k-i}-\sum_{j\in I} \binom{n-m_j}{k-d_j-i}
\end{equation}
that is, we write the cardinalities of the iterated lower shadows of $S$ in such a way that all the terms $\{(m_j,d_j)\}_{j\in I}$ are fixed, and the only thing it varies are the binomial denominators, which increase when $i$ decreases. Observe that if $d_j$ are large, then the terms $\binom{m_j}{k-d_j-i}$ equal $0$ so they do not contribute to the sum.
Then we find a translation invariant expression for
$|S|=\binom{n}{k}-\sum_{j\in I} \binom{n-m_j}{k-d_j}$
of the following type:
	\begin{equation} \label{eq:2}
|S|=\binom{n}{k}-\sum_{j\in I} \binom{n-m_j}{k-d_j}\doteq \binom{\beta_0}{k}+\binom{\beta_1}{k-1}+\binom{\beta_2}{k-2}+\cdots+\binom{\beta_{k-1}-1}{1}+\binom{a}{0}+\sum_{s\leq -1} \sum_{p\in \Z} \alpha_p \binom{p}{s}
\end{equation}
\noindent 
for some integer coefficients $\alpha_p$ and some  integer $a$. Notice that the part $\sum_{s\leq -1} \sum_{p\in \Z} \alpha_p \binom{p}{s}$ has value $0$ and the term $\binom{a}{0}$ has value $1$. (The existence of the term $\binom{a}{0}$ explains the $+1$ in the definition of $\beta_{k-1}$ in \eqref{eq.finding_b}, and the $-1$ in the term $\binom{\cdot}{1}$ in \eqref{eq:2}.)
By \eqref{eq:1} and the fact that the expression in \eqref{eq:2} is translation invariant, we conclude that
\begin{align} \label{eq:2}
|\Delta^i(S)|&=\binom{n}{k-i}-\sum_{j\in I} \binom{n-m_j}{k-d_j-i} \nonumber \\
\doteq& \binom{\beta_0}{k-i}+\binom{\beta_1}{k-1-i}+\binom{\beta_2}{k-2-i}+\ldots+\binom{\beta_{k-1}-1}{1-i}+\binom{a}{0-i}+\sum_{s\leq -1} \sum_{p\in \Z} \alpha_p \binom{p}{s-i} \nonumber 
\end{align}
thus showing that the sequence $(\beta_0,\ldots,\beta_{k-1})$ is the shadow $k$-binomial decomposition of $S$.\footnote{Note that we have
\[
\text{for all $i$},\qquad |\Delta^i(S)|= \binom{\beta_0}{k-i}+\binom{\beta_1}{k-1-i}+\binom{\beta_2}{k-2-i}+\ldots+\binom{\beta_{k-1}}{1-i}
\] but this equality is not translation invariant, as it does not make sense to ask for translation invariance when looking at the cardinality of a set. Further, the right hand side is not translation invariant with $\binom{n}{k-i}-\sum_{j\in I} \binom{n-m_j}{k-d_j-i}$ since the terms $\binom{a}{0-i}+\sum_{s\leq -1} \sum_{p\in \Z} \alpha_p \binom{p}{s-i}$ are missing.}
The arguments will then conclude the properties of the shadow $k$-binomial decomposition (namely its monotonicity) from the properties of the original pairs $\{(m_j,d_j)\}_{j\in I}$.

There are different combinations of $\{(m_j,d_j)\}_{j\in I}$ that give the same shadow $k$-binomial decomposition yet all of them are transltion invariant and/or leave a translation invariant \emph{leftover} $\binom{a}{0}+\sum_{s\leq -1} \sum_{p\in \Z} \alpha_p \binom{p}{s}$; indeed, being translation invariant induces an equivalence relation on the set of binomial expressions \cite[Proposition~14]{sv21}.
The computation of the pairs $\{(m_j,d_j)\}_{j\in I}$ is done in Section~\ref{sec:hypergraph} and Section~\ref{sec:hypotenusal}. Then the argument of the monotonicity of the $\beta_i$, Theorem~\ref{thm:main}, is given in Section~\ref{sec:proof_prop_trans} by the technical propositions Proposition~\ref{prop.translation}~\ref{en:prop_trans_1} and \ref{en:prop_trans_2}, which can be found as Section~\ref{sec:proof_prop_trans}. In summary, the monotonicity of the shadow $k$-binomial decomposition follows from the properties of the pairs $\{(m_j,d_j)\}_{j\in I}$ although in a not so direct way; all the ideas for the arguments leading to the proof of monotonicity are similar to those in \cite{sv21}.

\section{Hypergraph representation of a family of \texorpdfstring{$k$}{k}-sets
}\label{sec:hypergraph}

A family of $k$-sets, $S\subset \binom{[n]}{k}$, can be presented in several equivalent ways. The following is one of them.

\begin{definition}[Hypergraph of a family of $k$-sets]\label{def.hyp_family}
	Given a family of $k$-sets $S\subseteq \binom{[n]}{k}$, $H=H(S)=(V,E\subset \mathcal{P}(V))$, $E=E_1\cup \cdots \cup E_k$ with $E_i$ containing the edges of size $i$, is \emph{the hypergraph of $S$} when
	\begin{enumerate}[label=(\roman*)]
		\item $V=V(H)=[n]$
		\item for each $i\in[1,k]$, $e\in E_i\subset \binom{V(H)}{i}$ is an $i$-edge of $H$ if
		\begin{quote}
		$e\notin \Delta^{k-i}(S)$, \;\;  and 
	\;\;	$e$ is minimal (with respect to containment) with such property.
		\end{quote}
	\end{enumerate}
\end{definition}

\noindent We list below some equivalent formulations, and properties, that are relevant later.
\begin{enumerate}[label=(P\arabic*)]
	\item $H$ has no two edges contained in each other. Thus the hyperedges form an antichain in the poset $(\mathcal{P}([n]),\subset)$. If, in addition to the antichain property, we add the extra property that there is no $t< k$ and set $A\subset [n]$, $|A|=k-t$ with all the edges $\binom{[n]\setminus A}{t}\vee A$ in the hypergraph $H$, then such hypergraph is a hypergraph of a family of $k$-sets.
	\item \label{en:p1} $H$ is the hypergraph whose edges are the minimal elements in the poset $\bigcup_{i=0}^{k-1}\left[\binom{[n]}{k-i} \setminus \Delta^i(S)\right]$ ordered by inclusion. That is, $H$ records the information of $S$ and its lower shadows via its complements in $\binom{[n]}{k-i}$.
	\item \label{en:p2} Point \ref{en:p1} means that $H$ have the following  property: \begin{quote}if $e\in \binom{[n]}{k-j}$ contains an edge of $H$, then $e\notin \Delta^{j}(S)$ (equivalently, $e\in [\Delta^{j}(S)]^c$, or, also that no extension of an edge in $H$ to a $k$-set belong to $S$).\end{quote}
	\item The formulation of \ref{en:p2} implies that Definition~\ref{def.hyp_family} is a reformulation of the Stanley-Reisner ideal of the simplicial complex induced by the family $S$.\footnote{
Given a simplicial complex $\mathcal{S}$ in $[n]$ and a field $\mathbb{F}$, the \emph{Stanley-Reisner ideal} of $\mathcal{S}$, $I_{\mathcal{S}}$, is given by the monomials in $\mathbb{F}[x_{1},\ldots,x_{n}]$ such that $x_{i_1}\cdots x_{i_j}\notin \mathcal{S}$. Then the \emph{Steiner-Reisner ring} is given by $\mathbb{F}[x_{1},\ldots,x_{n}]/I_{\mathcal{S}}$. See \cite[Chapter~II]{stanley2007combinatorics}. In our case, the simplicial complex is given by the family of $k$-sets and all its lower shadows.}
	\item The hypergraph $H$ can have $1$-edges. These $1$-edges are different from isolated vertices of $H$ and are precisely those elements in $[n]$ not belonging to the support of $S$.
	\item \label{en:p6} Recall that two families $S_1$ and $S_2$ of $k$-sets of $[n]$ are said to be isomorphic if there exists a permutation $\sigma: [n] \to [n]$ mapping $k$-sets of $S_1$ to $k$-sets of $S_2$ and no $k$-sets of $S_1$ to no $k$-sets of $S_2$.
	Then \begin{quote} Two families of $k$-sets are isomorphic if and only if their respective hypergraphs are isomorphic.
		\end{quote}
	\item Using \ref{en:p2}, if $S\subseteq \binom{[n]}{k}$ has hypergraph $H(S)=([n],\{e_1,\ldots,e_s\})$, then, using the $\vee$ notation from \eqref{eq:vee_not},
	\begin{equation}\label{eq.more_gen}
	S^{c}= \bigcup_{i=1}^s \left[e_i\vee \binom{[n]\setminus e_i}{k-|e_i|}\right], \;\;\text{or, more generally}\;\; 	[\Delta^{j}(S)]^c= \bigcup_{i=1}^s\left[ e_i\vee \binom{[n]\setminus e_i}{k-j-|e_i|}\right]
	\end{equation}
	where $e_i\vee \binom{[n]\setminus e_i}{k-j-|e_i|}$ is an empty set if $|e_i|>k-j$.
	
	\item  \eqref{eq.more_gen} can be used in conjunction with Definition~\ref{def.hyp_family} to conclude the following:
	\begin{lemma}\label{lem:hyp_shadow}
		Let $H=([n],E_1\sqcup E_2\sqcup \cdots \sqcup E_k)$ be the hypergraph of $S\subseteq \binom{[n]}{k}$, with $E_i$ being the hyperedges of size $i$ in $H$. Then $H^{(k-i)}:=([n],E_1\sqcup E_2\sqcup \cdots \sqcup E_{k-i})$ is the hypergraph of $\Delta^i(S)\subseteq \binom{[n]}{k-i}$.
	\end{lemma}
\end{enumerate}

\subsection{Using the hypergraph to achive the strategy given in Secion~\ref{sec:sketch}}

By considering the new contributions of each edge $e_i$ of the hypergraph $H(S)=([n],\{e_1,\ldots,e_s\})$ to $S^c$ in the order of $\{e_1,\ldots,e_n\}$ given by the subindices, we can write $S^c$ as a disjount union: 
\begin{equation} \label{eq:adding_sets}
S^{c}= \bigsqcup_{i=1}^s \left[e_i\vee \binom{[n]\setminus e_i}{k-|e_i|}\setminus \left[\bigcup_{j=1}^{i-1} \left[e_j\vee \binom{[n]\setminus e_j}{k-|e_j|}\right]\right]\right]
\end{equation} 
Now the goal is to write each of the terms in \eqref{eq:adding_sets}, namely
\begin{equation} \label{eq:adding_edges_1}
\left[e_i\vee \binom{[n]\setminus e_i}{k-|e_i|}\setminus \left[\bigcup_{j=1}^{i-1} \left[e_j\vee \binom{[n]\setminus e_j}{k-|e_j|}\right]\right]\right]
\end{equation}
 as a disjoint union of extensions of edges of $H$. For that we combine 
  Lemma~\ref{lem:d} and Proposition~\ref{prop:tree_prop} to turn \eqref{eq:adding_edges_1} first into \eqref{eq:conseq_def_bi} and then into \eqref{eq:tree_count_1} which is counted by \eqref{eq:tree_count_2}; thus the disjoint union is compatible with the strategy oulined in  Section~\ref{sec:sketch} to find the pairs $\{(m_j,d_j)\}_{j\in I}$.

\begin{lemma} \label{lem:d}
	Given $S\subseteq\binom{[n]}{k}$ a family of $k$-sets with hypergraph $H=([n],\{e_1,\ldots,e_{n}\})$. 
	Let $B_j$ be the set of minimal dominating sets of $\{e_i\}_{i<j}$ not intersecting $e_j$:
	\begin{equation} \label{eq:def_bi}
	B_j= \{s \:|\: s \subseteq [n]\setminus e_j, \forall i<j\; [s \cap e_i\neq \emptyset], s\text{ minimal by inclusion with the previous two properties}\}
	\end{equation}
	then,
	\begin{equation}\label{eq:conseq_def_bi}
	\left(\left[e_j\vee \binom{[n]\setminus e_j}{k-|e_j|}\right]\setminus \left[\bigcup_{i=1}^{j-1} \left[e_i\vee \binom{[n]\setminus e_i}{k-|e_i|}\right]\right]\right)= \bigcup_{s\in B_j} \left[e_j \vee \binom{[n]\setminus [e_j\sqcup s]}{k-|e_j|}\right]
	\end{equation}
\end{lemma}
That is,  we capture the new extensions of $e_j$, with respect to those extensions of $\{e_i\}_{i<j}$, by avoiding extending in $s$, which is a minimal set of vertices that dominates $\{e_i\}_{i<j}$. 
\begin{proof}[Proof of Lemma~\ref{lem:d}]
	If an extension of $e_j$ fully contains $e_i$ with $i<j$, then the $k$-set has already been considered as an extension of $e_i$. In particular, each extension of $e_j$ in the set given by the left hand side of \eqref{eq:conseq_def_bi} should avoid every single edge in $\{e_i\}_{i<j}$. In particular, it should avoid a minimal set (minimal by containment) that intersects every single edge in $\{e_i\}_{i<j}$, and thus should avoid a set in $B_j$. This shows the claimed statement.
\end{proof}

\subsection{Codifying the sets of extensions: a rooted tree for every edge}
\label{sec:tree_coded}
Given
\begin{itemize}
	\item a family $S\subseteq \binom{[n]}{k}$ with hypergraph $H=([n],\{e_1,\ldots,e_n\})$ with some order of the edges (here the increasing order is indicated using the subindices).
	\item an index $j\in[n]$ and the $j$-th edge $e_j\in \{e_1,\ldots,e_n\}$ with associated set $B_j$; the sets $B_j$ are defined in \eqref{eq:def_bi},
	\item  and some order on the sets $B_j=\{s_1,\ldots,s_m\}$ (the sets $B_j$ are defined in \eqref{eq:def_bi}).
\end{itemize} 
the procedure below gives a rooted tree $T_{e_{j}}$ (examples of these trees given in Figure~\ref{fig.1}).
\begin{rmk}
The following construction of the trees below depends both on the order of the edges of the hypergraph and on the oder of the sets of each $B_j$.

Regarding the orders on the edges is more convenient for the arguments in the following sections to consider what we call a \emph{comfortable ordering} on the edges; that is, and ordering in which $e_i\leq e_j \iff i<j$.

Furthermore, some particular ordering on the sets of $B_j$ helps in the arguments in Section~\ref{sec:ext_fam_prescribed_depth}.

However, let us reiterate that the particular order chosen, in both cases, is not relevant, in the sense that any ordering leads to the same conclusions.\qed
\end{rmk}

\paragraph{Procedure to construct $T_{e_j}$.}
\begin{enumerate}[label=(S\arabic*)]
	\item Let $v_1=(1,s_1)$ be the root vertex. 
	
	Initializate the list of vertices to be processed: $L=\{v_1=(1,s_1)\}$.
	$L$ contains only internal vertices (the root being considered an internal vertex, as opposed to a leaf vertex, the internal vertices have an odd number in the first coordinate, and it is related with the number of vertices that have the same second coordinate).
	\item \label{en:step2} If $L$ is empty, the process finishes and $T=T_{e_j}$ is produced.
	\item If $L$ is non-empty, select the first vertex of $L$ denoted as $u=(2t-1,s_i)$, for some $t>0$.
	\item Let $P$ be the path from $u=(2t-1,s_i)$ to the root (both vertices included in the path). 
	
	Let $\ell_e(P)$ denote the union of the labels of the edges in $P$.
	
	Let $\ell_v(P\setminus u)$ denote the union of the labels of the vertices in $P$ (with the exception of that of $u$).
	
	The label of $u=(2t-1,s_i)$, denoted as $\ell_v(u)$, is then set to: \begin{equation}\label{eq:lab_vert}
	\ell_v(u)\leftarrow s_i\setminus \ell_v(P\setminus u)\end{equation} (In particular, the label of the root is $s_1$)
	\item For each of the possible $2^{|\ell_v(u)|}$ subsets $s\in\mathcal{P}(\ell_v(u))$, do the following:
	\begin{itemize}
		\item If [$s$ is the empty set] then:
		
		Add the leaf vertex $(2t,s_i)$ adjacent to $v=(2t-1,s_i)$ using 
a directed edge $f=((2t,s_i),(2t-1,s_i))$ pointing towards $(2t-1,s_i)$. 

Label the edge $f=((2t,s_i),(2t-1,s_i))$ just added with the empty set:
		\[
		\ell_e(f)\leftarrow \emptyset
		\]
		
Label the leaf vertex $(2t,s_i)$ just added with the double label $(\ell_v(P),\ell_e(P))$
		\begin{equation}\label{eq:lab_leaf}
		\ell_v(\;(2t,s_i)\;)\leftarrow (\ell_v(P),\ell_e(P)),
		\end{equation} where $P$ is the path from $(2t-1,s_i)$ to the root.
		\item If [$s$ is not the empty set] and [there exists 
		a set $s_0$ in $B_j$,   
		with $s_0\cap [\ell_e(P)\sqcup s]=\emptyset$],
		then:
		
		Let $r$ be the minimal index in the ordering of the sets $B_j$ with the property that $s_r\cap [\ell_e(P)\sqcup s]=\emptyset$.
		
		Add the internal vertex $(2t'+1,s_r)$, where $t'$ is the number of internal vertices with $s_r$ in the second coordinate that are already in the tree.

	Add an edge $((2t'+1,s_r),(2t-1,s_i))$ (from $(2t'+1,s_r)$ to its parent vertex $(2t-1,s_i)$).
	
	Label the edge $((2t'+1,s_r),(2t-1,s_i))$ with $s$
	\begin{equation}\label{eq:lab_edge}
	\ell_e(\;((2t'+1,s_r),(2t-1,s_i))\;)\leftarrow s
	\end{equation}
	
	Add the vertex $(2t'+1,s_r)$ to $L$, the list of internal vertices to be processed.
	
		\item If [$s$ is not the empty set] and [no such set $s_0$ exists in $B_j$], then do nothing and move to the next set $s$.
	\end{itemize}
	\item Remove $v=(2t-1,s_i)$ from $L$ and move to step \ref{en:step2}.
\end{enumerate}

Let us now show some properties of the trees $T_{e_j}$.

\begin{proposition}\label{prop:tree_prop}
With the notation and set up from the beginning of Section~\ref{sec:tree_coded}, for each edge $e_j$ the
 procedure above finishes, produces $T=T_{e_j}$ and 
	\begin{enumerate}[label=(\roman*)]
		\item \label{en:tree_prop_1}  	
		\begin{itemize}	
			\item $T$ is a rooted tree with directed and labelled edges, and labelled vertices. All labels are subsets of $[n]$, the edges are directed towards the root.
		\item The vertex set $V$ is a subset of $\mathbb{N}\times B_j$; with its projection to the second coordinate being at least $2$-to-$1$ and surjective.
		
		$T$ contains as many internal vertices as leaf vertices (the root vertex is considered an internal vertex); the internal vertices having odd first coordinate, and the leaf vertices having even first coordinate.
		
		Each leaf of $T$ is paired with its internal parent vertex; they both have the same set in $B_j$ as their second coordinate.
		\item  In an internal-vertex-to-root path, the second coordinates of their vertices are different sets of $B_j$.
		\item The label of the internal vertex $(2t-1,s_i)$ is a subset of $s_i$.
		
		 The leaf vertices have a pair of sets as label, the first being the union of the labels of the edges from its-parent-to-root path, and the second being the union of the labels of the vertices from its-parent-to-root path.
		\item The label of each edge is a subset (perhaps the $\emptyset$) of the label of its parent vertex (the vertex adjacent to the edge that is closer to the root).
		\end{itemize}
		\item \label{en:tree_prop_2} Let $\mathcal{L}$ denote the set of leaf vertices of $T$ (that are not the root vertex). For $f=(2t,s_r)\in \mathcal{L}$, let $P_f$ denote the unique path from the parent vertex of $f$ to the root. Let $\ell_e(P_f)$ (resp. $\ell_v(P_f)$) denote the union of the labels of the edges (resp. vertices) of $P_f$. (Equivalently, $\ell_e(P_f)$ is the first element in the label of $f$, and $\ell_v(P_f)$ is the second.) Then, $\ell_e(P_f)\subseteq \ell_v(P_f)$ and
		\begin{equation} \label{eq:tree_count_1}
		\left(\left[e_j\vee \binom{[n]\setminus e_j}{k-|e_j|}\right]\setminus \left[\bigcup_{i=1}^{j-1} \left[e_i\vee \binom{[n]\setminus e_i}{k-|e_i|}\right]\right]\right)=\bigsqcup_{f\in \mathcal{L}}\left[ e_j \vee  \binom{[n]\setminus [e_j\sqcup \ell_v(P_f)]}{k-|e_j|-|\ell_e(P_f)|}\right]
		\end{equation}
		\item \label{en:tree_prop_3} for the leaf $f=(2t,s_r)\in \mathcal{L}$,
		\begin{equation}
		\label{eq:tree_count_2}
		\left|\left[ e_j \vee  \binom{[n]\setminus [e_j\sqcup \ell_v(P_f)]}{k-|e_j|-|\ell_e(P_f)|}\right]\right|=\binom{n-|e_j|-|\ell_v(P_f)|}{k-|e_j|-|\ell_e(P_f)|}
		\end{equation}
	\end{enumerate}
\end{proposition}

\begin{proof}[Proof of Proposition~\ref{prop:tree_prop}]
For \ref{en:tree_prop_1}, we should proof that the process finishes and that every set in $B_j$ has an internal vertex in $T$ related with it (the other statements follow by construction). Indeed, by the construction, the tree clearly has a root, its edges and vertices are labelled, and has the same number of internal vertices as leafs (the empty set is always a subset).
	
Let us argue that the process finishes. Indeed, each leaf to root path can have, at most, $n$ vertices as the labels along the edges of the path are pairwise disjoint and they are all non-empty (if they are edges between internal vertices) subsets of $[n]$.

Let us argue that every set in $B_j$ is related to at least one internal vertex.
If the set is $s_1$, then it is clear as the root is related to it. Let $s_i$ be the set we are interested in. Consider $s'=s_1\setminus s_i\neq \emptyset$, then $s_i$ is a candidate to be the internal vertex at the end of the edge labeled $s'$ hanging from $s_1$. Assume that there is another candidate, smaller in the ordering, say $s_2$ which has label $s_2\setminus s_1\neq \emptyset$. At this point $s_i$ is a candidate for
the set $[s_2\setminus s_1]\setminus s_i$ which is different than the $\emptyset$ (indeed, if it would be the empty set, then $s_2\subset s_i$, which does not hold by the minimality on the sets of $B_j$). This argument would then be repeated; in general, a copy of $s_i$ can be found along the path with edges labelled $s_1\setminus s_i$, $[s_2\setminus s_1]\setminus s_i$, $[s_3\setminus [s_1\cup s_2]]\setminus s_i$, $\ldots$ and where $s_2,s_3,\ldots$ are the subsets of $B_j$ with index smaller than $i$ that are also candidates for each of the previous labels on the edges.
The argument finishes as each set $s_i\in B_j$ and the sets cannot be repeated along a path.
	
The part \ref{en:tree_prop_2} and \ref{en:tree_prop_3} follows by understanding which sets from \eqref{eq:tree_count_1} are related to a given leaf. We give a recursive argument on the depth of the tree. First we focuss on the extensions of $e_j$ that avoid the elements in $s_1\in B_j$ (and thus those are extensions of $e_j$ that are not extensions of previous edges $e_i$, $i<j$).
The edge with the label of the empty set give precisely the extensions of $e_j$ that avoid $s_1$.
Then if the edge $e$ attached to the root vertex is labelled $s$, the tree hanging from $e$ gives the extensions of $e_j$ that also contain $s$ and avoids $s_1\setminus s$, and also are not extensions of previous edges $e_i$, $i<j$ (the last condition is the reason that the set $s_r\in B_j$ related to the other end of $e$ is found); these are precisely the extensions of $e_j$ that we want to count and that are not considered when choosing the empty set as the label of an edge. 
(From a leaf-to-root path, we are considering the extensions of $e_j$ that avoids the elements in the labels of the vertices, while contains also the elements in the labels of the edges.)
This process of going deeper within the tree finishes and we end up finding all possible sets in $B_j$  (perhaps several times) due to their minimality. This tree-like structure  codifies a partition of the extensions \eqref{eq:conseq_def_bi} into sets of the type:
\[
e_j\vee \left\{\text{some fixed set of elements $F_i$}\right\} \vee \left\{\text{some choice from }\binom{n\setminus [e_j \sqcup F_i \sqcup Q_i]}{k-|e_j|-|F_i|}\right\}
\]
where the set $Q_i$ are some elements in $n$ that should be avoided to not extend also edges $e_t$, $t<j$, but have not been picked by $F_i$; the elements in $Q_i\cup F_i$ are the labels of the vertices in the leaf to root path.
(See Figure~\ref{fig.1} in Example~\ref{ex.1} for an example). 
 By Lemma~\ref{lem:d}, and the fact that these sets are pairwise disjoint, claims \ref{en:tree_prop_2} and \ref{en:tree_prop_3} follow.
\end{proof}


\begin{example}\label{ex.1}
	Consider the family of $4$-sets with $H(S)=([10],E_2\cup E_3)$ with
	\[E_2=\{(6,7),(7,8),(1,2)\}=\{e_1,e_2,e_3\}, \:
	E_3=\{(2,3,4),(5,9,10)\}=\{e_4,e_5\}\]
	The trees are depicted in Figure~\ref{fig.1}, together with the sizes of disjoint sets of edges counted. The families are as follows:
$B_1=\{\emptyset \}$, $B_2=\{(6)\}$, $B_3=\{(7),(6,8)\}$, $B_4=\{(1,7),(1,6,8)\}$,
 $B_5=\allowbreak\{(2,7),\allowbreak(1,3,7),\allowbreak(2,6,8),\allowbreak(1,3,6,8),\allowbreak(1,4,7),\allowbreak(1,4,6,8)\}$. 

	\begin{center}
				\begin{figure}[ht]
			\centering
			\fbox{
				\begin{tikzpicture}[scale=0.64,decoration={
					markings,
					mark=at position 0.5 with {\arrow{>}}}]
				\tikzstyle{vertex}=[circle,fill=black!100,minimum size=10pt,inner sep=0pt]
				
				
				
				
				%
				%
				
				\node[label=90:{$T_{(6,7)}$},label=270:{$\binom{10-2}{k-2}$}] (te1) at (-12,2) {$\emptyset$} ;
				
				\node[label=90:{$T_{(7,8)}$}] (te2) at (-10,2) {$6$} 
				child {node[label=270:{$\binom{10-2-1}{k-2}$}] {$\emptyset$} } ;
				
				\node[label=90:{$T_{(1,2)}$}] (te2) at (-7,2) {$7$} 
				child {node[label=270:{$\binom{10-2-1}{k-2}$}] {$\emptyset$} } 	
				child {	node {$6,8$} 
					child {node[label=270:{$\binom{10-2-3}{k-3}$}] {$\emptyset$} 
					}
					edge from parent node[right,red] {$7$}
				};
				
				\node[label=90:{$T_{(2,3,4)}$}] (te2) at (-3.3,2) {$7,1$}
				child {node[label=270:{$\binom{10-3-2}{k-3}$}]  {$\emptyset$}   }
				child { node {$\textcolor{blue}{1},\;6,8$} 
					child {node[label=270:{$\binom{10-3-4}{k-4}$}] {$\emptyset$}}
					edge from parent node[right,red] {$7$}
				};

				\node[label=90:{$T_{(5,9,10)}$}] (te2) at (5,2) {$7,2$}
				[level 1/.style={sibling distance=3.2cm},level 2/.style={sibling distance=1.9cm}]
				child {node[label=270:{$\binom{10-3-2}{k-3}$}] {$\emptyset$}} 
				child {node[] {$\textcolor{blue}{7},\;1,3$} 
					child {node[label=270:{$\binom{10-3-4}{k-4}$}] {$\emptyset$} }
					child {node {$\textcolor{blue}{7,1},\;4$} 	child {node[label=270:{$\binom{10-3-5}{k-5}$}] {$\emptyset$} }
						edge from parent node[right,red] {$3$}}
					edge from parent node[left,red] {$2$}}
				child {node {$\textcolor{blue}{2},\;6,8$}
					child {node[label=270:{$\binom{10-3-4}{k-4}$}] {$\emptyset$} }
					edge from parent node[left,red] {$7$}}
				child {node {$1,3,6,8$}
					child {node[label=270:{$\binom{10-3-6}{k-4}$}] {$\emptyset$} }
					child {node {$\textcolor{blue}{1,6,8},\;4$} 	child {node[label=270:{$\binom{10-3-7}{k-5}$}] {$\emptyset$} }
						edge from parent node[right,red] {$3$}} edge from parent node[right,red] {$2,7$}};
				
				\end{tikzpicture}
			}
			\caption{
				Example~\ref{ex.1}. The edges are labelled with red. The edges with no label have the label $\emptyset$. The binomial numbers correspond to the number of extensions given by their root-to-leaf path, and the $k$ indicates the size of the extended edges. The path root-to-leaf in $T_{(5,9,10)}$ given by $(7,2),(7,1,3),(7,1,4),\emptyset$ represents the extensions of $(5,9,10)$ containing also the $2$ and the $3$, but not containing $7$, $1$ or $4$. Note that the  labels on the leaf vertices are the cardinals of the partition sets. In this example, each set in $B_i$ has two vertices related to it, one of them being a leaf vertex, yet in general there may be more.}\label{fig.1}
		\end{figure}
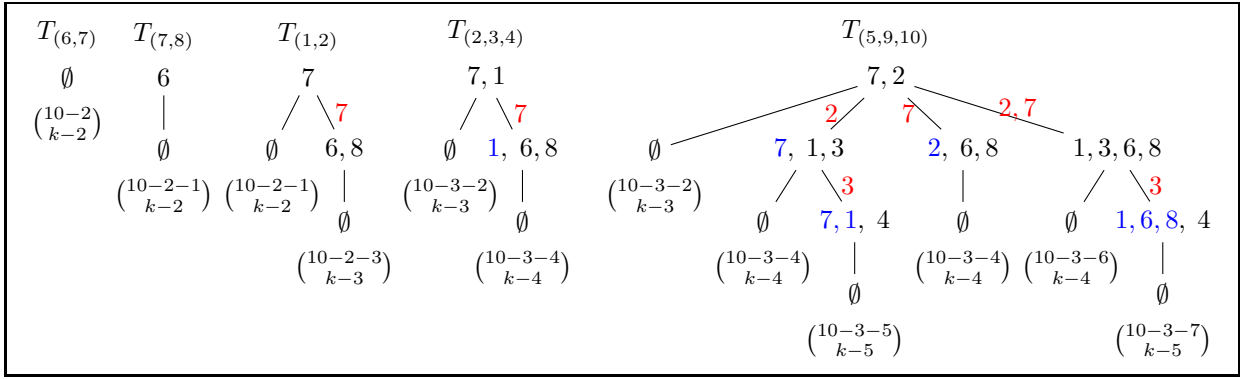
	\end{center}
\end{example}


We use Lemma~\ref{lem:hyp_shadow}, Lemma~\ref{lem:d}, and Proposition~\ref{prop:tree_prop}-\ref{en:tree_prop_2}-\ref{en:tree_prop_3} to conclude that the pairs $\{m_j,d_j\}_{j\in I}$ we were looking for in \eqref{eq:1} are precisely those from Proposition~\ref{prop:tree_prop}-\ref{en:tree_prop_3}, the reason being that, when finding $[\Delta^i(S)]^c$, we have $i$ less choices for the extensions, and thus
	\[\left|\left[ e_j \vee  \binom{[n]\setminus [e_j\sqcup \ell_v(P_f)]}{k-|e_j|-|\ell_e(P_f)|}\right]\right|=\binom{n-|e_j|-|\ell_v(P_f)|}{k-|e_j|-|\ell_e(P_f)|} \qquad \text{ turns to } \qquad \binom{n-|e_j|-|\ell_v(P_f)|}{(k-i)-|e_j|-|\ell_e(P_f)|}
\]
assuming that $|e_j|\leq k-i$ (for otherwise $e_j$ is not included in the hypergraph). Note that, even if we would consider $e_j$, whenever $|e_j|> k-i$, then $k-i-|e_j|-|\ell_e(P_f)|<0$, and thus it would not numerically affect the count towards $|\Delta^i(S)|$. Also observe that the tree may have binomial coefficients associated to leafs (or to leaf-to-root paths) that are numerically $0$ for some values of $i$, but that they should be considered as when $i$ increase, they contribute.

\subsection{Hypergraph of the initial segment in the colex order} \label{sec:hypergraph_of_colex}


Let us describe the hypergraph of $I_{n,k}(m)$ (the initial segment in the colexicographical order) of size $m=\binom{a_0}{k}+\ldots+\binom{a_t}{k-t}$ in $\binom{[n]}{k}$.
$I_{n,k}(m)$ is the family of $k$-sets:
\begin{align}
\binom{[a_0]}{k} \;\sqcup\; \{a_0+1\}\vee \binom{[a_1]}{k-1}\;\sqcup\; \cdots \sqcup \; \{a_{t-1}+1,a_{t-2}+1,\ldots,a_0+1\}\vee \binom{[a_t]}{k-t} \label{eq:colex}
\end{align}
Therefore, the vertex set is $[n]$ and the edges of the hypergraph are:
\begin{equation}\label{eq:edges_in_colex}
\begin{cases}
E_1=\{(n),(n-1),\ldots,(a_0+2)\}=\{e_1,\ldots,e_{n_1}\}\\
E_2=\{(a_0,a_0+1),(a_0-1,a_0+1),\ldots,(a_1+2,a_0+1)\}=\{e_{n_1+1},\ldots,e_{n_1+n_2}\}\\
E_3=\{(a_1,a_1+1,a_0+1),\ldots,(a_2+2,a_1+1,a_0+1)\}=\{e_{n_1+n_2+1},\ldots,e_{n_1+n_2+n_3}\}\\
E_4=\{(a_2,a_2+1,a_1+1,a_0+1),\ldots,(a_3+2,a_2+1,a_1+1,a_0+1)\}=\{e_{n_1+n_2+n_3+1},\ldots,e_{n_1+n_2+n_3+n_4}\}\\
\ldots\\
E_t=\{(a_{t-2},a_{t-2}+1,\ldots,a_0+1),\ldots,(a_{t-1}+2,a_{t-2}+1,\ldots,a_0+1)\}=\allowbreak\{e_{n_1+\ldots+n_{t-1}+1},\ldots,\allowbreak e_{n_1+\ldots+n_t}\}
\end{cases}
\end{equation}
We use the notation $H_C(n_1,\ldots,n_t)$ to denote the hypergraph of the initial segment in the colex order that has $n_i$ edges of size $i$. See $H_C(n_1,\ldots,n_t)$ in Figure~\ref{fig.2} for an example of the shape of the hypergraph; we notice that the hypergraph looks like a \emph{tree}. There are two types of vertices; those always belong exclusively to a unique edge are renamed $u_i$, and those that may belong to more edges are renamed $v_i$. Therefore, we have
{\footnotesize \begin{equation}\label{eq:vertices_of_hyp_colex}
\begin{cases}
\{n,n-1,\ldots,a_0+2 \}\; = \;\{ u_{1},\ldots,u_{n_1}\} \\ 
\{a_0,a_0-1,\ldots,a_1+2\}\;=\;\{u_{n_1+1},\ldots, u_{n_1+n_2}\} \\
\{a_1,a_1-1,\ldots,a_2+2\}\;=\;\{u_{n_1+n_2+1},\ldots, u_{n_1+n_2+n_3}\} \\
\ldots \\
\{a_{t-2},\ldots,a_{t-1}+2\}\;=\;\{u_{n_1+\cdots+n_{t-1}+1},\ldots, u_{n_1+\ldots+n_t}\} \\
\end{cases}
\text{and}\;\;\;
\{a_0+1,a_1+1,\ldots,a_{t-2}+1\}\;=\{v_1,v_2,\ldots,v_{t-1}\}\\
\end{equation}}

Note that all the edges of size $i\geq 2$ contain the vertices $\{a_{0}+1,\ldots,a_{i-2}+1\}$ (which have been renamed to $\{v_1,\ldots,v_{i-1}\}$).
If we order the edges $H_C(n_1,\ldots,n_t)$ by the subindex as in \eqref{eq:edges_in_colex}, the tree $T_{e_i}$ given in Section~\ref{sec:tree_coded} and associated to the edge $e_i$ has exactly one leaf-to-root path as $B_i=\{\;\{u_1,\ldots,u_{i-1}\}\;\}$,
 and it is associated to the binomial coefficient $\binom{n-(i-1)-|e_i|}{k-|e_i|}$. Therefore, it can be easily shown that:
\begin{equation} \label{eq:4}
|S|=\binom{n}{k}-\sum_{1\leq i\leq |E(H)|} \binom{n-(i-1)-|e_i|}{k-|e_i|}\doteq \binom{a_0}{k}+\ldots+\binom{a_t}{k-t}
\end{equation}
which, when extended to its full $k$-binomial decomposition if necessary, gives its shadow $k$-binomial decomposition. 

\begin{lemma} \label{lem:colex_B_i}
	Let $S$ be a family of sets, and $H$ its hypergraph with edges $e_1,\ldots,e_m$ ordered comfortably (so $|e_i|\leq|e_j| \iff i<j$).
	$S$ is the initial segment in the colexicographical order if and only if, for every edge $e_i\in H$, $|B_i|=1=\{s_i\}$ and $|s_i|=i-1$. ($B_i$ are defined by \eqref{eq:def_bi}.)
\end{lemma}

\begin{proof}[Proof of Lemma~\ref{lem:colex_B_i}]
	From left to right holds as the hypergraph of the initial segment has such properties. 
	
	From right to left, we show the contrapositive. For that, we give some more detail on the situation on what happens when the hypergraph stop being a hypergraph of an initial segment in the colex order when we add the edges one by one; this case analysis is useful later on. 
	The set of edges $\{e_1,\ldots,e_m\}$ ordered by its index, is said to be in a \emph{supercomfortable ordering} if
	\begin{itemize}
		\item $|e_i|\leq |e_j|$ if and only if $i\leq j$, (comfortable ordering) and
		\item the index $t\geq 1$ such that, 
		\begin{quote}for each $j$, $1\leq j\leq t$, $\{e_1,\ldots,e_j\}$ are the edges of a colex,\footnote{The hypergraph induced by the edges is isomorphic to a hypergraph of an initial segment in the colex order.} and either $e_{t+1}$ does not exists, or $\{e_1,\ldots,e_{t+1}\}$ is not the hypergraph of a colex.
		\end{quote} is maximum (when all the possible orderings with the first condition are considered).
	\end{itemize}
 Let $\text{sc}( \{e_1,\ldots,e_m\})$ denote such parameter $t$ from the second condition.

Any subhypergraph of the hypergraph of an initial segment in the colex order family is the hypergraph of an initial segment in the colex order of the corresponding size (or isomorphic to one). In particular, the hypergraph of the initial segment in the colex order has a
	supercomfortable ordering with $t=m$ (actually, any ordering with $|e_i|\leq |e_j|$ if and only if $i\leq j$, is supercomfortable in this case).
	Furthermore, if $S$ is not the initial segment in the colex order, then its hypergraph $H$ is not the hypergraph of the initial segment in the colex order; in particular, there is a minimum $i$, $2\leq i\leq k$ for which the subhypergraph $H^{(i)}$ is not the hypergraph of an initial segment in the colex order; this means that, if $\{e_1,\ldots,e_m\}$ are the edges of $H$, there is a supercomfortable ordering for which $\text{sc}( \{e_1,\ldots,e_m\})=t<m$. Thus we conclude
	\begin{quote}
		$S$ is a family with hypergraph with edges $\{e_1,\ldots,e_m\}$. Then, there is a supercomfortable ordering of  $\{e_1,\ldots,e_m\}$ with $\text{sc}( \{e_1,\ldots,e_m\})=m$ if and only if $S$ is the initial segment in the colex order.
	\end{quote}
	
\begin{claim}\label{cl:edges_in_non_colex}
	Let $S\subset \binom{[n]}{k}$ be not isomorphic to the initial segment in the colex order, with the edges in its hypergraph $\{e_1,\ldots,e_m\}$ assumed to be in a supercomfortable ordering. Let $t=\text{sc}( \{e_1,\ldots,e_m\})< m$.
	Let $\{u_1,\ldots,u_t\}$ and $\{v_1,\ldots,v_{|e_t|-1}\}$ be the vertices as denoted in \eqref{eq:vertices_of_hyp_colex} for the subgraph induced by $\{e_1,\ldots,e_t\}$.\footnote{Note that, if $e_t$ is the unique edge of its size, then there are different choices for $u_t$ and $v_{|e_{t-1}|-1},v_{|e_{t-1}|},\ldots,v_{|e_t|-1}$, all vertices of degree $1$ in $\{e_1,\ldots,e_t\}$.}
	 Then either:
 \begin{enumerate}[label*=(\Alph*)]
	\item \label{exA} There exists an edge $e_i$ with $|e_i|=|e_{t+1}|$ and $i\geq t+1$ and an edge $e\in \{e_1,\ldots,e_t\} $ with the property that $e$ has all its vertices covered by at least two edges from among $\{e_1,\ldots,e_t,e_{i}\}$.
\end{enumerate}
or, alternatively, 
\begin{enumerate}[label*=(\Alph*),resume]
	\item \label{exB} For each edge $e_i$ with $|e_i|=|e_{t+1}|$, there is a set of $t+1$ vertices which is uniquely covered by $\{e_s\; :\; s\in [1,t]\cup \{i\}\}$ (equivalently, for each $e_j, 1\leq j\leq t$, there is  a vertex of degree $1$ in the hypergraph induced by $\{e_1,\ldots,e_t\}$ covered by $e_j$ but not by $e_i$). Then there are some cases:
	\begin{enumerate}[label*=.\roman*]
		\item \label{exb.1} If $e_t$ is not the unique edge of size $|e_t|$ in $\{e_1,\ldots,e_t\}$, then
			$e_i$ does not cover a vertex from among $\{v_1,\ldots,v_{|e_t|-1}\}$.
		\item \label{exb.2}If $e_t$ is the unique edge of size $|e_t|$ in $\{e_1,\ldots,e_t\}$, then either
		\begin{enumerate}[label*=.\roman*]
		 \item \label{exb.2.1} $e_i$ does not cover a vertex from among $\{v_1,\ldots,v_{|e_{t-1}|-1}\}$, or, if that is not the case
		 \item \label{exb.2.2} $e_i$ does not cover at least two vertices from among $\{u_t,v_{|e_{t-1}|},v_{|e_{t-1}|+1},\ldots,v_{|e_{t}|-1}\}$.
		 \end{enumerate}
	\end{enumerate}
\end{enumerate}
\end{claim}
	
\begin{proof}[Proof of Claim~\ref{cl:edges_in_non_colex}]
Since $\{e_1,\ldots,e_{t}\}$ is in supercomfortable order and they induce a hypergraph of the initial segment in the colex order, it is supported (assuming the notation from \eqref{eq:vertices_of_hyp_colex}) on vertices 
$\{u_1,\ldots,u_t\}$ and $\{v_1,\ldots,v_{|e_t|-1}\}$ where $e_i$, $i\in\{1,\ldots,t\}$, is the unique edge from among $\{e_1,\ldots,e_{t}\}$ containing $u_i$, and $v_{j}$ is contained by all the edges of size $\geq j-1$.

Assume now that $\{e_1,\ldots,e_t,e_{t+1}\}$ would have been the hypergraph of the initial segment in the colex order, then the vertices of $e_{t+1}$ would have been $\{v_1,\ldots,v_{|e_t|-1}\}$ and $|e_{t+1}|-|e_t|+1$ other vertices, different from $\{u_1,\ldots,u_t\}$.
Let us thus negate such situation for each edge of size $|e_{t+1}|$.
There is some case analysis:

\medskip
\noindent\emph{Case 1: $e_t$ is the unique edge of its size from among $\{e_1,\ldots,e_t\}$.}
 If that is the case, then the all the vertices from $\{u_t,v_{|e_{t-1}|},v_{|e_{t-1}|+1},\ldots,v_{|e_{t}|-1}\}$ are covered (in the hypergraph of $\{e_1,\ldots,e_t\}$) exclusively by $e_t$, and thus their roles are interchangable. Now there is a case analysis according to the situation of the rest of the edges $e_i$ with $|e_i|=|e_{t+1}|$.

\medskip
\noindent\emph{Case 1.1: some edge $e_i$ with $|e_i|=|e_{t+1}|$ covers a vertex $\{u_1,\ldots,u_{t-1}\}$.} Say it covers $u_j$, then we are in situation \ref{exA}, as all the vertices of the edge $e_j$ have degree $\geq 2$.

\medskip
\noindent\emph{Case 1.2: none of the edges $e_i$ with $|e_i|=|e_{t+1}|$ covers a vertex $\{u_1,\ldots,u_{t-1}\}$.} This requires a further case analysis.

\medskip
\noindent\emph{Case 1.2.1: some edge $e_i$ with $|e_i|=|e_{t+1}|$ covers all the vertices $\{u_t,v_{|e_{t-1}|},v_{|e_{t-1}|+1},\ldots,v_{|e_{t}|-1}\}$.} Then we are in case \ref{exA}, and all the vertices of $e_t$ have degree $\geq 2$.

\medskip
\noindent\emph{Case 1.2.2: no edge $e_i$ with $|e_i|=|e_{t+1}|$ covers all the vertices $\{u_t,v_{|e_{t-1}|},v_{|e_{t-1}|+1},\ldots,v_{|e_{t}|-1}\}$.}
Then each edge $e_i$ with $|e_i|=|e_{t+1}|$ has at least one vertex of degree $1$ in the hypergraph of $\{e_1,\ldots,e_t,e_{i}\}$, by examining its size (it can only cover at most the vertices $\{v_1,\ldots,v_{|e_{t-1}|-1}\}$, and no vertex from among $\{u_1,\ldots,u_{t-1}\}$ while also having size $\geq |e_t|$ and not covering all the vertices of $\{u_t,v_{|e_{t-1}|},v_{|e_{t-1}|+1},\ldots,v_{|e_{t}|-1}\}$.
In this escenario, it can happen that $e_i$ does not cover all the vertices from among $\{v_1,\ldots,v_{|e_{t-1}|-1}\}$, then we are in case \ref{exb.2.1}.
the other possibility is that it covers all the vertices from among $\{v_1,\ldots,v_{|e_{t-1}|-1}\}$, then it should leave two vertices of $e_t$ free, for otherwise $e_i$ would make $\{e_1,\ldots,e_t,e_i\}$ be the hypergraph of the initial segment in the colex order. Hence we are in the situation \ref{exb.2.2}.

\medskip
\noindent\emph{Case 2: $e_t$ is not the unique edge of its size from among $\{e_1,\ldots,e_t\}$.}
 If that is the case, then the only vertices of degree one in $\{e_1,\ldots,e_t\}$ are $\{u_1,\ldots,u_t\}$, and all the vertices $\{v_1,\ldots,v_{|e_{t}|+1}\}$ have degree $\geq 2$. Again, there is a case analysis according to the situation of the rest of the edges $e_i$ with $|e_i|=|e_{t+1}|$.

\medskip
\noindent\emph{Case 2.1: some edge $e_i$ with $|e_i|=|e_{t+1}|$ covers a vertex $\{u_1,\ldots,u_{t}\}$.} Say it covers $u_j$, then we are in situation \ref{exA}, as all the vertices of the edge $e_j$ have degree $\geq 2$.

\medskip
\noindent\emph{Case 2.2: none of the edges $e_i$ with $|e_i|=|e_{t+1}|$ covers a vertex $\{u_1,\ldots,u_{t}\}$.} Then we are in a case similar to 1.2.2.
Each $e_i$ with $|e_i|=|e_{t+1}|$ has at least one vertex of degree $1$ in the hypergraph of $\{e_1,\ldots,e_t,e_{i}\}$, by examining its size (it can only cover at most the vertices $\{v_1,\ldots,v_{|e_{t}|-1}\}$, and no vertex from among $\{u_1,\ldots,u_{t}\}$ while also having size $\geq |e_t|$ by the first condition of the supercomfortable ordering, and not covering all the vertices of any edge, by the definition of the hypergraph of a family of sets.
Now, if $e_i$ covers all the vertices from among $\{v_1,\ldots,v_{|e_{t}|-1}\}$, then $\{e_1,\ldots,e_t,e_i\}$ would be a supercomfortable ordering of a hypergraph of the initial segment in the colex order (as $e_i$ would not cover any of the vertices $\{u_1,\ldots,u_t\}$ for otherwise it would not be the hypergraph of the family as the edge $e_i$ would contain another edge, as it is at least as large as $e_t$, which is the largest from among $\{e_1,\ldots,e_t\}$), thus contradicting the assumption that $t=\text{sc}( \{e_1,\ldots,e_m\})$.
Thus, there is a vertex from among $\{v_1,\ldots,v_{|e_{t}|-1}\}$ not in $e_i$, and we are in case \ref{exb.1}.
\end{proof}

Let us now resume the argument from right to left. We are going to show the contrapositive, so assume that $S$ is not the initial segment in the colex order, and let $t=\text{sc}( \{e_1,\ldots,e_m\})<m$. We use Claim~\ref{cl:edges_in_non_colex}. If we are in case \ref{exA} we may assume that the given $i$ is $t+1$ without breaking the supercomfortable ordering.
We claim that there is a set in $B_{t+1}$ of size $<t$; indeed, consider a dominating set of $\{e_1,\ldots,e_t\}$ of vertices not belonging to $e_{t+1}$, obtained greedily by first dominating $e_j$, the edge all whose degrees are $\geq 2$. In particular, this vertex $v$ covers at least two edges from among $\{e_1,\ldots,e_t\}$; by greedily adding vertices to $\{v\}$ into a full dominating set $s$, covering at least one new edge with each new vertex, we have $|s|<t$; thus a minimal dominating set contained in $s$ also have $<t$ vertices, and thus the claim is proven.

Now, for all the cases \ref{exB} we do not modify the supercomfortable ordering and just examine $B_{t+1}$.
For the case \ref{exb.1}, following the arguments of the claim, there is a vertex $v_j$ from among $\{v_1,\ldots,v_{|e_{t}|-1}\}$ not contained in $e_{t+1}$. From the hypothesis on $e_t$ (not being its unique edge of its size in $\{e_1,\ldots,e_t\}$), we conclude that the degree of $v_j$ is $\geq 2$, then any minimal set dominating $\{e_1,\ldots,e_t\}$ with vertices not in $e_{t+1}$, and containing $v_j$ has size $<t$ (by the same argument as before).
The argument for \ref{exb.2.1} is analogous as the one for \ref{exb.1} just given.
For the case \ref{exb.2.2} that is not covered by \ref{exb.2.1}, let $w_1,w_2$ be two vertices from among $\{u_t,v_{|e_{t-1}|},v_{|e_{t-1}|+1},\ldots,v_{|e_{t}|-1}\}$ not in $e_{t+1}$. Then
$\{u_1,\ldots,u_{t-1},w_1\}$ and $\{u_1,\ldots,u_{t-1},w_2\}$ are two dominating sets belonging to $B_{t+1}$, as they are both minimal.
With this exhaustive case analysis, we conclude the proof of the right to left implication, and of the lemma.
\end{proof}


\section{Hypotenusal process}\label{sec:hypotenusal}

We begin this section by introducing the hypothenusal numbers, then in Section~\ref{sec:bin_balls_wall} we introduce a procedure from where the hypothenusal numbers can be extracted and that allows to interpret the relation between the shadow $k$-binomial decompositions with the trees $T_{e_i}$. 

\subsection{Hypotenusal numbers} \label{sec:def_of_hyp_num}

The sequence of \emph{hypotenusal numbers} $a[i]$ are defined  as the difference between two consecutive \emph{Hamilton's numbers} $h[i]$:
\begin{equation} \label{eq:hypotenusal}
a[n]:=\begin{cases}
1 & n=0 \\
h[n+1]-h[n] & n>0
\end{cases}
,\;\; \text{ with }
\;\;
h[n]:=\begin{cases}
2 & n=1 \\
2+\sum_{i=1}^{n-1}\frac{\left[ (-1)^{i+1} \prod_{k=0}^i \left(h[n-i]-k\right)\right]}{(i+1)!} & n>1
\end{cases}
\end{equation}
Has the following first elements:
\begin{align}
a[0]&=1 \nonumber\\
a[1]&=1 \nonumber\\
a[2]&=2 \nonumber\\
a[3]&=6 \nonumber\\ 
a[4]&=36 \nonumber\\
a[5]&=876 \nonumber\\
a[6]&=408696 \nonumber\\
a[7]&=83762796636 \nonumber\\
a[8]&=3508125906207095591916 \label{eq:hypot_numbers}\\
a[9]&=6153473687096578758445014683368786661634996 \nonumber \\
a[10]&=18932619208894981833333582059033329370801260096062214926751788496235698477988081702676  \nonumber 
\end{align}
This sequence can be found as \cite[A001660]{oeis}.
The name ``hypotenusal numbers'' comes from the fact that they lie in the hypotenuse of the triangle found in \cite[Page~496]{lucas1891theorie} and in \cite{sylvester1887x} and reproduced in Table~\ref{t.1}.
{\footnotesize \begin{table}[htb]
	\centering
	\begin{tabular}{rrrrrrrr}
		\hline
		\textbf{1}& 0 & 0 & 0 &0 &0 &0 &\ldots \\
		\hline
		&\textbf{1}& 1& 1 &1 &1 &1 &\ldots \\
		\hline
		&& \textbf{2}& 3& 4 &5 &6 &\ldots\\
		&&1 &5& 9& 14& 20& \ldots \\
		\hline
		&&&  \textbf{6}& 15& 29 &49 &\ldots \\
		&&&5& 21& 50& 99& \ldots \\
		&&& 4 &26& 76& 175 &\ldots \\
		&&&  3 &30& 106& 281& \ldots \\
		&&&  2 &33& 139& 420& \ldots \\ 
		&&&  1 &35& 174& 594 &\ldots \\
		\hline
		&&&& \textbf{36}& 210& 804 &\ldots \\
		&&&&  35& 246& 1050& \ldots \\
		&&&&	  34 &281& 1331& \ldots \\
		&&&&	  33& 315& 1646 &\ldots \\
		&&&&\ldots & \ldots & \ldots &\ldots\\
		\hline
	\end{tabular}
	\caption{The hypotenusal numbers are those shown in boldface in the table. The Hamilton's numbers are found as $h[n+1]=1+\sum_{i=0}^{n} a[i]$, both sequences are given in \cite[Page~496]{lucas1891theorie}.  The number in the table at row $j$ column $i$ is found by adding all the numbers with row $j-1$ and column up to $i$, included (the numbers not appearing are set to $0$), with the exception of the first nonzero elements of the row; these are obtained either as said before if $(j-1,i-1)$-th element is non-zero, or by substracting one from the element above $(j-1,i)$.}\label{t.1}
\end{table}}
By inspection, one can guess that the sequence grows as a double exponential (since the number of digits is roughly doubled after each iteration in \eqref{eq:hypot_numbers}). In Section~\ref{sec:growth_hypotenusal} we show that this is indeed the case; however, the arguments become more visual after the introduction of the hypothenusal process in the following section.

\subsection{Bins, balls, and a wall.} \label{sec:bin_balls_wall}

The triple $(\mathcal{B}_t,w,t)$ formed by the triple
\[
\text{(multiset of \emph{balls} at iteration $t$, position $w$ of the \emph{wall}, \emph{iteration counter} $t$)}\]
 where
\begin{itemize}
	\item the \emph{iteration counter} $t\geq 0$, is an integer.
	\item the multiset is indexed by the integers, each integer corresponding to a \emph{bin} located and containing some \emph{balls}.
	\item Each ball is a pair of integers $(j,s)$: $s$ being a \emph{delay counter}  (it indicates that the ball should not be \emph{processed} before the iteration counter $s$), and $j$ being the \emph{position} of the ball (the bin where the ball is located).
	\item $w$ is an integer and is identified with the \emph{wall} which is located at position $w-1/2$.
\end{itemize}
is said to be a \emph{bin-balls-wall configuration}.

\paragraph{The hyptenusal process.}
Given a bin-balls-wall configuration $(\mathcal{B}_0,w_0,0)$, we can uniquely find a sequence of bin-balls-walls configurations indexed by the iteration counter, $\{(\mathcal{B}_t,\allowbreak w,t)\}_{t\in \N}$,  which is created using the followingly described \emph{hypotenusal process}.

At the beginning of each iteration step $t$ (the first iteration being $t=0$), we initialize a set of to-be-processed balls $D\subseteq \mathcal{B}_t$ as all the balls with delay counter $s\leq t$. All the balls with delay counters $s>t$ are placed in the new set $\mathcal{B}_{t+1}$.  Let $w$ be the position of the wall. We say that the ball $b=(j,s)$ in $D$ can be \emph{legally processed} if it is to the left of the wall or right at the wall: so its position $j$ satisfies $j\leq w$.

If the set of balls to be processed is not empty and none of them can be legally processed, then we say that the \emph{process terminates abruptly}. 
The legally processable ball $b=(j,s)\in D$ is processed as follows.
\begin{itemize}
	\item Add balls $\{(j,s+1),\ldots,(w-2,s+1),(w-1,s+1)\}$ at the positions $j,j+1,\ldots,w-2,w-1$,  (thus adding no ball if $j=w$) with the delay mark $s+1$, and were $w$ is the current position of the wall.
	 These new balls are added to $D$ whenever $s+1\leq t$, and, when $s+1>t$, they are added to $\mathcal{B}_{t+1}$.
	\item  remove $b=(j,s)$ from $D$. 
	\item move the wall from the position $w$ to the position $(w+1)$.
\end{itemize}
After all the balls from $D$ are processed, we increment the iteration counter $t\leftarrow t+1$ and repeat the process with $(\mathcal{B}_t,w,t)\leftarrow (\mathcal{B}_{t+1},w',t+1)$, where $w'$ is the last position of the wall obtained during the iteration $t$.

The sequence $\{(\mathcal{B}_t,w,t)\}_{t\in \N}$ obtained from the initial condition $(\mathcal{B}_0,w_0,0)$ by this procedure is said to be the \emph{output of the hypotenusal process with initial condition $(\mathcal{B}_0,w_0,0)$}.

The pair $(\mathcal{B}_0,w,0)$ formed by an initial configuration of balls $\mathcal{B}_0$, with finitely many balls in each bin, together with an initial wall position $w$ is said to be \emph{good} if the process described above never terminates abruptly.
From each good initial configuration $(\mathcal{B}_0,w,0)$, we obtain the sequence $w_0,w_1,\ldots,w_n,\ldots$, where $w_j$ is the position of the wall at the beginning of the iteration step $j$  (so $w_0=w$).
This process is dubbed as the \emph{hypotenusal process} since the sequence of the hypotenusal numbers can be retrieved from a certain initial condition (see Lemma~\ref{lem:hip_num}); however, other initial conditions give rise to other sequences, thus generalizing this sequence of hypotenusal numbers. The name bins-balls-wall comes from imagining the process as a set of balls that push a wall to the right, while leaving the trace of their trajectory as balls to be later reused. See Figure~\ref{fig:hyp_example} for an example.

{\small
	\begin{figure}[htb]
		\centering 
		\begin{tikzpicture}
		\node (a) at (0,0)
		{
			\begin{tikzpicture}[scale=0.5]
			\foreach \i in {0,...,11}
			\draw[lightgray] (\i,0)--(\i,1);
			\draw[lightgray] (0,0)--(11,0);
			\draw[color=red] (0.5,0.3) circle(3pt);
			\draw[color=blue] (0.5,0.6) circle(3pt);
			\draw[thick] (0,-0.5)-- ( 0,1.5);
			\end{tikzpicture}
		};
		\node (c) at (a.south) [anchor=north,yshift=-0.2cm]
		{
			\begin{tikzpicture}[scale=0.5]
			\foreach \i in {0,...,11}
			\draw[lightgray] (\i,0)--(\i,1);
			\draw[lightgray] (0,0)--(11,0);
			\draw[color=blue] (0.5,0.6) circle(3pt);
			\draw[thick] (1,-0.5)-- ( 1,1.5);
			\end{tikzpicture}
		};
		
		\node (e) at (c.south) [anchor=north,yshift=-0.2cm]
		{
			\begin{tikzpicture}[scale=0.5]
			\foreach \i in {0,...,11}
			\draw[lightgray] (\i,0)--(\i,1);
			\draw[lightgray] (0,0)--(11,0);
			\draw[color=blue] (0.5,0.6) circle(3pt);
			\draw[thick] (2,-0.5)-- ( 2,1.5);
			\end{tikzpicture}
			
		};
		
		\node (e1) at (e.south) [anchor=north,yshift=-0.2cm]
		{
			\begin{tikzpicture}[scale=0.5]
			\foreach \i in {0,...,11}
			\draw[lightgray] (\i,0)--(\i,1);
			\draw[lightgray] (0,0)--(11,0);
			\draw[color=blue] (0.5,0.6) circle(3pt);
			\draw[color=blue] (1.5,0.6) circle(3pt);
			\draw[thick] (3,-0.5)-- ( 3,1.5);
			\end{tikzpicture}
		};
		\node (f) at (e1.south) [anchor=north,yshift=-0.2cm]
		{
			\begin{tikzpicture}[scale=0.5]
			\foreach \i in {0,...,11}
			\draw[lightgray] (\i,0)--(\i,1);
			\draw[lightgray] (0,0)--(11,0);
			\draw[color=red] (0.5,0.6) circle(3pt);
			\draw[color=blue] (1.5,0.9) circle(3pt);
			\draw[thick] (3,-0.5)-- ( 3,1.5);
			\end{tikzpicture}
		};
		
		\node (g) at (f.south) [anchor=north,yshift=-0.2cm]
		{
			\begin{tikzpicture}[scale=0.5]
			\foreach \i in {0,...,11}
			\draw[lightgray] (\i,0)--(\i,1);
			\draw[lightgray] (0,0)--(11,0);
			\draw[color=red] (0.5,0.6) circle(3pt);
			\draw[color=red] (1.5,0.6) circle(3pt);
			\draw[color=red] (2.5,0.6) circle(3pt);
			\draw[color=blue] (1.5,0.9) circle(3pt);
			\draw[thick] (4,-0.5)-- ( 4,1.5);
			\end{tikzpicture}
		};
		
		\node (h) at (g.south) [anchor=north,yshift=-0.2cm]
		{
			
			\begin{tikzpicture}[scale=0.5]
			\foreach \i in {0,...,11}
			\draw[lightgray] (\i,0)--(\i,1);
			\draw[lightgray] (0,0)--(11,0);
			\draw[color=red] (0.5,0.6) circle(3pt);
			\draw[color=red] (1.5,0.6) circle(3pt);
			\draw[color=red] (2.5,0.6) circle(3pt);
			\draw[color=blue] (1.5,0.9) circle(3pt);
			\draw[color=blue] (2.5,0.9) circle(3pt);
			\draw[color=blue] (3.5,0.9) circle(3pt);
			\draw[thick] (5,-0.5)-- ( 5,1.5);
			\end{tikzpicture}
			
		};

		\node (i) at (h.south) [anchor=north,yshift=-0.2cm]
		{
			\begin{tikzpicture}[scale=0.5]
			\foreach \i in {0,...,11}
			\draw[lightgray] (\i,0)--(\i,1.8);
			\draw[lightgray] (0,0)--(11,0);
			\foreach \i in {0,...,4}
			\draw[color=red] (\i+0.5,0.3) circle(3pt);
			\foreach \i in {1,...,5}
			\draw[color=blue] (\i+0.5,0.6) circle(3pt);
			\foreach \i in {1,...,6}
			\draw[color=magenta] (\i+0.5,0.9) circle(3pt);
			\foreach \i in {2,...,7}
			\draw[color=violet] (\i+0.5,1.2) circle(3pt);
			\foreach \i in {2,...,8}
			\draw[color=olive] (\i+0.5,1.5) circle(3pt);
			\foreach \i in {3,...,9}
			\draw[color=brown] (\i+0.5,1.8) circle(3pt);

			\draw[thick] (11,-0.5)-- ( 11,2.5);
			
			%
			%
			\end{tikzpicture}
		};
		\draw [->] (a)--(c);
		\draw [->] (c)--(e);
		\draw [->] (e)--(e1);
		\draw [->] (e1)--(f);
		\draw [->] (f)--(g);
		\draw [->] (g)--(h);
		\draw [->] (h)--(i);
	\end{tikzpicture}
	\caption{Hypotenusal process for $(\{(0,0),(0,0)\},0,0)$; first picture is the initial configuration, second picture is after the first ball of step 0 processed. Third picture corresponds to the configuration after the two balls in the first configuration have been processed (end of iteration $0$). Fourth picture represents the initial position at end of iteration $1$, then the first and second ball of iteration $2$. Final picture corresponds to the end of iteration $3$.} \label{fig:hyp_example}
\end{figure}
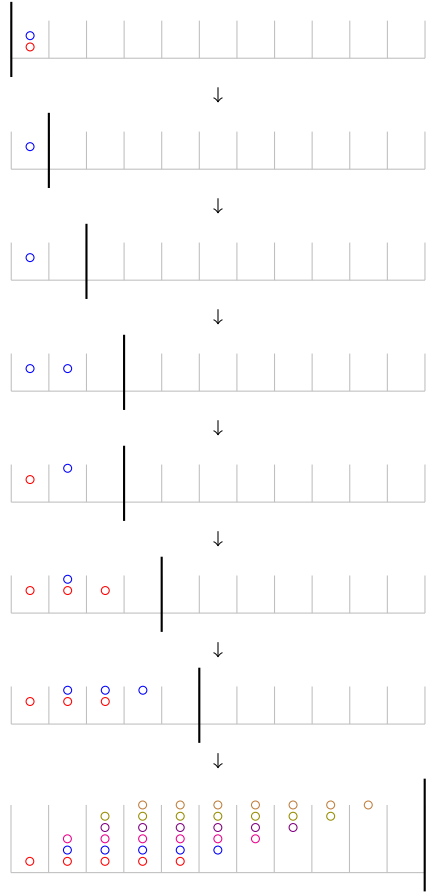}

\subsection{Hypothenusal process, hypothenusal numbers and families of \texorpdfstring{$k$}{k}-sets}

The following Observation~\ref{obs:dict_bbw-sbd} justify the consideration of the hypothenusal process in this work, while Lemma~\ref{lem:hip_num} justify its name. The proof of the first observation is given within the proof of Proposition~\ref{prop.translation}.

\begin{observation} \label{obs:dict_bbw-sbd}
The family $S\subseteq \binom{[n]}{k}$ with hypergraph $H=([n],\{e_1,\ldots,e_m\})$, produces an initial set of balls-bins-wall configuration $(\mathcal{B}_0,w,0)$ as follows.
If a non-root leaf vertex of the tree $T_{e_i}$ associated to the edge $e_i$ gives a coefficient $\binom{n-j-s}{k-s}$, then we place a ball in the position $j$ and with delay $s$. The wall is initially located at position $w=0$.
We use $(\mathcal{B}(S),0,0)$ to denote this initial configuration.
\end{observation}

\begin{lemma} \label{lem:hip_num}
	The sequence of hypotenusal numbers $a[i]$, $i\geq 0$, can be obtained as  $a[i]=w_{i+1}-w_{i}$, where $w_i$ is the sequence of walls obtained 
	from the hypothenusal process with the initial configuration $(\mathcal{B}_0,\allowbreak w,0)=\allowbreak (\{(0,0)\},1,0)$, and with $w_{0}=w$. 
\end{lemma}

\begin{proof}
	Consider the following table $T$ (which is similar to Table~\ref{t.1}, but not the same):
	\begin{center}
		\begin{tabular}{rrrrrrrr}
			\hline
			1& \textbf{1} & 1 & 1 &1 &1 &1 &\ldots \\
			\hline
			&1& \textbf{2}& 3 &4 &5 &6 &\ldots \\
			\hline
			&& 2& 5& 9 &14 &20 &\ldots\\
			&&1 &\textbf{6}& 15& 29& 49& \ldots \\
			\hline
			&&&  6& 21& 50 &99 &\ldots \\
			&&&5& 26& 76& 175& \ldots \\
			&&& 4 &30& 106& 281 &\ldots \\
			&&&  3 &33& 139& 420& \ldots \\
			&&&  2 &35& 174& 594& \ldots \\ 
			&&&  1 &\textbf{36}& 210& 804 &\ldots \\
			\hline
			&&&& 36& 246& 1050 &\ldots \\
			&&&&  35& 281& 1331& \ldots \\
			&&&&	  34 &315& 1646& \ldots \\
			&&&&	  33& 348& 1994 &\ldots \\
			&&&&\ldots & \ldots & \ldots &\ldots\\
			\hline
		\end{tabular}
	\end{center}
	where the numbers in boldface are the hypothenusal numbers, the position $T(i,j)$ is obtained from the previous column as follows
	\begin{equation} \nonumber
	\begin{cases}
	1 & i=1 \wedge j\geq 1 \\
	0 &  i\geq 2 \wedge j=1 \\
	\sum_{k=1}^{i}T(k,j-1) & i\geq2\wedge j\geq 2 \wedge T(i,j-1)>0 \text{ (sums along prev. column from top to curr. row) } \\
	T(i-1,j) & i\geq2\wedge j\geq 2\wedge T(i-1,j)>0 \wedge T(i-1,j-1)>0 \wedge T(i,j-1)=0 \text{ (repeat last)} \\
	T(i-1,j)-1 & i\geq 2\wedge j\geq 2\wedge T(i-1,j)>0\wedge T(i-1,j-1)=0\wedge T(i,j-1)=0 \text{ (decrease by 1)} \\
	0 & \text{ otherwise}
	\end{cases}
	\end{equation}
	That is to say: the first row is the all $1$ and the first column is the all $0$ (except the first element of the row being $1$). The boldfaced element on each column, besides the on the first column, appears at the row where the previous column has its last non-zero element. Then, from that point on, we repeat the boldfaced element and decrease it by one unit in the remaining rows. The first element in the column are the partial sums of the elements in the previous column up to that row's position (as long as they they are non-zero).
	
	Now, by closely examining the hypothenusal process induced by the starting position configuration $(\{(0,0)\},0,0)$, one can observe deduce that the element $T(i,j)$ gives the number of balls in the $i$-th bin at the end of the $(j-1)$-th iteration. 
	
	The element $L(i,j)$ in Table~\ref{t.1} is (considering the second row to be the first one in the notation $L$) generated as follows:
	\begin{equation}
	\begin{cases}
	1 & i=1 \wedge j\geq 1 \\
	0 & i=1 \wedge j>1 \\
	\sum_{k=1}^j L(i-1,k) & i\geq 2 \wedge j\geq 2 \wedge L(i-1,j-1)>0 \\
	L(i-1,j)-1 & i\geq 2 \wedge j\geq 2 \wedge L(i-1,j-1)=0 \wedge L(i-1,j)\geq 1 \\
	0 &\text{otherwise}
	\end{cases}
	\end{equation}
	creating the following table $L$ (which is the table from Table~\ref{t.1} without the first row):
	\begin{center}
		\begin{tabular}{rrrrrrr}
			\hline
			\textbf{1}& 1& 1 &1 &1 &1 &\ldots \\
			\hline
			& \textbf{2}& 3& 4 &5 &6 &\ldots\\
			&1 &5& 9& 14& 20& \ldots \\
			\hline
			&&  \textbf{6}& 15& 29 &49 &\ldots \\
			&&5& 21& 50& 99& \ldots \\
			&& 4 &26& 76& 175 &\ldots \\
			&&  3 &30& 106& 281& \ldots \\
			&&  2 &33& 139& 420& \ldots \\ 
			&&  1 &35& 174& 594 &\ldots \\
			\hline
			&&& \textbf{36}& 210& 804 &\ldots \\
			&&&  35& 246& 1050& \ldots \\
			&&&	  34 &281& 1331& \ldots \\
			&&&	  33& 315& 1646 &\ldots \\
			&&&\ldots & \ldots & \ldots &\ldots\\
			\hline
		\end{tabular}
	\end{center}
To conclude the almost-equality between the two tables (for almost all $(i,j)$, we have $T(i,j)=L(i,j-1)$)
we observe that:
 $T(i,j)=T(i-1,j)+T(i,j-1)$ if $T(i,j-1)\neq 0$, and $T(i,j)=\max\{T(i-1,j)-1,0\}$ if $T(i,j-1)= 0$, while
$L(i,j)=L(i-1,j)+L(i,j-1)$ if $L(i-1,j-2)\neq 0$, and $L(i,j)=\max\{L(i-1,j)-1,0\}$ if $L(i-1,j-1)= 0$, and $L(i,j)=L(i-1,j)+L(i-1,j-1)$ if $L(i-1,j-1)>0$ and $L(i-1,j-2)=0$.
\end{proof}

\begin{lemma}
	Assume that the hypothenusal process induced by the initial configuration $(\mathcal{B}_0,\allowbreak w_0,0)$
	\begin{itemize}
		\item  does not end abruptly,
		\item has $w_0,w_1,\ldots,w_j,\ldots$ as its sequence of walls,
		\item at the beginning of iteration $i$, there is a ball $(w_i-1,i)$.
	\end{itemize}
	Then, $w_{i+j+1}-w_{i+j}\geq a[j]$ for each $j\geq 0$, where $a[j]$ is the $j$-th hypothenusal number.
\end{lemma}

\begin{proof}
	A ball $(i,j)$ in the hypothenusal process is said to be a \emph{descendant} of another $(s,j-1)$ if the ball $(i,j)$ is produced as a consequence of processing the ball $(s,j-1)$.
	The wall is pushed by the balls it has weakly to its left: if the ball $(s,j)$ is such that $s\leq w$, $w$ the position of the wall, then it generates $w-s$ descendant balls. By Lemma~\ref{lem:hip_num}, $(w_i-1,i)$ generates, at each iteration $i+j$, at least $a[j]$ balls to the left of the wall that are going to be processed in the next step (we just take in consideration the descendant balls up to the position that would give the hypothenusal numbers, and no more; these relative positions are the ones deduced from the process in Lemma~\ref{lem:hip_num}.) The result then follows.
\end{proof}

\subsection{Growth of the hypotenusal numbers} \label{sec:growth_hypotenusal}

For a detailed account of the growth of the hypothenusal numbers see
\cite[Page~71, second equation from the bottom]{sylvester1888iv} where the authors show that, asymptotically,
\begin{align} \label{eq:asymp_hamilt_sylv}
\frac{a[i]}{2}&=\left(\frac{a[i-1]}{2}\right)^2+
\frac{4}{3}\left(\frac{a[i-1]}{2}\right)^{\frac{3}{2}}
+
\frac{11}{18}\left(\frac{a[i-1]}{2}\right)+\frac{10}{81}\left(\frac{a[i-1]}{2}\right)^{\frac{3}{4}} \nonumber \\
&\qquad +\frac{11}{45}\left(\frac{a[i-1]}{2}\right)^{\frac{1}{2}}
\left[\frac{2^3}{3^3}\left(\frac{a[i-1]}{2}\right)^{\frac{1}{8}}+\frac{2^4}{3^4}\left(\frac{a[i-1]}{2}\right)^{\frac{1}{16}}+\ldots+\frac{2^j}{3^j}\left(\frac{a[i-1]}{2}\right)^{\frac{1}{2^j}}+\ldots\right] 
\end{align}
Note that \eqref{eq:asymp_hamilt_sylv} corrects some imprecisions with respect to the asymptotic expression appearing in \cite[Page~311, third equation from the top]{sylvester1887x}.

In the remaining of this section we use the ideas of correspondence between the hypothenusal process and the hypothenusal numbers given in Lemma~\ref{lem:hip_num} to provide some simple recurrences between $a[i]$ and $a[i-1]$ and thus deducing an expression for the asymptotic growth.
For a lower bound for $a[i]$ we can show:
\begin{claim}
	\label{cl.2}
	\begin{equation}\label{eq.hypoth}
	a[i]\geq \binom{a[i-1]+1}{2}+a[i-1], \text{ for }i\geq 2
	\end{equation}
\end{claim}
\noindent which implies that 
\begin{equation}\label{eq.hypoth2}
a[i]\geq\binom{a[i-1]+1}{2}+a[i-1]=\frac{(a[i-1]+1)a[i-1]}{2}+\frac{a[i-1]}{2}\geq \frac{a[i-1]^2}{2}, \text{ for }i\geq 2
\end{equation}
Now, for $i =4$, $2^{2^{i-2}+1}=32$ and $a[4]=36$, a straight inductive argument shows that, if $a[i-1]\geq 2^{2^{(i-1)-2}+1}$ then  $a[i]\geq 2^{2^{i-2}+1}$ as $a[i]\geq a[i-1]^2/2\geq 2^{2^{i-3}+1}\cdot 2^{2^{i-3}+1}/2 = 2^{2^{i-2}+1}$. Hence, we obtain that:
\begin{equation}\label{eq.hypoth3}
a[i]\geq 2^{2^{i-2}+1}, \text{ for }i\geq 4
\end{equation}
which is a double exponential increase. 
On the other hand, for an upper bound:
\begin{claim} \label{cl.3}
	\begin{equation} \label{eq.hypoth4}
	a[i-1]^2\geq a[i], \text{ for }i\geq 4\qquad \text{ and furthermore } \qquad 
	2^{2^i}\geq a[i] 
	\end{equation}
\end{claim}
\noindent and we conclude that
\begin{equation} \label{eq.hyp_all}
2^{2^i}\geq a[i]\geq 2^{2^{i-2}+1},\text{ for $i\geq 4$}
\end{equation}

\begin{proof}[Proof of Claim~\ref{cl.2}]
	$a[i]$ is the number of balls on the bins and balls configurations that are active at the beginning of step $i$.
	Consider the configuration for $i=2$ that leads to $a[2]=2$. This has a ball at position $0$, a ball at position $1$, no ball on position $2$, and the wall at position $3$; so $a[2]=2$ as there are two balls that are active. 
	We modify the hypotenusal process as follows, first observe that, from $a[1]$ onwards, all the balls are to the left of the wall, and actually the position $w$ has no balls. Then we place all the balls to the position $w-1$, run the hypotenusal process, and repeat this procedure (relocating the balls to the spot $w'-1$). Since all the balls are relocated to the right of them, they generate less descendant balls. Therefore, if in this process we obtain the numbers $a'[i]$, we have $a[i]\geq a'[i]$.
	Now we shall observe that we obtain that
	$a'[i]=\binom{a'[i-1]+1}{2}+a'[i-1]$, since the $a[i-1]$ balls generate a triangle with $1+2+\ldots+a'[i]=\binom{a'[i-1]+1}{2}$ balls in it, without considering the original $a'[i-1]$ balls, that can be later added.
	Therefore, we obtain the desired inequality \eqref{eq.hypoth} for $a[i]$.
\end{proof}

\begin{proof}[Proof of Claim~\ref{cl.3}]
	We use a similar strategy as in Claim~\ref{cl.2}. The first elements for the second inequality are checked by hand. $a[i]$ is the number of balls that are active at the beginning of step $i$. Since the $a[i]$ is obtained by just beginning with a single ball (and no further balls with delays to the right of the wall), the balls in the bins $w-1$, $\ldots$, $w-a[i]$ form an isosceles, right-angle triangle with side $a[i]$, and all the balls have been generated by previous balls, all located to the left of the $w-a[i]$. Now, for all the positions $j\in[0,w-a[i]-1]$, the number of balls at position $j$ is strictly larger than the number of balls at position $j-1$; this can be shown by induction and the fact that the balls at the column $j$ are all the balls at column $j$ at step $i$ plus all the balls at columns to the left of $j$. The result then follows by induction (with the caveat that, since we have a triangle to the right of the columns $j\in[0,w-a[i]-1]$, the result also follows, as we will be adding some (strictly) more elements to the right columns, in the next iteration).
	
	Since the number of balls at these positions decrease, they can be encased in a right-angle, isosceles triangle whose side is $a[i]$ as well. This shows that $a[i]^2+a[i]\geq a[i+1]$ (essentially, the two isosceles triangles are covered by square term $a[i]^2$, but then the main diagonal should be counted twice, hence the extra term $a[i]$).
	Finally, when $i\geq 4$, we can be sure that, in the part of $j\in[0,w-a[i]-1]$, the terms from the bin $w-a[i]-1$ towards the beginning decrease by more than two units from one bin to the next, and thus the claimed inequality follows).
\end{proof}


\section{First applications of the hypotenusal process} \label{sec:app_hip}

In this section we use the hypothenusal process to argue the main results of this article.

\subsection{Monotonicity of the shadow k-binomial decomposition} \label{sec:proof_prop_trans}

Theorem~\ref{thm:main} follows from the following proposition.

\begin{proposition} \label{prop.translation}
	Let $S\subset \binom{[n]}{k}$ be a family of sets, and $(\mathcal{B}(S),0,0)$ its initial bins-balls-wall configuration of $S$. Then
	\begin{enumerate}[label=(\roman*)]
		\item\label{en:prop_trans_1} the hypotenusal process from Section~\ref{sec:bin_balls_wall} does not end abruptly for the initial conditions  $(\mathcal{B}(S),0,0)$.
		\item \label{en:prop_trans_2}	If $w_1,\ldots,w_{k}$ are the positions of the wall at the end of the steps $1,\ldots, k$, then the shadow $k$-binomial decomposition of $S$ is given by
		\[
		|S|=\binom{n-w_1-1}{k}+\binom{n-w_2-2}{k-1}+\ldots+\binom{n-w_{k-1}-(k-1)}{2}+\binom{n-w_{k}-(k)+1}{1}
		\]
		\item\label{en:prop_trans_3} If $\mathcal{B}_{k+1}$ is the set of balls at the end of the step $k$, then
		\begin{align}
		|S|&\doteq\binom{n-w_1-1}{k}+\binom{n-w_2-2}{k-1}+\ldots+\binom{n-w_{k-1}-(k-1)}{2}+\binom{n-w_{k}-(k)+1}{1}+\nonumber\\ 
		&\qquad +\sum_{(j,s)\in \mathcal{B}_{k+1}} \binom{n-j-s}{k-s} 
		\end{align}
		with all the delays $s$ satisfying $s\geq k+1$.
	\end{enumerate}
\end{proposition}

\noindent The second part of Proposition~\ref{prop.translation} corresponds to Theorem~\ref{thm:hypotenusal}. The proof of Theorem~\ref{thm:main} can be found at the end of this section.

\begin{proof}[Proof of Proposition~\ref{prop.translation}]
Let us begin by showing the part \ref{en:prop_trans_2} assuming \ref{en:prop_trans_1}.
We use the translation from Observation~\ref{obs:dict_bbw-sbd}: the ball corresponding to the binomial coefficient $\binom{n-j-s}{k-s}$ is located at position $j$ with delay $s$. Observe that all the balls have delay $\geq 1$, thus at iteration $0$ nothing happens.

Let us proceed by induction. We begin with the expression $\binom{n-w_0}{k}=\binom{n-0}{k}$ at the end of the iteration $0$, regarding the size of $S$ (since no elements have been substracted yet). 
We break
\begin{equation}\label{eq:break_ini}
\binom{n-w_0}{k}:=\binom{n-w}{k}\doteq\binom{n-w-1}{k}+\binom{n-w-1}{k-1}
\end{equation}
If there are no balls with delay $1$, then we
determine that $w_1=0$, keep $\binom{n-w_1-1}{k}$ as the first term of the shadow $k$-binomial decompsition, and continue the bin-ball-wall process with
$\binom{n-w_1-1}{k-1}$.
If there is a ball with delay $1$, then by \ref{en:prop_trans_1} the ball $(j,1)$ can be processed and thus we have $j\leq w$ so $j\leq 0$, which is related to a binomial coefficient $\binom{n-j-1}{k-1}$. We can break $\binom{n-j-1}{k-1}$ as follows:
\begin{align}
&\binom{n-j-1}{k-1}\doteq\binom{n-j-1-1}{k-1}+\binom{n-j-1-1}{k-2} \nonumber \\
&\doteq\binom{n-j-1-2}{k-1}+\binom{n-j-1-2}{k-2}+\binom{n-j-1-1}{k-2} \nonumber \\
&\doteq\cdots \nonumber \\
&\doteq \binom{n-w-1}{k-1}+\binom{n-w-1}{k-2}
+\cdots+\binom{n-j-1-1}{k-2} \nonumber \\ 
&\doteq \binom{n-w-1}{k-1}+\binom{n-(w-1)-2}{k-2}+\binom{n-(w-2)-2}{k-2}+\cdots+\binom{n-j-2}{k-2} \label{eq:break_2}
\end{align}
so we cancel both terms $\binom{n-w-1}{k-1}$ from \eqref{eq:break_ini} and \eqref{eq:break_2} and we are adding balls to the positions $j,j-1,\ldots,w-2,w-1$ with delay $2$ as claimed. Then we break the term 
$\binom{n-w-1}{k}$ using the sum of binomial coefficients as in \eqref{eq:break_ini} and continue to process the balls of delay $1$ until none is left to be processed, breaking the remaining term $\binom{n-w'}{k}\doteq \binom{n-w'-1}{k}+\binom{n-w'-1}{k-1}$, determining that $w_1=w$ at the end of iteration $1$ and continue doing the same for the term $\binom{n-w'-1}{k-1}$.

In general we identify the wall at position $w$ after finishing the iteration $t$ with the binomial coefficient $\binom{n-w-t}{k-t}$, which we call the \emph{leading binomial coefficient}.
(This becomes the term $\binom{n}{k}$ at iteration $0$, setting the stage for the base of the induction.)

To perform the iteration $t+1$.
We begin by breaking the leading binomial coefficient as:
\begin{equation}\label{eq:break_wall}
\binom{n-w-t}{k-t}\doteq\binom{n-w-t-1}{k-t}+\binom{n-w-t-1}{k-t-1}\doteq\binom{n-(w+1)-t}{k-t}+\binom{n-w-(t+1)}{k-(t+1)}
\end{equation}
then, if there is no ball to be processed, the term $\binom{n-(w+1)-t}{k-t}$ becomes the corresponding term in the shadow $k$-binomial decomposition with the wall being $w$, and we move on towards the next iteration with leading binomial coefficient $\binom{n-w-(t+1)}{k-(t+1)}$.
If there are some balls with delay counter $t+1$, say $(j,t+1)$ and can be legally processed so $j\leq w$, then it corresponds to a binomial coefficient $\binom{n-j-(t+1)}{k-(t+1)}$. 
Break $\binom{n-j-(t+1)}{k-(t+1)}$ as follows:
\begin{align}
&\binom{n-j-(t+1)}{k-(t+1)}\doteq\binom{n-j-1-(t+1)}{k-(t+1)}+\binom{n-j-1-(t+1)}{k-(t+2)} \nonumber \\
&\doteq\binom{n-j-2-(t+1)}{k-(t+1)}+\binom{n-j-2-(t+1)}{k-(t+2)}+\binom{n-j-1-(t+1)}{k-(t+2)} \nonumber \\
&\doteq\cdots \nonumber \\
&\doteq \binom{n-w-(t+1)}{k-(t+1)}+\binom{n-w-(t+1)}{k-(t+2)}
+\cdots+\binom{n-j-1-(t+1)}{k-(t+2)} \nonumber \\ 
&\doteq \binom{n-w-(t+1)}{k-(t+1)}+\binom{n-(w-1)-(t+2)}{k-(t+2)}+\binom{n-(w-2)-(t+2)}{k-(t+2)}+\cdots+\binom{n-j-(t+2)}{k-(t+2)} \nonumber 
\end{align}
so we cancel both terms $\binom{n-w-(t+1)}{k-(t+1)}$ and we are adding balls to the positions $j,j-1,\ldots,w-2,w-1$ with delay $t+2$ as claimed.
At this point we break the wall again
\[
\binom{n-w-t-1}{k-t}\doteq\binom{n-(w+1)-t-1}{k-t}+\binom{n-(w+1)-(t+1)}{k-(t+1)}
\]
which exactly corresponds to moving the wall from $w$ to position $w+1$. And the process is then repeated.

Therefore, the movement of the balls and the wall is the same as having a leading binomial coefficient and process the binomial coefficients, as long as it leads in a good way ($j\leq w$), so it can be legally processed.

Now that we have seen the translation between the operations at the level of binomial coefficients and the one at the level of bins-balls-wall, let us show that, actually, the initial configurations do not end abruptly (part \ref{en:prop_trans_1}).

If we process a ball at some point, then all the descendant balls are also processable. Furthermore:
\begin{observation} \label{obs:1}
If a ball is not eliminated (equivalently: it generates some descendants), and the process does not abruptly end, then there are always balls to be processed (comming from the descendants of the balls that have not been completely eliminated).
\end{observation}

Let $B_j$ 
be the subsets of $[n]$ from Lemma~\ref{lem:d} associated with $e_j$.
 Each of these generates a ball from the initial configuration (associated to the path induced by the leaf vertex hanging from a vertex associated to the set in the tree $T_{e_j}$).
Each of the sets in $B_j$ has size $\leq j-1$, since each set $s\in B_j$ dominates $j-1$ edges, and it is minimal with such property, thus we need at most $j-1$ elements to intersect that many edges.
Since the delay in each edge is given by the number of elements that we have already selected (namely $|e_j|$ plus the different elements that we should select from the different sets $s_i$ along the leaft-to-root path until we reach the desired $s$) we have an binomial of the type
\begin{equation}\label{eq:picked_avoided}
\binom{[n]-|e_j|-|s|-|\{\text{elements picked to reach s}\}|}{k-|e_j|-|\{\text{elements picked to reach s}\}|}, \qquad \text{the elements in $s$ are said to be \emph{avoided}}
\end{equation}
then we are going to be able to process such a ball if we have previously processed $|s|$ other elements in previous iterations (or at the same iteration but previously), since after each iteration the correting term on the wall (that translates between delaying/iteration time, value of the wall, and the binomial numerator) is also lowered by one. If we order the edges comfortably (ordered increasingly according to their size), and then arbitrarily from among the edges of the same size, then this is going to be satisfied, as the smaller edges are processed first, and the size of $s$ is at most the number of previously seen edges (and thus the number of times the wall has already been moved, at least). This shows that the set of balls induced by a hypergraph is always good, thus proving \ref{en:prop_trans_1}.

Part~\ref{en:prop_trans_3} follows from the arguments that shows Part~\ref{en:prop_trans_2} and by the fact that all the operations that have been conducted in the parallelism between the hypotenusal process and the correspondence with the ball and the binomial coefficients are translation invariant. We are also using \ref{en:prop_trans_1} to know that the process indeed can be carried over until the step $k$.
\end{proof}

\begin{proof}[Proof of Theorem~\ref{thm:main}]
	Due to Proposition~\ref{prop.translation}~\ref{en:prop_trans_1}, the hypotenusal process never ends abruptly. As the positions of the wall are clearly non-increasing, by the nature of the process, the dictionary given by Proposition~\ref{prop.translation}~\ref{en:prop_trans_2} shows that the coefficients in the shadow $k$-binomial decomposition are monotonously decreasing (with the exception of the last one).
\end{proof}
An important remark follows from Proposition~\ref{prop.translation}~\ref{en:prop_trans_3} and Theorem~\ref{thm:charac}.
\begin{remark}\label{rmk:extremal_wall}
$S\subset \binom{[n]}{k}$ is extremal $\iff$ the bin-ball-wall process induced by the hypergraph of $H$ is such that $n-w_k-k+1>0$, so $w_k\leq n-k$.

In particular, the number of edges in the hypergraph of an extremal family is $\leq n-k$.
\end{remark}

Theorem~\ref{thm:hypotenusal} is given by Proposition~\ref{prop.translation}~\ref{en:prop_trans_2}. Theorem~\ref{thm:solid_part} also follows from Proposition~\ref{prop.translation}~\ref{en:prop_trans_2} after realizing that the hypergraph of the family $S'$ from \eqref{eq:family_ext} is the hypergraph of $S$ with the additiono of $r$ isolated vertices. Then the computations from Proposition~\ref{prop.translation}~\ref{en:prop_trans_2} is the same but with substituting $n$ by $n+r$; then there exists an $r_0$ after which substituting $n+r$ with $r\geq r_0$ makes the last term $(n+r)-w_k-(k)+1$ to be strictly positive, and thus by Theorem~\ref{thm:charac}, the family $S'$ from Theorem~\ref{thm:solid_part} is extremal.

\subsection{On the initial segment in the colex order: alternative proof of Theorem~\ref{thm:card}} \label{sec:alt_to_mtfg}

\begin{proof}[Proof of Theorem~\ref{thm:card}]
From right to left, we can use the examples of extremal families from \cite[Proposition~2.5, Examples~2.4]{furgri86}; in Section~\ref{sec:cons_a_b} the reader can find some other concrete families that give further examples in the form of Construction~A and Construction~B.

From left to right, we show the contrapositive. Assume that $S$ is not the initial segment in the colex order and that $\ell(\a)<k$. Using Claim~\ref{cl:edges_in_non_colex}, and its notation, and the argument in Lemma~\ref{lem:colex_B_i}, we conclude that:
\begin{obsalt} \label{obs:coeff}
The tree associated to $B_{t+1}$ either induces a binomial coefficient of the type $\binom{n-|e_{t+1}|-|s|}{k-|e_{t+1}|}$ with $|s|<t$ for the cases \ref{exA}, \ref{exb.1}, \ref{exb.2.1},
or at least one binomial coefficients of the type $\binom{n-|e_{t+1}|-(t)}{k-|e_{t+1}|}$ and at least one  of the type $\binom{n-(|e_{t+1}|+1)-(t)}{k-(|e_{t+1}|+1)}$.\footnote{In the second case, the term $\binom{n-|e_{t+1}|-(t)}{k-|e_{t+1}|}$ follows by consider the root of the tree $T_{e_{t+1}}$ with the blocking set $\{u_1,\ldots,u_{t-1},w_1\}$, and the term $\binom{n-(|e_{t+1}|+1)-(t)}{k-(|e_{t+1}|+1)}$ as its neighbour were we have picked $w_1$ and added $w_2$ to the blocking set as a forbidden vertex.}
\end{obsalt}
Following the translation of Observation~\ref{obs:dict_bbw-sbd} to the hypothenusal process, in both cases these binomial coefficients are associated to balls that are not eliminated by the wall, and thus they have always descendants. This implies that the last two coefficients of the shadow k-binomial decomposition always differ by, at least, a unit. Since we are assuming that the family is extremal, the shadow $k$-binomial decomposition is the same as the $k$-binomial decomposition, and thus this case is shown.

Furthermore, by the preceding argument, together with the fact that, in order to determine that the hypergraph is not the hypergraph of the initial segment in the colex order we need the interaction of at least two edges of the hypergraph, we conclude that: there exists at least a pair of consecutive coefficients of the shadow $k$-binomial decomposition whose difference is, at least, two. Therefore, it cannot be the case that $m=\binom{q}{k}-1$. This finishes the proof.
\end{proof}

\subsection{Adding and substracting elements to an extremal family of sets} \label{sec:adding_and_substracting}

Once an extremal family has been found, it naturaly generates several extremal families closely related by their hypergraphs. The arguments leading to Proposition~\ref{prop.translation} show that the shadow $k$-binomial decomposition of any subhypergraph generate strictly less balls; indeed, order the edges in such a way that the removed edges are the last ones, thus those associated balls do not appear, and the following result follows.

\begin{proposition} \label{p.other_extremal_families} \label{thm:families}
	Let $S$ be a family of $k$-sets with hypergraph $H(S)$. 
	Let $(b_0,\ldots,b_{k-1})$ be the shadow $k$-binomial decomposition sequence for $S$. 
	
	Then any family that induces a subhypergraph  of $H(S)$ (over the same vertex set), has a shadow $k$-binomial decomposition $(b_0',\ldots,b_{k-1}')$ with
	$b_i'\geq b_i$ for each $i\in[0,k-1]$.
	
	In particular, if $S$ is extremal, any family induced by a subhypergraph of $H(S)$ is also extremal and contains $S$ as a subfamily.
	%
	%
	%
\end{proposition}

In \cite[Example~2.4]{furgri86}, the authors argue that, whenever the extremal family has a cardinality with $a_{k-1}>1$ (here $a_{k-1}$ correspond to the last element of full k-binomial decomposition), then we can remove an element of the family and keep the family extremal.
The situation is slightly differerent when we want to add an element $e$ to the extremal family. When in the $k$-binomial decomposition of the size of the family both $a_{k-1}$ and $a_{k-2}$ exists (so $a_{k-1}<a_{k-2}$), then the addition of an element should not add any set in its shadow. In the hypergraph of the family, this translates to the fact that $e$ is a $k$-edge of the hypergraph. Indeed, if $e\notin S$, then there exists an edge $e_i\in E(H)$ (perhaps more) such that $e_i\subseteq e$; now, if $e_i=e$, then the family is still extremal (we can see it from two perspectives: a subhypergraph of an extremal family is extremal by Proposition~\ref{thm:families}, or by the definition of the hypergraph, so if $e\in E(H)$, then all the elements in $\Delta(e)$ belong to $\Delta(S)$ or do not belong to the complement). However, if $e_i\subsetneq e$, then the elements $b$ of size $k-1$ with $e_i\subseteq b \subsetneq e$ would belong to $\Delta(S\cup e)$ yet they did not belong to $\Delta(S)$, and thus the set is not extremal.
In particular, this argument shows the following two results:

\begin{theorem}\label{t.other_families3}
	Let $S$ be a family of $k$-sets (not necessarily extremal).
	If the hypergraph of $S$, $H(S)$, does not contain an edge of size $k$, then
	\begin{center}
		either \qquad 
		$S$ is the colex 
		\qquad	or, if that is not the case, \qquad
		for each $s\in S^c$, $S\cup s$ is not extremal.
	\end{center}
\end{theorem}

If we use Theorem~\ref{thm:card} in order to refine Theorem~\ref{t.other_families3} we conclude the following.

\begin{theorem}
	\label{t.inest}
	Let $H$ be the hypergraph associated to the extremal family $S\subseteq \binom{[n]}{k}$  with support $P$.
	Let  $s\in \binom{P}{k}\setminus S$. Then the following holds.
	Let $e\in H$ be an edge of the hypergraph such that $e\subseteq s$. 
	
	If $|e|=k$, then $e$ is the unique edge in $H$ such that $e\subseteq s$, and $S\cup s$ is extremal.
	
	If $|e|<k$, then $S\cup s$ is extremal if and only if 
	\begin{itemize}
		\item $S$ is the initial colex segment, and $S\cup s$ is the initial colex segment,
		\item and $e$ has the maximal size among all the edges in $H(S)$. In particular, $H(S)$ has no edges of size $k$,
		\item and $e$ is the unique edge of $H(S)$ contained in $s$.
	\end{itemize} 
\end{theorem}

\noindent Let us now give a reinterpretation of several results in the context of adding or substracting elements to an extremal family.

\begin{theorem} \label{thm:extending_families}
	Let $S$ be an extremal family and let $a_{k-1}$ denote, if it exists, the last binomial numerator of the $k$-binomial decomposition of $|S|$.
	Then:
	\begin{enumerate}
	 \item \label{en:short1} $S$ is maximal with respect to the inclusion (any set being added to $S$ makes $S$ not extremal) if and only if
	the hypergraph of $S$ does not have any edge of size $k$ and is not the initial segment in the colex order.
	\item \label{en:short2} $S$ is minimal with respect to the inclusion (any set being removed from $S$ makes $S$ not extremal) if and only if $a_{k-1}=1$ and $S$ cannot be seen as the initial segment of the colex order with one set added.
	\item \label{en:short3} There is a sequence of the elements in $\binom{[n]}{k}\setminus S$, $\{s_1,\ldots,s_m\}$ for which all $\{S\cup \{s_1,\ldots,s_i\}\}_{i\in[m]}$ are extremal if and only if $\Delta(S)$ is the initial segment in the colex order (equivalently, if $H^{(k-1)}$ is the hypergraph of the initial segment in the colex order).
	\item \label{en:short4} There is a sequence of the elements in $S$, $\{s_1,\ldots,s_m\}$ for which all $\{S\setminus \{s_1,\ldots,s_i\}\}_{i\in[m]}$ are extremal if and only if $a_{k-1}\geq 1$ and $S$ is obtained from an initial segment of the colex order by adding $a_{k-1}$ edges.
	\end{enumerate}
\end{theorem}

\begin{proof}
	\ref{en:short1} and \ref{en:short3} follow from Theorem~\ref{t.inest}, the latter by observing that, after removing all the edges of size $k$ from the hypergraph, then either the family is the initial segment in the colex order, or will stop being extremal.
	
	Parts \ref{en:short2} and \ref{en:short4} follow similarly but from Theorem~\ref{thm:card} in combination with the fact that removing an edge removes one unit from $a_{k-1}$, the last element in the the binomial decomposition of $|S|$
\end{proof}

In Section~\ref{sec:maximal_chains} there is a further discussion on maximal chains of extremal families in the poset of families ordered by inclusion.

\subsection{Any family can be an induced family of an extremal family}

Theorem~\ref{thm:solid_part} follows from the arguments of Proposition~\ref{prop.translation} and Theorem~\ref{thm:charac}. Indeed, observe that the edges of the hypergraph of the family $S'$ from \eqref{eq:family_ext} are the same as those for $S$, yet the hypergraph of $S$ has $[n]$ as its vertex set, while the hypergraph of $S'$ has $[n+r]$. Now, the arguments of Proposition~\ref{prop.translation} show that, if the shadow $k$-binomial decomposition of $S$ is $(b_0,\ldots,b_{k-1})$, then the shadow $k$-binomial decomposition of $S'$ is equal to $(b_0+r,\ldots,b_{k-1}+r)$. Therefore, for $r$ large enough $b_{k-1}+r\geq 1$ and then Theorem~\ref{thm:charac} gives that $S'(r)$ is extremal, as the statement of Theorem~\ref{thm:solid_part} claims.

\subsection{Depth of a family}


The depth of the family is the minimal index $j$ for which $\Delta^j(S)$ is the initial segment in the colex order. Since $\Delta^{k-1}(S)$ is a family of singletons, it is the initial segment of the colex order, and the shadow of the initial segment in the colex order is again the initial segment of the colex order, such $j$ is well defined. Thus, 
in a family $S$ of depth $j$, $\Delta^j(S)$ is the initial segment in the colex order, while $\Delta^{j-1}(S)$ is not (we understand that $\Delta^{-1}(S)$ does not exists when $j=0$), and in any supercomfortable ordering the order of the edge $e_{\text{sc}(E(H))+1}$ is $k-j+1$ (here we are considering that the edge $e_{\text{sc}(E(H))+1}$ does not exists if $S$ is the initial segment in the colex order); further, if the family has $H$ as its hypergraph, then $H^{(k-j)}$ is the hypergraph of an initial segment in the colex order, while $H^{(k-j+1)}$ is not (here we are assuming that either $j>0$ or, if that is not the case, we understand that the hypergraph $H^{(k+1)}$ does not exists).

\begin{proposition}[Depth and hypotenusal numbers]\label{p.shad_hypo}
	Let $S\subseteq \binom{[n]}{k}$ be a family of $k$-sets with depth $j\geq 1$, and let
	$\binom{b_0}{k}+\cdots+\binom{b_{k-1}}{1}=|S|$
	be the shadow $k$-binomial decomposition of $S$.
	Then we have
	\[
	\begin{cases}
	b_{k-j+i}-b_{k-1-j+i}\geq a[i]+1 &\text{ if } 0\leq i <j-1 \\
	b_{k-j+i}-b_{k-1-j+i}\geq a[i] &\text{ if } i=j-1
	\end{cases}
	\]
	or
	\[
	\begin{cases}
	b_{k-j+i}-b_{k-1-j+i}\geq a[i-1]+1 &\text{ if } 0\leq i <j-1 \\
	b_{k-j+i}-b_{k-1-j+i}\geq a[i-1] &\text{ if } i=j-1
	\end{cases}
	\]
	where $a[i]$ are the hypothenusal numbers, and $a[-1]=1$. There are examples where the bounds are tight.
\end{proposition}

That is, from the moment where the shadows $\Delta^{k-1}(S),\ldots,\Delta^{k-i}(S),\ldots,\Delta^1(S),S$ stop being the initial segment in the colexicographical order, then the coefficients of the $k$-binomial decomposition decrease by, at least, their corresponding hypotenusal number.

\begin{proof}[Proof of Proposition~\ref{p.shad_hypo}]
Since the family of sets has depth $j\geq 1$, $S$ is not the colex, $H^{(k-j)}$ is the hypergraph of an initial segment in the colex order, while $H^{(k-j+1)}$ is not. Let $t=\text{sc}(E(H))$, thus $\{e_1,\ldots,e_t\}$ induces the initial segment of the colex order. Then $e_t$ has size $\leq k-j+1$ and the size of $e_{t+1}$ is $=k-j+1$. Using Claim~\ref{cl:edges_in_non_colex} and Observation~\ref{obs:coeff} in the argument for the proof of Theorem~\ref{thm:card} in Section~\ref{sec:alt_to_mtfg}, we conclude that
the tree $T_{e_{t+1}}$ of $B_{t+1}$ induces either a binomial coefficient
of the type $\binom{n-|e_{t+1}|-|s|}{k-|e_{t+1}|}$ with $|s|<t$,
or at least two binomial coefficients $\binom{n-|e_{t+1}|-(t)}{k-|e_{t+1}|}$.
In the first case, in the bin-ball-wall process given using the translation given by Observation~\ref{obs:dict_bbw-sbd}, the ball induced by $\binom{n-|e_{t+1}|-|s|}{k-|e_{t+1}|}$ has some descendants (as $|s|<t$), while the wall is to the right of the wall in the first iteration of the configuration giving the hypothenusal numbers Lemma~\ref{lem:hip_num}.
In particular, and since the process does not end abruptly by Proposition~\ref{prop.translation}\ref{en:prop_trans_1} we can first process the balls descendants from $\binom{n-|e_{t+1}|-|s|}{k-|e_{t+1}|}$, and keep special track of those descendants that would appear if the wall would have been at the position given by the hypothenusal numbers. In particular, we observe that there are always at least as many balls as those in Lemma~\ref{lem:hip_num} as the wall in the corresponding process is consistently  to the right of position of the wall in the Lemma~\ref{lem:hip_num}; thus the claim with $a[i]$ follows in this case.
When there are the terms $\binom{n-|e_{t+1}|-(t)}{k-|e_{t+1}|}$ and 
$\binom{n-(|e_{t+1}|+1)-(t)}{k-(|e_{t+1}|+1)}$, the first term forces the wall to drop by a unit, while the second term gives the hypothenusal numbers using the same argument as before from the following iteration onwards, thus we obtain the instance $a[i-1]$ instead of $a[i]$.

The bounds are achieved with equality with the family given by the set of hyperedges $\{(1,2),(3,4,5)\}$ for the second case, and by the $\{(1,2),(1,3),(2,3,4)\}$ in the first case.
\end{proof}

Now Theorem~\ref{thm:hn} follows from Proposition~\ref{p.shad_hypo} using the example of the hypergraph $\{(1,2),(3,4,5)\}$ for any $k$ while $[n]$ being large enough.

Combining the following statements: Proposition~\ref{p.shad_hypo}, and the lower bound \eqref{eq.hyp_all} (which gives a rate of descend on the binomial coefficients of the shadow $k$-binomial decompositions), and the characterization Theorem~\ref{thm:charac} (that claims that $b_{k-1}\geq 1$ if $S$ is to be extremal), and Proposition~\ref{prop.translation}\ref{en:prop_trans_2}, we obtain the following:

\begin{theorem}[Bound on the depth]\label{thm:bound_depth}
	Let $S\subseteq \binom{[n]}{k}$ a family of $k$-sets. Assume that $S$ is extremal and has depth $j$. 
	Then, $j\leq \max\{\log_2(\log_2(n))+4,5\}$.
\end{theorem}

Theorem~\ref{thm:depth} follows immediately from Theorem~\ref{thm:bound_depth}.

Theorem~\ref{thm:depth3} follows by observing that if $\Delta^4(S)$ is not extremal, then there exists a ball $b$ with descendants for at least three iterations of the bin-balls-wall process. In the first iteration $i$ after the ball creates a descendant (or is not deleted by the wall), if the separation of $a_i-a_{i-1}$ is linear in $n$, it means that such ball creates a linear number of descendants (indeed, process that ball the last one on that iteration), hence it becomes a quadratic number of balls in the next iteration, and thus it reflects in the fact that the third iteration needs a separation which should be quadratic in $a_i-a_{i-1}$ (just by the balls that are descendats from $b$); since the  proportion of binomial decompositions such that $a_{k-1}-a_{k-2}$ is quadratic in $a_{k-3}-a_{k-4}$ (and $a_{k-2}-a_{k-3}$ is linear in $a_{k-3}-a_{k-4}$) goes to zero as $n\to \infty$, Theorem~\ref{thm:depth3} follows.


\section{Extremal families with prescribed depth} \label{sec:ext_fam_prescribed_depth}

Theorem~\ref{thm:card} shows that, for certain binomial decompositions, the only extremal family is the initial segment in the colex order. In the previous section we have examined the role of the depth in placing restrictions on consecutive coefficients of the shadow $k$-binomial Proposition~\ref{p.shad_hypo}, and some of its consequences as Theorem~\ref{thm:bound_depth}. The following question then becomes natural:
\begin{quote}
If $S\subset\binom{[n]}{k}$ has depth $j$, and is extremal, which cardinalities can it have?
\end{quote}
We solve the decision problem in Theorem~\ref{thm:decision}, which implies Theorem~\ref{thm:alg}, in the following way:
\begin{quote}
	If $S$ is an extremal family in $\binom{[n]}{k}$, of size $m$ and with depth $j$, then either Construction~A or Construction~B (which follow the corresponding restrictions on the edges given in Claim~\ref{cl:edges_in_non_colex} and are given in Section~\ref{sec:cons_a_b}) produce an extremal family with the same cardinality $m$ and the same depth $j$. Both of these constructions can be found using a reasonably fast algorithm.
	\end{quote}
In particular, if we are given a $n$, $k$ and $j$ and a $k$-binomial decomposition $\a$, we can decide the existence of an extremal family by trying  Construction~A and Construction~B; if both fail to give an extremal family, then no extremal family with those characteristics exists.

\subsection{Constructions generalizing the colex} \label{sec:cons_a_b}

The constructions $A/A'$, and $B/B'$ below are given using the hypergraphs of the families and are inspired by the dichotomy regarding the hypergraph given in Claim~\ref{cl:edges_in_non_colex}.

Recall that $H_C(n_1,\ldots,n_k)$, as introduced in Section~\ref{sec:hypergraph_of_colex}, denotes the hypergraph of the initial segment in the colex order with $n_i$ edges of size $i$; in particular, $n_1+\cdots+ n_t\leq n-t+1$ for each $t\leq k$; we extend this notation so that $H(n_1,\ldots,n_k)$ signifies a graph with $n_i$ edges of size $i$.
In all the cases below we use a comfortable or supercomformatable ordering on the edges of the hypergraphs, so $e_i\leq e_j\iff |e_i|\leq |e_j|$. Since we want these graphs to give families with prescribed depth $j$, thus $H^{(k-j+1)}$ is the first hypergraph not being the colex, we should have at least one edge of size $k-j+1$, thus $n_{k-j+1}\geq 1$.

Throughout this section, we shall use depth $k-j$ instead of $j$, when focussing on the hypergraph constructions.

\paragraph{Construction $B$ and $B'$.} 
Consider the notation of the vertices from Section~\ref{sec:hypergraph_of_colex} for $H_C(n_1,\ldots,n_k)$. Assume that $n_{j+1}\geq 1$ and that $n_1+\ldots+n_{j+1}\geq 2$,\footnote{We want $H_{B,j}^{(j)}(n_1,\ldots,n_k)$ to be the hypergraph of the initial segment of the colex order, while $H_{B,j}^{(j+1)}(n_1,\ldots,n_k)$ is not; equivalently, that the family has depth $k-j$. This implies that the assumption $n_{j+1}\geq 1$ and $n_1+\ldots+n_{j+1}\geq 2$ is adecuate, for otherwise the family cannot have such depth.} this implies that on the subhypergraph $H_C^{(j+1)}(n_1,\ldots,n_k)$ with edges $\{e_1,\ldots,e_{n_1+n_2+\ldots+n_{j+1}}\}$, there exists a vertex $v_s$ with $1\leq s\leq j$ with degree $\geq 2$: let $t$ be the largest index for a vertex $v_t$ with such property.

The hypergraph of the \emph{Construction B}, $H_{B,j}(n_1,\ldots,n_k)$, is then obtained from $H_C(n_1,\ldots,n_k)$ by
exchanging the vertex $v_{t}$ by $v_{j+1}$ in the edge $e_{n_1+n_2+\ldots+n_{j+1}}$.
The edges of $H_{B,j}$ are supercomfortably ordered and $H_{B,j}^{(j+1)}(n_1,\ldots,n_t)$ satisfies the condition \ref{exB} in Claim~\ref{cl:edges_in_non_colex}.

Given $r+1\in [n_1+\ldots+n_{j}+1,n_1+\ldots+n_{j+1}]$, and  $t$ the largest index of a vertex $v_t$ that has degree $\geq 2$ in the subhypergraph of $H_C(n_1,\ldots,n_k)$ induced by the edges $\{e_1,\ldots,e_r,e_{r+1}\}$ (similar as before). Then the hypergraph of the \emph{Construction $B'$}, $H_{B',j,r}\allowbreak (n_1,\ldots,\allowbreak n_k)$, is obtained from  $H_C(n_1,\ldots,n_k)$ by
exchanging the vertex $v_{t}$ by $v_{j+1}$ in the edge $e_{r+1}$.
The edges of $H_{B',j,r}\allowbreak (n_1,\ldots,\allowbreak n_k)$ are comfortably (but, if $r+1\neq n_1+\ldots+n_{j+1}$ not supercomfortably) ordered. Upon reordering the edges of size $j+1$ so that $e_{r+1}$ becomes the last one of its size, then $H_{B',j,r}^{(j+1)}\allowbreak (n_1,\ldots,\allowbreak n_k)$ satisfies \ref{exB} in Claim~\ref{cl:edges_in_non_colex}. When $r+1= n_1+\ldots+n_{j+1}$ Construction $B$ and Construction $B'$ coincide.

The families induced by the constructions $H_{B,j}(n_1,\ldots,n_k)$ and $H_{B',j,r}\allowbreak (n_1,\ldots,\allowbreak n_k)$ have depth $k-j$.

\paragraph{Construction $A$ and $A'$.}
As before, let $t$ the largest index of a vertex $v_t$ with degree $\geq 2$ in the subhypergraph of $H_C(n_1,\ldots,n_k)$ induced by $\{e_1,\ldots,e_{n_1+\ldots+n_{j+1}}\}$.

The hypergraph of the \emph{Construction $A$}, $H_{A,j}(n_1,\ldots,n_k)$,  is obtained from $H_C(n_1,\ldots,n_k)$ by, in the edge $e_{n_1+n_2+\ldots+n_{j+1}}$, exchanging the vertex $v_{t}$ by $u_{n_1+\ldots+n_{j+1}-1}$, and the vertex $u_{n_1+\ldots+n_{j+1}}$ by $v_{j+1}$.
$H_{A,j}^{(j+1)}(n_1,\ldots,n_t)$ satisfies the condition \ref{exA} from Claim~\ref{cl:edges_in_non_colex}.\footnote{The vertex $u_{n_1+\ldots+n_{j+1}-1}$ is the unique vertex covered by the edge $e_{n_1+\ldots+n_{j+1}-1}$ in the initial segment of the colex order, and this is not an edge of size $1$, for otherwise $H_{A,j}^{(j+1)}$ would contain only one edge of size strictly larger than one, which would mean it is the initial segment in the colex order. Observe also that $e_{n_1+\ldots+n_{j+1}-1}$ has all the vertices of degree $\geq 2$ in $H_{A,j}(n_1,\ldots ,n_k)$.}

Given $r+1\in [n_1+\ldots+n_{j}+1,n_1+\ldots+n_{j+1}]$, and  $t$ the largest index of a vertex $v_t$ that has degree $\geq 2$ in the subgraph of $H_C(n_1,\ldots,n_k)$ induced by the edges $\{e_1,\ldots,e_r,e_{r+1}\}$ (similar as before). Then
the hypergraph of the family for the \emph{Construction $A'$}, $H_{A',j,r}\allowbreak(n_1, \ldots,\allowbreak n_k)$, is obtained from $H_C(n_1,\ldots,n_k)$ by exchanging the vertex $v_{t}$ by $u_{r}$, and the vertex $u_{r+1}$ by $v_{j+1}$ in the edge $e_{r+1}$.
The edges of $H_{B',j,r}\allowbreak (n_1,\ldots,\allowbreak n_k)$ are comfortably (but, if $r+1\neq n_1+\ldots+n_{j+1}$ not supercomfortably) ordered. Upon reordering the edges of size $j+1$ so that $e_{r+1}$ becomes the last one, then $H_{A',j,r}^{(j+1)}\allowbreak (n_1,\ldots,\allowbreak n_k)$ satisfies \ref{exA} in Claim~\ref{cl:edges_in_non_colex}. When $r+1= n_1+\ldots+n_{j+1}$, Construction $A$ and Construction $A'$ coincide.

The families induced by the constructions $H_{A,j}(n_1,\ldots,n_k)$ and $H_{A',j,r}\allowbreak (n_1,\ldots,n_k)$ have depth $k-j$.

See Figure~\ref{fig:hyp_example} for an example of Construction $A$ and Construction $B$.
The families $H_{A',\cdot}$ and $H_{B',\cdot}$ are generalizations of the $A$ and $B$ that allows to partition the argument of Theorem~\ref{thm:decision} in 
two smaller steps.

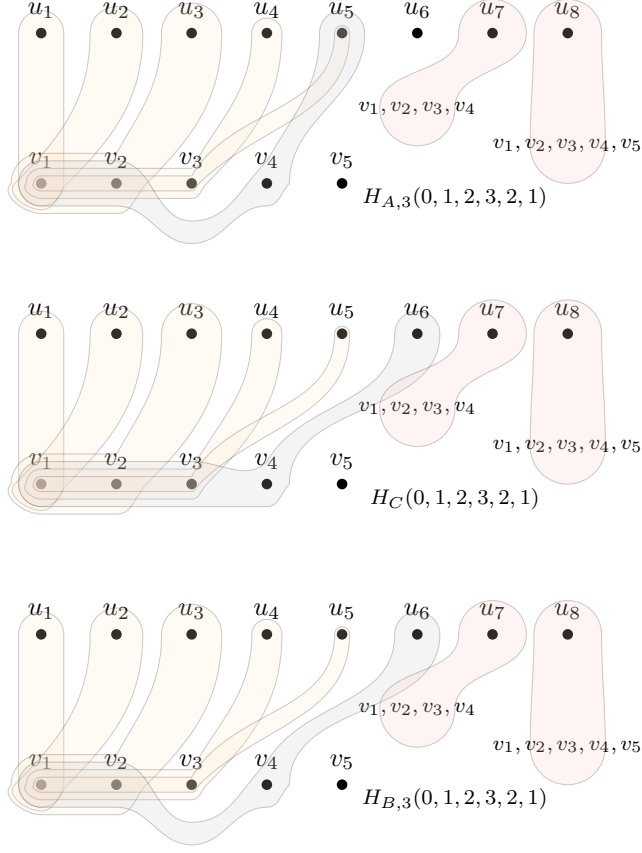
\begin{figure}[htb]
	\centering
\begin{tikzpicture}
\tikzstyle{vertex}=[circle,fill=black!100,minimum size=7pt,inner sep=0pt]
\tikzstyle{vertex2}=[circle,fill=black!100,minimum size=4pt,inner sep=0pt]

\node[vertex2,label={above:$u_1$}] (u11) at (-10,4) {};
\node[vertex2,label={above:$u_2$}] (u21) at (-9,4) {};
\node[vertex2,label={above:$u_3$}] (u31) at (-8,4) {};
\node[vertex2,label={above:$u_4$}] (u41) at (-7,4) {};
\node[vertex2,label={above:$u_5$}] (u51) at (-6,4) {};
\node[vertex2,label={above:$u_6$}] (u61) at (-5,4) {};
\node[vertex2,label={above:$u_7$}] (u71) at (-4,4) {};
\node[vertex2,label={above:$u_8$}] (u81) at (-3,4) {};

\node[] (e11) at (-5,3) {{\footnotesize $v_1,v_2,v_3,v_4$}};

\node[] (e21) at (-3,2.5) {{\footnotesize $v_1,v_2,v_3,v_4,v_5$}};

\node[] (n1) at (-4.5,1.8) {{\footnotesize $H_{A,3}(0,1,2,3,2,1)$}};

\node[vertex2,label={above:$v_1$}] (v11) at (-10,2) {};
\node[vertex2,label={above:$v_2$}] (v21) at (-9,2) {};
\node[vertex2,label={above:$v_3$}] (v31) at (-8,2) {};
\node[vertex2,label={above:$v_4$}] (v41) at (-7,2) {};
\node[vertex2,label={above:$v_5$}] (v51) at (-6,2) {};


\filldraw[fill=orange!20!white,opacity=0.2] ($(u11)+(-0.3,0)$) 
to[out=270,in=90] ($(v11) + (-0.3,0)$) 
to[out=270,in=180] ($(v11) + (0,-0.3)$)
to[out=0,in=270] ($(v11) + (0.3,0)$)
to[out=90,in=270] ($(u11) + (0.3,0)$)
to[out=90,in=0] ($(u11) + (0,0.3)$)
to[out=180,in=90] ($(u11)+(-0.3,0)$);

\filldraw[fill=orange!20!white,opacity=0.2] ($(u21)+(-0.35,0)$) 
to[out=270,in=45] ($(v11) + (-0.35,0)$) 
to[out=225,in=180] ($(v11) + (0,-0.35)$)
to[out=0,in=225] ($(v11) + (0.35,0)$)
to[out=45,in=270] ($(u21) + (0.35,0)$)
to[out=90,in=0] ($(u21) + (0,0.35)$)
to[out=180,in=90] ($(u21)+(-0.35,0)$);

\filldraw[fill=orange!20!white,opacity=0.2] ($(u31)+(-0.4,0)$) 
to[out=270,in=60] ($(v21) + (0,0.4)$) 
to[out=180,in=0] ($(v11) + (0,0.4)$)
to[out=180,in=90] ($(v11) + (-0.4,0)$) 
to[out=225,in=180] ($(v11) + (0,-0.4)$)
to[out=0,in=180] ($(v21) + (0,-0.4)$)
to[out=0,in=225] ($(v21) + (0.4,0)$)
to[out=45,in=270] ($(u31) + (0.4,0)$)
to[out=90,in=0] ($(u31) + (0,0.4)$)
to[out=180,in=90] ($(u31)+(-0.4,0)$);

\filldraw[fill=orange!20!white,opacity=0.2] ($(u41)+(-0.2,0)$) 
to[out=270,in=60] ($(v31) + (0,0.2)$) 
to[out=180,in=0] ($(v11) + (0,0.2)$)
to[out=180,in=90] ($(v11) + (-0.2,0)$) 
to[out=225,in=180] ($(v11) + (0,-0.2)$)
to[out=0,in=180] ($(v31) + (0,-0.2)$)
to[out=0,in=225] ($(v31) + (0.2,0)$)
to[out=45,in=270] ($(u41) + (0.2,0)$)
to[out=90,in=0] ($(u41) + (0,0.2)$)
to[out=180,in=90] ($(u41)+(-0.2,0)$);

\filldraw[fill=orange!20!white,opacity=0.2] ($(u51)+(-0.1,0)$) 
to[out=270,in=60] ($(v31) + (0,0.1)$) 
to[out=180,in=0] ($(v11) + (0,0.1)$)
to[out=180,in=90] ($(v11) + (-0.1,0)$) 
to[out=225,in=180] ($(v11) + (0,-0.1)$)
to[out=0,in=180] ($(v31) + (0,-0.1)$)
to[out=0,in=225] ($(v31) + (0.1,0)$)
to[out=45,in=270] ($(u51) + (0.1,0)$)
to[out=90,in=0] ($(u51) + (0,0.1)$)
to[out=180,in=90] ($(u51)+(-0.1,0)$);

\filldraw[fill=black!20!white,opacity=0.2] ($(u51)+(-0.3,0)$) 
to[out=270,in=60] ($(v41) + (0,0.3)$) 
to[out=225,in=0] ($(v31) + (0,-0.5)$) 
to[out=180,in=0] ($(v21) + (0,0.3)$)
to[out=180,in=0] ($(v11) + (0,0.3)$)
to[out=180,in=90] ($(v11) + (-0.3,0)$)
to[out=270,in=180] ($(v11) + (0,-0.3)$) 
to[out=0,in=180] ($(v21) + (0,-0.3)$)
to[out=0,in=180] ($(v31) + (0,-0.8)$)
to[out=0,in=180] ($(v41) + (0,-0.3)$) 
to[out=0,in=225] ($(v41) + (0.3,0)$) 
to[out=90,in=270] ($(u51) + (0.3,0)$)
to[out=90,in=0] ($(u51) + (0,0.3)$)
to[out=180,in=90] ($(u51)+(-0.3,0)$);

\filldraw[fill=red!20!white,opacity=0.2] ($(u71)+(-0.45,0)$) 
to[out=270,in=90] ($(e11) + (-0.5,0)$) 
to[out=270,in=180] ($(e11) + (0,-0.5)$) 
to[out=0,in=270] ($(e11) + (0.5,0)$) 
to[out=90,in=270] ($(u71) + (0.45,0)$) 
to[out=90,in=0] ($(u71) + (0,0.45)$)
to[out=180,in=90] ($(u71) + (-0.45,0)$);

\filldraw[fill=red!20!white,opacity=0.2] ($(u81)+(-0.45,0)$) 
to[out=270,in=90] ($(e21) + (-0.5,0)$) 
to[out=270,in=180] ($(e21) + (0,-0.5)$) 
to[out=0,in=270] ($(e21) + (0.5,0)$) 
to[out=90,in=270] ($(u81) + (0.45,0)$) 
to[out=90,in=0] ($(u81) + (0,0.45)$)
to[out=180,in=90] ($(u81) + (-0.45,0)$);


\node[vertex2,label={above:$u_1$}] (u12) at (-10,0) {};
\node[vertex2,label={above:$u_2$}] (u22) at (-9,0) {};
\node[vertex2,label={above:$u_3$}] (u32) at (-8,0) {};
\node[vertex2,label={above:$u_4$}] (u42) at (-7,0) {};
\node[vertex2,label={above:$u_5$}] (u52) at (-6,0) {};
\node[vertex2,label={above:$u_6$}] (u62) at (-5,0) {};
\node[vertex2,label={above:$u_7$}] (u72) at (-4,0) {};
\node[vertex2,label={above:$u_8$}] (u82) at (-3,0) {};

\node[] (e12) at (-5,-1) {{\footnotesize $v_1,v_2,v_3,v_4$}};

\node[] (e22) at (-3,-1.5) {{\footnotesize $v_1,v_2,v_3,v_4,v_5$}};

\node[] (n2) at (-4.5,-2.2) {{\footnotesize$H_{C}(0,1,2,3,2,1)$}};

\node[vertex2,label={above:$v_1$}] (v12) at (-10,-2) {};
\node[vertex2,label={above:$v_2$}] (v22) at (-9,-2) {};
\node[vertex2,label={above:$v_3$}] (v32) at (-8,-2) {};
\node[vertex2,label={above:$v_4$}] (v42) at (-7,-2) {};
\node[vertex2,label={above:$v_5$}] (v52) at (-6,-2) {};


\filldraw[fill=orange!20!white,opacity=0.2] ($(u12)+(-0.3,0)$) 
to[out=270,in=90] ($(v12) + (-0.3,0)$) 
to[out=270,in=180] ($(v12) + (0,-0.3)$)
to[out=0,in=270] ($(v12) + (0.3,0)$)
to[out=90,in=270] ($(u12) + (0.3,0)$)
to[out=90,in=0] ($(u12) + (0,0.3)$)
to[out=180,in=90] ($(u12)+(-0.3,0)$);

\filldraw[fill=orange!20!white,opacity=0.2] ($(u22)+(-0.35,0)$) 
to[out=270,in=45] ($(v12) + (-0.35,0)$) 
to[out=225,in=180] ($(v12) + (0,-0.35)$)
to[out=0,in=225] ($(v12) + (0.35,0)$)
to[out=45,in=270] ($(u22) + (0.35,0)$)
to[out=90,in=0] ($(u22) + (0,0.35)$)
to[out=180,in=90] ($(u22)+(-0.35,0)$);

\filldraw[fill=orange!20!white,opacity=0.2] ($(u32)+(-0.4,0)$) 
to[out=270,in=60] ($(v22) + (0,0.4)$) 
to[out=180,in=0] ($(v12) + (0,0.4)$)
to[out=180,in=90] ($(v12) + (-0.4,0)$) 
to[out=225,in=180] ($(v12) + (0,-0.4)$)
to[out=0,in=180] ($(v22) + (0,-0.4)$)
to[out=0,in=225] ($(v22) + (0.4,0)$)
to[out=45,in=270] ($(u32) + (0.4,0)$)
to[out=90,in=0] ($(u32) + (0,0.4)$)
to[out=180,in=90] ($(u32)+(-0.4,0)$);

\filldraw[fill=orange!20!white,opacity=0.2] ($(u42)+(-0.2,0)$) 
to[out=270,in=60] ($(v32) + (0,0.2)$) 
to[out=180,in=0] ($(v12) + (0,0.2)$)
to[out=180,in=90] ($(v12) + (-0.2,0)$) 
to[out=225,in=180] ($(v12) + (0,-0.2)$)
to[out=0,in=180] ($(v32) + (0,-0.2)$)
to[out=0,in=225] ($(v32) + (0.2,0)$)
to[out=45,in=270] ($(u42) + (0.2,0)$)
to[out=90,in=0] ($(u42) + (0,0.2)$)
to[out=180,in=90] ($(u42)+(-0.2,0)$);

\filldraw[fill=orange!20!white,opacity=0.2] ($(u52)+(-0.1,0)$) 
to[out=270,in=60] ($(v32) + (0,0.1)$) 
to[out=180,in=0] ($(v12) + (0,0.1)$)
to[out=180,in=90] ($(v12) + (-0.1,0)$) 
to[out=225,in=180] ($(v12) + (0,-0.1)$)
to[out=0,in=180] ($(v32) + (0,-0.1)$)
to[out=0,in=225] ($(v32) + (0.1,0)$)
to[out=45,in=270] ($(u52) + (0.1,0)$)
to[out=90,in=0] ($(u52) + (0,0.1)$)
to[out=180,in=90] ($(u52)+(-0.1,0)$);

\filldraw[fill=black!20!white,opacity=0.2] ($(u62)+(-0.3,0)$) 
to[out=270,in=60] ($(v42) + (0,0.3)$) 
to[out=225,in=0] ($(v32) + (0,0.3)$) 
to[out=180,in=0] ($(v22) + (0,0.3)$)
to[out=180,in=0] ($(v12) + (0,0.3)$)
to[out=180,in=90] ($(v12) + (-0.3,0)$)
to[out=270,in=180] ($(v12) + (0,-0.3)$) 
to[out=0,in=180] ($(v22) + (0,-0.3)$)
to[out=0,in=180] ($(v32) + (0,-0.3)$)
to[out=0,in=180] ($(v42) + (0,-0.3)$) 
to[out=0,in=225] ($(v42) + (0.3,0)$) 
to[out=90,in=270] ($(u62) + (0.3,0)$)
to[out=90,in=0] ($(u62) + (0,0.3)$)
to[out=180,in=90] ($(u62)+(-0.3,0)$);

\filldraw[fill=red!20!white,opacity=0.2] ($(u72)+(-0.45,0)$) 
to[out=270,in=90] ($(e12) + (-0.5,0)$) 
to[out=270,in=180] ($(e12) + (0,-0.5)$) 
to[out=0,in=270] ($(e12) + (0.5,0)$) 
to[out=90,in=270] ($(u72) + (0.45,0)$) 
to[out=90,in=0] ($(u72) + (0,0.45)$)
to[out=180,in=90] ($(u72) + (-0.45,0)$);

\filldraw[fill=red!20!white,opacity=0.2] ($(u82)+(-0.45,0)$) 
to[out=270,in=90] ($(e22) + (-0.5,0)$) 
to[out=270,in=180] ($(e22) + (0,-0.5)$) 
to[out=0,in=270] ($(e22) + (0.5,0)$) 
to[out=90,in=270] ($(u82) + (0.45,0)$) 
to[out=90,in=0] ($(u82) + (0,0.45)$)
to[out=180,in=90] ($(u82) + (-0.45,0)$);


\node[vertex2,label={above:$u_1$}] (u13) at (-10,-4) {};
\node[vertex2,label={above:$u_2$}] (u23) at (-9,-4) {};
\node[vertex2,label={above:$u_3$}] (u33) at (-8,-4) {};
\node[vertex2,label={above:$u_4$}] (u43) at (-7,-4) {};
\node[vertex2,label={above:$u_5$}] (u53) at (-6,-4) {};
\node[vertex2,label={above:$u_6$}] (u63) at (-5,-4) {};
\node[vertex2,label={above:$u_7$}] (u73) at (-4,-4) {};
\node[vertex2,label={above:$u_8$}] (u83) at (-3,-4) {};

\node[] (e13) at (-5,-5) {{\footnotesize $v_1,v_2,v_3,v_4$}};

\node[] (e23) at (-3,-5.5) {{\footnotesize $v_1,v_2,v_3,v_4,v_5$}};

\node[] (n3) at (-4.5,-6.2) {{\footnotesize$H_{B,3}(0,1,2,3,2,1)$}};

\node[vertex2,label={above:$v_1$}] (v13) at (-10,-6) {};
\node[vertex2,label={above:$v_2$}] (v23) at (-9,-6) {};
\node[vertex2,label={above:$v_3$}] (v33) at (-8,-6) {};
\node[vertex2,label={above:$v_4$}] (v43) at (-7,-6) {};
\node[vertex2,label={above:$v_5$}] (v53) at (-6,-6) {};


\filldraw[fill=orange!20!white,opacity=0.2] ($(u13)+(-0.3,0)$) 
to[out=270,in=90] ($(v13) + (-0.3,0)$) 
to[out=270,in=180] ($(v13) + (0,-0.3)$)
to[out=0,in=270] ($(v13) + (0.3,0)$)
to[out=90,in=270] ($(u13) + (0.3,0)$)
to[out=90,in=0] ($(u13) + (0,0.3)$)
to[out=180,in=90] ($(u13)+(-0.3,0)$);

\filldraw[fill=orange!20!white,opacity=0.2] ($(u23)+(-0.35,0)$) 
to[out=270,in=45] ($(v13) + (-0.35,0)$) 
to[out=225,in=180] ($(v13) + (0,-0.35)$)
to[out=0,in=225] ($(v13) + (0.35,0)$)
to[out=45,in=270] ($(u23) + (0.35,0)$)
to[out=90,in=0] ($(u23) + (0,0.35)$)
to[out=180,in=90] ($(u23)+(-0.35,0)$);

\filldraw[fill=orange!20!white,opacity=0.2] ($(u33)+(-0.4,0)$) 
to[out=270,in=60] ($(v23) + (0,0.4)$) 
to[out=180,in=0] ($(v13) + (0,0.4)$)
to[out=180,in=90] ($(v13) + (-0.4,0)$) 
to[out=225,in=180] ($(v13) + (0,-0.4)$)
to[out=0,in=180] ($(v23) + (0,-0.4)$)
to[out=0,in=225] ($(v23) + (0.4,0)$)
to[out=45,in=270] ($(u33) + (0.4,0)$)
to[out=90,in=0] ($(u33) + (0,0.4)$)
to[out=180,in=90] ($(u33)+(-0.4,0)$);

\filldraw[fill=orange!20!white,opacity=0.2] ($(u43)+(-0.2,0)$) 
to[out=270,in=60] ($(v33) + (0,0.2)$) 
to[out=180,in=0] ($(v13) + (0,0.2)$)
to[out=180,in=90] ($(v13) + (-0.2,0)$) 
to[out=225,in=180] ($(v13) + (0,-0.2)$)
to[out=0,in=180] ($(v33) + (0,-0.2)$)
to[out=0,in=225] ($(v33) + (0.2,0)$)
to[out=45,in=270] ($(u43) + (0.2,0)$)
to[out=90,in=0] ($(u43) + (0,0.2)$)
to[out=180,in=90] ($(u43)+(-0.2,0)$);

\filldraw[fill=orange!20!white,opacity=0.2] ($(u53)+(-0.1,0)$) 
to[out=270,in=60] ($(v33) + (0,0.1)$) 
to[out=180,in=0] ($(v13) + (0,0.1)$)
to[out=180,in=90] ($(v13) + (-0.1,0)$) 
to[out=225,in=180] ($(v13) + (0,-0.1)$)
to[out=0,in=180] ($(v33) + (0,-0.1)$)
to[out=0,in=225] ($(v33) + (0.1,0)$)
to[out=45,in=270] ($(u53) + (0.1,0)$)
to[out=90,in=0] ($(u53) + (0,0.1)$)
to[out=180,in=90] ($(u53)+(-0.1,0)$);

\filldraw[fill=black!20!white,opacity=0.2] ($(u63)+(-0.3,0)$) 
to[out=270,in=60] ($(v43) + (0,0.3)$) 
to[out=225,in=0] ($(v33) + (0,-0.5)$) 
to[out=180,in=0] ($(v23) + (0,0.3)$)
to[out=180,in=0] ($(v13) + (0,0.3)$)
to[out=180,in=90] ($(v13) + (-0.3,0)$)
to[out=270,in=180] ($(v13) + (0,-0.3)$) 
to[out=0,in=180] ($(v23) + (0,-0.3)$)
to[out=0,in=180] ($(v33) + (0,-0.8)$)
to[out=0,in=180] ($(v43) + (0,-0.3)$) 
to[out=0,in=225] ($(v43) + (0.3,0)$) 
to[out=90,in=270] ($(u63) + (0.3,0)$)
to[out=90,in=0] ($(u63) + (0,0.3)$)
to[out=180,in=90] ($(u63)+(-0.3,0)$);


\filldraw[fill=red!20!white,opacity=0.2] ($(u73)+(-0.45,0)$) 
to[out=270,in=90] ($(e13) + (-0.5,0)$) 
to[out=270,in=180] ($(e13) + (0,-0.5)$) 
to[out=0,in=270] ($(e13) + (0.5,0)$) 
to[out=90,in=270] ($(u73) + (0.45,0)$) 
to[out=90,in=0] ($(u73) + (0,0.45)$)
to[out=180,in=90] ($(u73) + (-0.45,0)$);

\filldraw[fill=red!20!white,opacity=0.2] ($(u83)+(-0.45,0)$) 
to[out=270,in=90] ($(e23) + (-0.5,0)$) 
to[out=270,in=180] ($(e23) + (0,-0.5)$) 
to[out=0,in=270] ($(e23) + (0.5,0)$) 
to[out=90,in=270] ($(u83) + (0.45,0)$) 
to[out=90,in=0] ($(u83) + (0,0.45)$)
to[out=180,in=90] ($(u83) + (-0.45,0)$);

\end{tikzpicture}
\caption{Example of \texorpdfstring{$H_{A,j}$}{HAj},   \texorpdfstring{$H_{C}$}{HC}, and \texorpdfstring{$H_{B,j}$}{HBj}.}
\label{fig.2}
\end{figure}

\subsection{Decision problem: either Construction~\texorpdfstring{$A$}{A} or \texorpdfstring{$B$}{B} are closer to the colex} \label{sec:extremal_a_b}

Let us now prove the result that implies Theorem~\ref{thm:alg}.

\begin{theorem}[Decision problem on extremality given depth and cardinality] \label{thm:decision}
	Let $S\subset \binom{[n]}{k}$ be an extremal family with depth $(k-j)\geq 0$, then:
	
	\noindent either 
	\begin{quote}there exists a vector of non-negative integers $(n_1,\ldots,n_k)$ for which $H_{A,j}(n_1,\ldots,n_k)$ has the same shadow $k$-binomial decomposition as $S$ (is extremal and has the same cardinality as $S$),
		\end{quote} or 
	\begin{quote}there exists a vector of non-negative integers $(n_1',\ldots,n_k')$ for which $H_{B,j}(n_1',\ldots,n_k')$ has the same shadow $k$-binomial decomposition as $S$ (is extremal and has the same cardinality as $S$).
		\end{quote}
	Moreover, the existence of $(n_1,\ldots,n_k)$ or $(n_1',\ldots,n_k')$ can be determined with an algorithm with running time $O(nk)$; the running time considers that the family $S$ is being given as a hypergraph, and the hypergraph of a family is considered a valid output as an answer.
\end{theorem}

\paragraph{Sketch of the proof.}
If $S$ is an extremal family of depth $(k-j)\geq 0$ with hypergraph $H$, then $H^{(j)}$ is the hypergraph of an initial segment of the colex order, while $H^{(j+1)}$ is not. In particular, when incrementally adding the edges of size $j+1$ to $H^{(j)}$, there is a point in which the give subhypergraph is not the hypergraph of an initial segment in the colex order. In fact, we may order the edges of $H$ supercomfortably so that the edge of size $j+1$ when this occurs has the largest possible index; let $t+1$ be such index (so $t=\text{sc}(E(H))$. Now, the edge $e_{t+1}$ faces the dichothomy from the cases \ref{exA} and \ref{exB} from Claim~\ref{cl:edges_in_non_colex} above.

In the case that $H$ behaves like \ref{exA} (respectively \ref{exB}), we show the existence of a non-negative sequence $(n_1,\ldots,n_k)$ (resp. $(n_1,\ldots,n_k)$)  for which the construction $H_{A',j,t}(n_1,\ldots,\allowbreak n_k)$ (respectively  $H_{B',j,t}(n_1,\ldots,\allowbreak n_k)$) gives an extremal family of $k$-sets with the same cardinality as $S$.
Finally, we show that if $H_{A',j,t}(n_1,\ldots,\allowbreak n_k)$ (respectively $H_{B',j,t}(n_1,\ldots,n_k)$) gives an extremal family, then there exists another sequence of non-negative integers $(n_1',\ldots,n_k')$ (resp. $(n_1',\ldots,n_k')$) for which $H_{A,j}(n_1',\ldots,n_k')$ (respectively $H_{B,j}(n_1',\ldots,n_k')$) exists, are extremal and have the same cardinality.

\subsubsection{If \texorpdfstring{$S$}{S} satisfies \texorpdfstring{\ref{exA}}{31}, then Construction \texorpdfstring{$A$}{A} from \texorpdfstring{Section~\ref{sec:cons_a_b}}{Section 3} is \texorpdfstring{\emph{better}}{better}}

\begin{proposition}
	 \label{prop:conap_best}
	If $S$ is an extremal family of depth $k-j$ with hypergraph $H$ that satisfies \ref{exA} for $H^{(j+1)}$ (actually for $\{e_1,\ldots,e_{t+1}\}$, with $t=\text{st}(E(H))$), then there exists a sequence $(n_1',\ldots,n_{k}')$ of non-negative integers such that $H_{A',j,t}(n_1',\ldots,n_{k}')$ is extremal and has the same $k$-binomial decomposition as $S$.
\end{proposition}

\begin{proof}
	Let $(n_1,\ldots,n_{k})$ be the sequences for the number of edges of each size of $H$.
	Let $e\in\{e_1,\ldots,e_t\}$ be the edge all whose vertices have degree $\geq 2$ in $\{e_1,\ldots,e_t,e_{t+1}\}$. In particular, $e_{t+1}$ covers all the vertices of degree $1$ of $e$ in $\{e_1,\ldots,e_t\}$ and, since no edge in $H$ is completely inside another, then there exists a vertex $w$ of $e$ of degree $d\geq 2$ in $\{e_1,\ldots,e_t\}$ that $e_{t+1}$ does not cover. As $w\notin e_{t+1}$, then $w$ also has degree $d$  in $\{e_1,\ldots,e_t,e_{t+1}\}$, and $B_{t+1}$ contains a dominating set of size $\leq 1+(t-d)$ (greedily expand $w$ with vertices in $V(\{e_1,\ldots,e_t\})\setminus V(e_{t+1})$; since each additional vertex dominates at least an additional edge, and there are $t-d$ edges in $\{e_1,\ldots,e_t\}$ not yet dominated by $w$, we need at most $\leq 1+(t-d)$ vertices in total; by considering a minimal dominating set within it, the claim follows.)
	
	 Since $\{e_1,\ldots,e_{t}\}$ form the initial segment of the colex order,
	 we use the notation given in Section~\ref{sec:hypergraph_of_colex} and conclude that the vertex with minimum degree, from those of degree $\geq 2$ in $\{e_1,\ldots,e_t\}$ is $v_{s-1}$, where $s$ is the order of the edge $e_{t-1}$. Let $d'$ denote such minimum degree, thus we have $d'\leq d$.
	 
	\begin{claim} \label{cl:pos_balls}
		For each of the edges $e_i$, $i\in[t+1,n_1+\cdots+n_{j+1}+n_{j+2}+\ldots+n_k]$, each dominating set in $B_i$ have size $\leq i-1-1$. In particular, the balls induced by $B_i$ are located at positions  $\leq i-1-1$.
	\end{claim}
	\begin{proof}[Proof of Claim~\ref{cl:pos_balls}]
		 As we should dominate $i-1$ edges, but when dominating the edge $e$ we are dominating at least an additional edge, the first part of the claim follows. The second is obtained from the first using Observation~\ref{obs:dict_bbw-sbd}.
	\end{proof}

Now consider the bins-balls-wall configuration from the hypergraph $H$. The balls given by the edges $\{e_1,\ldots,e_t\}$, when processed first leave no descendant and the wall at position $t$ (as the edges $\{e_1,\ldots,e_t\}$ constitute the initial segment of the colex order). Next, let us focus on the balls placed by the other edges $e_i$, $i\in[t+1,n_1+\cdots+n_k]$ using Claim~\ref{cl:pos_balls}.
Order the blocking sets in each $B_i$, $i\in[t+1,n_1+\cdots+n_{k}]$, so that the blocking sets highlighted above (the one blocking set for $e_{t+1}$ of size $\leq 1+(t-d)$, and one blocking set for each $e_i$, $i\in[t+2,n_1+\cdots+n_{k}]$ of size $\leq i-1-1$) are the first ones, or, equivalently, they are at the root at the tree $T_{e_i}$.
Translated into the bins-balls-wall configuration using Observation~\ref{obs:dict_bbw-sbd}, these highlighted sets generate balls at positions and delays
\begin{equation} \label{eq:pos_balls}
(1+(t-d),|e_{t+1}|),\; (t+2-1-1,|e_{t+2}|),\;\ldots,\;(n_1+\ldots+n_{k}-1-1,|e_{n_1+\cdots+n_{k}}|)
\end{equation}
or, if the corresponding blocking set is even smaller, then they are in a position more to the left of the indicated. As per Claim~\ref{cl:pos_balls} the edges $e_i$, $i\in[t+1,n_1+\cdots+n_k]$, may generate further balls, with strictly larger delays, yet all the balls related with the edge $e_i$, for $i\in[t+2,n_1+\cdots+n_{k}]$, are weakly to the left of $(i-1-1,|e_{i}|)$.

For the subhypergraphs $H_{A',j,t}(n_1,\ldots,n_{j+1},n_{j+2}',\ldots, n_k')$ (same $n_1,\ldots,n_{j+1}$ as $H$, with the terms $n_{j+2}',\ldots, n_k'$ as parameters, the edges are denoted by $e_i'$), the balls for the first $t$ edges have the same position and delay as for $H$, since they are the same edges. The balls of the edges indexed by $i\in[t+1,n_1+\ldots+n_{j+1}+n_{j+2}'+\cdots+n_k']$ have exactly one dominating set in each $B_i'$: for $e_{t+1}'$ it has size $1+(t-d')$ while for the other indices, the dominating set has size $=i-1-1$. These blocking sets give balls at positions:
\begin{align} \label{eq:pos_balls_2}
&(1+(t-d'),|e_{t+1}'|),\; (t+2-1-1,|e_{t+2}'|),\;\ldots,\;(n_1+\ldots+n_{j+1}-1-1,|e_{n_1+\cdots+n_{j+1}}'|), \\
&\;\; (n_1+\ldots+n_{j+1}+1-1-1,|e_{n_1+\cdots+n_{j+1}+1}'|),\;\ldots, \nonumber \\
&\qquad (n_1+\ldots+n_{j+1}+n_{j+2}'+\cdots+n_k'-1-1,|e_{n_1+\cdots+n_{j+1}+n_{j+2}'+\cdots+n_k'}'|) \nonumber
\end{align}
 with no additional balls.

The numbers $n_{j+2}',\ldots,n_k'$ are found as follows.
Consider the balls with delay $j+1$ in $H$ and those with delay $j+1$ in $H_{A',j,t}(n_1,\ldots,n_{j+1},0,\ldots,0)$; we can pair them as follows:
\begin{displaymath}
\begin{cases}
(1+(t-d'),|e_{t+1}'|) \leftrightarrow (1+(t-d),|e_{t+1}|), \\
(x_{t+2},|e_{t+2}|)\leftrightarrow (t+2-1-1,|e_{t+2}|),\\
\ldots,\\
(x_{n_1+\cdots+n_{j+1}},|e_{n_1+\cdots+n_{j+1}}|) \leftrightarrow (n_1+\cdots+n_{j+1}-1-1,|e_{n_1+\cdots+n_{j+1}}'|)
\end{cases}
\end{displaymath}
so that the ball from $H$ is weakly to the left from its paired ball in $H_{A',j,t}$. The wall is at the same position in both instances.
Process the balls simultaneously in pairs, beginning with the balls at smaller positions. Then the ball from $H$ generates at least as many balls as its pair in $H_{A',j,t}$. From the generated balls, repair them so that the descendents from the balls of $H$ with smaller position may become unpaired. Since the wall is at the same position, all the descendants from the balls in $H_{A',j,t}$ are paired with a ball descendent from $H$ at the same position.
This procedure is done until all the balls are processed.

Now the balls for $H$ with delay $j+2$ are: those descending from some ball of delay $j+1$, some balls coming from edges $e_i$, $i\leq n_1+\cdots+n_{j+1}$, which initially had larger delay (in some instances when $|B_i|\geq 2$), and some balls coming from edges $e_i$, $i\in [n_1+\cdots+n_{j+1}+1,n_1+\cdots+n_{j+1}+n_{j+2}]$. 

The balls from $H_{A',j,t}(n_1,n_2,\ldots,n_{j+1},0,\ldots,0)$ with delay $j+2$ are only those coming from descendants of edges $e_i$, $i\leq n_1+\cdots+n_{j+1}$, and thus they are all paired with some ball from $H$ with delay $j+2$.
Now consider  $H_{A',j,t}(n_1,n_2,\ldots,n_{j+1},n_{j+2}',\ldots,0)$ with $n_{j+2}'$ being the difference between the number of balls with delay $j+2$ in $H$ and the the number of balls with delay $j+2$ in $H_{A',j,t}(n_1,n_2,\ldots,n_{j+1},0,\ldots,0)$. Add the edges to $H_{A',j,t}(n_1,n_2,\ldots,n_{j+1},0,\ldots,0)$ one by one, and pair the new ball with the unpaired balls with delay $j+2$ with the following priority.
 
First pair the balls coming from the edges in $H$ indexed by $i\leq n_1+\cdots+n_{j+1}$ that have not been processed before (as they originally had delay $j+2$); all these blocking sets have size $\leq i-1-1$ (by the arguments leading to \eqref{eq:pos_balls} and the fact that all the edges leading up to $e_t$ have no balls with delay larger than $|e_i|$), thus these balls are at positions $\leq i-1-1$. Adding an edge $e_l'$ to $H_{A',j,t}$, $l\geq n_1+\cdots+n_{j+1}$ introduces a ball at position $=l-1-1$ which is to the right of the ball from $H$. Thus the balls from $H_{A',j,t}$ are to the right of the balls from $H$.

Second, pair the balls coming from edges $e_i$, $i\in [n_1+\cdots+n_{j+1}+1,n_1+\cdots+n_{j+1}+n_{j+2}]$, in $H$ with delay $j+2$.
Each edge induces exactly one ball with delay $j+2$, namely the ball from the root vertex from $T_{e_i}$. These are the balls highlighted in \eqref{eq:pos_balls}, thus they are at position $\leq i-1-1$ with delay $|e_i|$. The $n_{j+2}$ edges added to $H_{A',j,t}$ to pair them implies that $n_1+\ldots+n_{j+1}+n_{j+2}'\geq n_1+\ldots+n_{j+1}+n_{j+2}$; also the edge $e_{i'}'$ added to compensate the ball from $e_i$ is such that $i'\geq i$ (since we may have already add some edge in the previous step, and $n_1+\ldots+n_{j+1}\geq n_1+\ldots+n_{j+1}$), thus by \eqref{eq:pos_balls_2} induces a ball $(i'-1-1,|e_{i'}'|)$ which is to the right of the one from $e_i$.

Third, those balls coming from edges in $H$ indexed $i\leq n_1+\cdots+n_{j+1}$ that had a parent ball; the parent was paired (we are repairing the balls from $H_{A',j,t}$ to those of larger positions) with a ball $(s,|e_{i'}|)$ from an edge $e_{i'}'$ in $H_{A',j,t}$, with $s$ strictly (we are assuming that it gives some unpaired edges) to the right from the ball from $H$, the added edge of size $j+2$ gives a ball that is to the right of $(s,|e_{i'}|)$ (see \eqref{eq:pos_balls_2}) as it has a larger index; note that the parent may leave induce some additional unpaired balls, but all of those have consecutive positions, as so are the positions of the added balls in $H_{A',j,t}$ (see \eqref{eq:pos_balls_2}). Thus all the paired balls from $H_{A',j,t}$ are to the right of the balls from $H$.

Therefore, we have a matching number of balls, all the balls from $H_{A',j,t}$ are to weakly to the right of the balls from $H$, and we have $n_1+\ldots+n_{j+1}+n_{j+2}'\geq n_1+\ldots+n_{j+1}+n_{j+2}$. Therefore we run the iteration $j+2$. Then we iterate and keep repeating the previously described procedure, to conclude that the $n_{j+2}',\ldots,n_k'$ can always be found, so that $H_{A',j,t}(n_1,\ldots,n_{j+1},n_{j+2}',\ldots,n_k')$ match the wall positions of $H$. Thus the claim follows. Before finishing, let us mention that, since the position of the wall is $\leq n-k$, then there are less than $n-k$ edges in $H_{A',j,t}(n_1,\ldots,n_{j+1},n_{j+2}',\ldots,n_{j+r}',0,\ldots,0)$, and so always a vertex of degree $0$  so that an edge can be added, and the procedure can continue.
\end{proof}

\begin{proposition}[Construction $A$ from Section~\ref{sec:cons_a_b} is \emph{as good as} $A'$] \label{prop:cona_best}
	If $S$ is an extremal family of depth $k-j$ with hypergraph $H_{A',j,t}(n_1,\ldots,n_{k})$ then there exists a sequence $(n_1,\ldots,n_{j+1},n_{j+1}',\ldots,n_{k}')$ of non-negative integers with the property that the family induced by $H_{A,j}(n_1,\ldots,n_{j+1},n_{j+1}',\ldots,n_{k}')$ is extremal and has the same $k$-binomial decomposition as $S$.
\end{proposition}

\begin{proof}The positions of the balls from the edges $e_i$, $i\in[n_1+\cdots+n_{j}+1,n_1+\cdots+n_{j}+n_{j+1}]$ in $H_{A',j,t}$ are
\begin{equation}\label{eq:balls_in_ap}
n_1+\ldots+n_{j}-1,\;\ldots,\;t-1,\qquad 1+(t-d),\; t+2-1-1,\;\ldots,\;n_1+\ldots+n_{j+1}-1-1
\end{equation}
whereas for $H_{A,j}$ are
\begin{equation}\label{eq:balls_in_a}
n_1+\ldots+n_{j}-1,\ldots,n_1+\ldots+(n_{j+1}-1)-1,n_1+\ldots+n_{j+1}-1-d'
\end{equation}
where $d'$ 
is the degree of the vertex with smaller degree with the condition  $d'\geq 2$ in $\{e_1,\ldots,e_{n_1+\ldots+n_{j+1}-1}\}\subset E(H_{A,j})$ 
(the vertex with minimum degree $d'$ is $v_{s-1}$, where $s$ is the size of the edge $e_{n_1+\ldots+n_{j+1}-2}$).

There are now several cases involving the relation between $t+1$ as defined in Section~\ref{sec:cons_a_b}, and $n_1+\cdots+n_{j+1}$. These cases are used later as well.
\begin{enumerate}[label=(C\arabic*)]
		\item \label{en:alc1} $e_{t+1}$ is the first edge added to $H^{(j+1)}$ with respect to $H^{(j)}$. There are two cases:
		\begin{enumerate}[label=(C\arabic{enumi}.\alph*)]
			\item \label{en:alc11} The previous jump (in the $k$-binomial decomposition $(b_0,\ldots,b_{k-1})$), which is also larger than $1$, has difference equal to $2$. (So there exists an $s\geq 1$ for which $b_{j-i}-b_{j-1-i}=1$ for all  $i\in [1,s-1]$, and then $b_{j-s}-b_{j-1-s}=2$.)
			\item\label{en:alc12} The previous jump (in the $k$-binomial decomposition) larger than $1$ is strictly larger than $2$.
			(So there exists an $s\geq 1$ for which $b_{j-i}-b_{j-1-i}=1$ for all  $i\in [1,s-1]$, and then $b_{j-s}-b_{j-1-s}\geq 3$.)
		\end{enumerate}
		\item \label{en:alc2}  $e_{t+1}$ is the second edge added to $H^{(j+1)}$ with respect to $H^{(j)}$.
		\item \label{en:alc3} $e_{t+1}$ is at least the third edge added to $H^{(j+1)}$ with respect to $H^{(j)}$.
\end{enumerate}

\noindent The positions of the balls in \eqref{eq:balls_in_ap} and \eqref{eq:balls_in_a} is different, but they are equivalent in the cases:
\begin{itemize}
	\item Case~\ref{en:alc3} and Case~\ref{en:alc2} if there are only two edges of size $j+1$ (so Case~\ref{en:alc3} does not occur),
	\item Case~\ref{en:alc11} and Case~\ref{en:alc12} if there is only one edge of degree $j+1$.
\end{itemize}
	In the rest of the cases, the position of the balls is strictly more favorable to $H_{A,j}$ versus $H_{A',j,t}$. The remaining of the argument is similar as in Proposition~\ref{prop:conap_best} (namely the addition of edges whenever needed to $H_{A,j}^{(r)}$ with respect to $H_{A',j,t}^{(r)}$); as all the paired balls are weakly to the right in $H_{A,j}^{(r)}$ with respect to those in $H_{A',j,t}^{(r)}$, if $H_{A',j,t}$ is extremal (the wall satisfies $w_k\leq n-k$), then so is the corresponding $H_{A,j}$.
\end{proof}

\subsubsection{If \texorpdfstring{$S$}{S} satisfies \texorpdfstring{\ref{exB}}{21}, then Construction \texorpdfstring{$B$}{B} from Section~\ref{sec:cons_a_b} is \emph{better}}

\begin{proposition}
	 \label{prop:conb_best}
	If $S$ is an extremal family of depth $k-j$ with hypergraph $H$ that satisfies \ref{exB} for $H^{(j+1)}$ (actually for $\{e_1,\ldots,e_{t+1}\}$, with $t=\text{st}(E(H))$), then there exists a sequence $(n_1',\ldots,n_{k}')$ of non-negative integers with the property that the family induced by $H_{B',j,t}(n_1',\ldots,n_{k}')$ is extremal and has the same $k$-binomial decomposition as $S$.
\end{proposition}

\begin{proof}[Proof of Proposition~\ref{prop:conb_best}]
	Let $(n_1,\ldots,n_{k})$ be the sequence of the number of edges of each size for $H$. We use the vertex notation from Section~\ref{sec:hypergraph_of_colex}. The following claim follows from examining the case analysis in \ref{exB}.
	\begin{claim} \label{cl.4}
		If condition \ref{exB} is to be satisfied, then $|e_{t}\setminus e_{t+1}|\geq 2$.
	\end{claim}

\noindent Claim~\ref{cl.4} could be applied for all the other edges $e_r$, $r\in[t+1,n_1+\cdots+n_{j+1}]$, and thus we conclude that each edge $e_r$, $r\in[t+1,n_1+\cdots+n_{j+1}]$ is such that $e_r$ does not cover $e_t$ in, at least, two vertices.

This implies that, for each $r\in[t+1,n_1+\cdots+n_{j+1}]$, the family of dominating sets $B_r$ will have at least two sets $P_1$, $P_2$, one containing a subset of $u_1,\ldots,u_{t-1},w_{r,1}$ and another containing a subset of $u_1,\ldots,u_{t-1},w_{r,2}$, where $w_{r,1}$ and $w_{r,2}$ are two of the vertices in $e_t$ avoided by $e_r$; the first subset always contains $w_{r,1}$ and the second always contains $w_{r,2}$.
	
Let us further assume that these two sets $P_1$ and $P_2$ have both size $r-1$. In that case, we can further assume that $P_1$ is obtained from $P_2$ by exchanging $w_{r,1}$ with $w_{r,2}$. Then, we can order the dominating sets in $B_r$ to make $P_1$ and $P_2$ the first and second dominating sets; this ordering implies that the tree corresponding to $e_{r}$ has (at least) two leaf-to-root paths, one corresponding to the binomial coefficient $\binom{n-|e_r|-(r-1)}{k-|e_r|}$ (ball at position $r-1$ and with delay $|e_r|$) and another one to the binomial coefficient $\binom{n-|e_r|-(r-1)-1}{k-|e_r|-1}$ (ball at position $r-1$ and with delay $|e_r|+1$). 
	The other case is that one of these sets have size $<r-1$; then we place such set as the first dominating set in $B_r$, and this is associated with a binomial coefficient $\binom{n-|e_r|-s}{k-|e_r|}$ (ball at position $s$ with $s<r-1$ and with delay $|e_r|$).

	Now we do a case analysis to construct $H_{B',j,t}$ and compare it with $H$ following the cases for the edges $e_{t+1}$ in $H$ described in the proof of Proposition~\ref{prop:cona_best}. The argument works with $n_1=n_1',\ldots,n_{j+1}=n_{j+1}'$, and some $n_i'\geq n_i$, for $i\geq [j+2,k]$, to be determined depending by $H$, and given by the procedure described below.
	
	\medskip
		\noindent 
		\emph{Case~\ref{en:alc11}.} For $t+1$, and with an appropriate ordering 
		(and using the arguments just described), 
		$e_{t+1}$ 
		induces (at least) two balls (for $H$); one associated to the binomial coefficient $\binom{n-|e_{t+1}|-s}{k-|e_{t+1}|}$ with $s\leq t$ (associated to the dominating set
		\[
		[\text{$v_{i}$ for some $i\leq j-1$}]\;  \cup \; \{u_{j'}\}_{e_{j'} \text{ not containing }v_i},\] and another ball given by the binomial coefficient $\binom{n-|e_{t+1}|-t-1}{k-|e_{t+1}|-1}$ (associated to the dominating set 
		\[\{u_1,\ldots,u_t\},\] where, using the notation from \eqref{eq:picked_avoided}, we have ``picked'' $v_i$).
		These two binomial coefficients induce, respectively, a ball at position $s\leq t$ and delay $|e_{t+1}|$, and a ball at position $t$ with delay $|e_{t+1}|+1$ (the $+1$ occurs as one vertex, $v_i$ has been ``picked'').

		In  $H_{B',j,t}$, $e_{t+1}$ gives exactly two dominating sets which, if ordered appropriately, are associated to the binomial coefficient 
		$\binom{n-|e_{t+1}|-t}{k-|e_{t+1}|}$ and $\binom{n-|e_{t+1}|-t-1}{k-|e_{t+1}|-1}$ which gives, respectively, a ball at position $t$ with delay $|e_{t+1}|$, and a ball at position $t$ with delay $|e_{t+1}|+1$.
		
		For $r\in(t+1,n_1+\cdots+n_{j+1}]$, $e_r\in E(H)$ induces at least one of the following two options.
		\begin{itemize}
			\item If both blocking sets $P_1$ and $P_2$ described above have exactly size $r-1$, then we obtain two balls/binomial coefficients: one ball is at position $r-1$ and has delay $|e_r|$, and another ball is at position $r-1$ and delay $|e_r|+1$. These are associated to the two binomial coefficients: $\binom{n-|e_r|-(r-1)}{k-|e_r|}$ and  $\binom{n-|e_r|-(r-1)-1}{k-|e_r|-1}$.
			\item Otherwise, we obtain a ball at position $s<r-1$ and delay $|e_r|$ associated to a binomial coefficient $\binom{n-|e_r|-(s)}{k-|e_r|}$.
		\end{itemize} 
		On the other side in $H_{B',j,t}$, $e_{r}$, $r\in(t+1,n_1+\cdots+n_{j+1}]$, gives exactly two dominating sets which are associated to the binomial coefficients 
		$\binom{n-|e_{r}|-(r-1)}{k-|e_{r}|}$ and $\binom{n-|e_{r}|-(r-1)-1}{k-|e_{r}|-1}$ and that give a ball at position $r-1$ and delay $|e_r|$, and a ball at position $r-1$ and delay $|e_r|+1$ respectively.
		
		In particular, we see that, at the delay $j+1=|e_{t+1}|$, both configurations have the same number of balls, with the balls from $H_{B',j,t}$ being weakly to the right of the corresponding balls in $H$.
		For larger delays, the balls in $H$ either are matched with the balls from $H_{B',j,t}$, or they create a ball that compensates (it is weakly to the left) the corresponding ball from $H_{B',j,t}$ (in the case that $B_r$ has a blocking set of size $<r-1$).

	The terms $n_i', i\in[j+2,k]$, for $H_{B',j,t}$ are found using the same strategy as in Proposition~\ref{prop:conap_best}. However, each additional edge $e_{i'}'$ for $H_{B',j,t}$ gives exactly one ball at position $i'-1$ and delay $|e_{i'}'|$, which are always weakly to the left with the ball being paired with coming from $H$.
		
		\medskip
		\noindent \emph{Case~\ref{en:alc12}} is similar as Case~\ref{en:alc11}, but this time we have one ball at position $t-d+1$, $d\geq 2$, the degree of $v_{|e_{t}|-1}$ for the construction $B'$, and one at position $t$ and delay $j+1=|e_{t+1}|$, while in the other case the degree is also $t-d'+1$ with $d'\geq d\geq2$ for $e_{t+1}$, and a ball at position $t$ and delay $j+1=|e_{t+1}|$. For the rest of the edges $e_{r}$, we have at least a ball at position at most $r-1-d+1$, $d\geq 2$, for $H$, and two balls at position $r-1$ (one with a delay of $j+1=|e_{t+1}|$) in the case of $H_{B',j,t}$, which is better than for $H$. The remaining sizes of edges are handled as in the previous case.

		\medskip
		\noindent \emph{Case~\ref{en:alc2}} is similar to \ref{en:alc11}, and \emph{Case~\ref{en:alc3}} is similar to \ref{en:alc12}.
\end{proof}

\paragraph{Construction $B$ is \emph{better} than Construction $B'$.}
Lets observe how we can pair the balls from $H_{B,j}$ with those from $H_{B',j,t}$.
$H_{B,j}$ give balls associated with the binomial coefficients:
\begin{align}
&\binom{n-|e_{n_{1}+\ldots+n_{j+1}}|-(n_1+\ldots+n_j)}{k-|e_{n_{1}+\ldots+n_{j+1}}|},\binom{n-|e_{n_{1}+\ldots+n_{j+1}}|-(n_1+\ldots+n_j+n_{j+1}-1)-1}{k-|e_{n_{1}+\ldots+n_{j+1}}|-1}, \nonumber \\
&\qquad \left\{
\binom{n-|e_{j}|-(j-1)}{k-|e_{j}|}\right\}_{j\in [n_{1}+\ldots+n_{j}+1,\;n_{1}+\ldots+n_{j+1}-1]\sqcup [n_{1}+\ldots+n_{j+1}+1,\;n_{1}+\ldots+n_{k}] }\nonumber
\end{align}
Note that the coefficients 
\[\left\{
\binom{n-|e_{j}|-(j-1)}{k-|e_{j}|}\right\}_{j\in [n_{1}+\ldots+n_{j},\;n_{1}+\ldots+n_{j+1}-1]\sqcup [n_{1}+\ldots+n_{j+1}+1,\;n_{1}+\ldots+n_{k}] }\] have a matching one in $H_{B',j,t}$.
Then the coefficients \[\binom{n-|e_{n_{1}+\ldots+n_{j+1}}|-(n_1+\ldots+n_j)}{k-|e_{n_{1}+\ldots+n_{j+1}}|},\binom{n-|e_{n_{1}+\ldots+n_{j+1}}|-(n_1+\ldots+n_j+n_{j+1}-1)}{k-|e_{n_{1}+\ldots+n_{j+1}}|-1}\]
when compared with the unmatched $H_{B',j,t}$, which are
{\small\begin{align}\binom{n-|e_{t+1}|-(t)+d'-1}{k-|e_{t+1}|}&,\binom{n-|e_{t+1}|-(t)-1}{k-|e_{t+1}|-1},
	\binom{n-|e_{t+2}|-(t+1)-1}{k-|e_{t+2}|-1},\ldots,\nonumber \\
	&\qquad \binom{n-|e_{n_{1}+\ldots+n_{j+1}}|-(n_1+\ldots+n_{j+1}-1)-1}{k-|e_{n_1+\ldots+n_{j+1}}|-1} \nonumber 
	\end{align}}
where $d'$ is the degree of $v_{|e_{t}|-1}$ in $\{e_1,\ldots,e_t\}$, which equals $t-(n_1+\ldots+n_j-1)$, so we have 
{\small\begin{align}\binom{n-|e_{t+1}|-(n_1+\ldots+n_j)}{k-|e_{t+1}|}&,\binom{n-|e_{t+1}|-(t)-1}{k-|e_{t+1}|-1},
	\binom{n-|e_{t+2}|-(t+1)-1}{k-|e_{t+2}|-1},\ldots, \nonumber \\
	&\qquad\binom{n-|e_{n_{1}+\ldots+n_{j+1}}|-(n_1+\ldots+n_{j+1}-1)-1}{k-|e_{n_1+\ldots+n_{j+1}}|-1} \nonumber
	\end{align}}
and then we see that $H_{B,j}$ creates fewer balls at step $j+1$, as it does create balls at the same position as $H_{B',j,t}$ but with no delayed balls.
Then we follow the same procedure as with Proposition~\ref{prop:conb_best} but for $H_B$ instead as for $H_{B'}$, thus adding edges as necessary. All the edges added to construction $B$ are always weakly to the right of those balls that are being matched from the $B'$ construction.

\subsubsection{Algorithmic complexity of the decision problem}

After showing Propositions~\ref{prop:conap_best}, \ref{prop:cona_best}  \ref{prop:conb_best}, and the last comment on how the Construction $B$ is better than Construction $B'$, the only part of Theorem~\ref{thm:decision} left to shown is the algorithmic part.

The algorithm works as follows. With the desired depth $k-j$ and the desired $k$-binomial decomposition given, we first try the construction $A$, adding edges as necessary to adjust the shadow $k$-binomial decomposition terms to the $k$-binomial decomposition desired. The following are the things we check.
\begin{itemize}
	\item If, at some point in iteration $i$ (either by processing the previous obtained balls, or when adding the balls associated to the new edges), we obtain a wall that is strictly below $n-i$, then we stop the process and determine that there is no extremal family following the contruction $A$.
	\item  If $i<k$ and we are already creating enough balls so that, in the next iteration, the wall will go below $n-(i-1)$, then we also stop the process and determine that there is no extremal family following the contruction $A$.
	\item  If, when processing the previouly found balls at step $i$, and before adding new edges, the wall goes below the desired term in the $k$-binomial decomposition, then we stop the process and determine that there is no extremal family following the contruction $A$ with the desired $k$-binomial decomposition.
\end{itemize} 
If the process finishes and we obtain a valid construction for $A$, then we have obtained an extremal family with the desired depth and $k$-binomial decomposition.
If the process for $A$ has failed, we try the construction $B$ following an analogous procedure and stopping conditions but for Construction $B$ instead than Construction $A$. If the process finishes and we obtain a valid construction for $B$, then we have obtained an extremal family with the desired depth and $k$-binomial decomposition. If the process for $B$ has also failed, then we determine that, for the depth $k-j$ and the $k$-binomial decomposition given, there does not exists any extremal family.

The correctness of the algorithm follows from Propositions~\ref{prop:conap_best}, \ref{prop:cona_best}  \ref{prop:conb_best} and the last comment on how the Construction $B$ is better than $B'$.

Regarding the running time, we observe that all we are doing is placing and processing balls using the bins-balls-wall process. There are at most $n$  balls at each step (if the family is extremal then there are at most $n$ balls to be processed at each step, and if there is a step in which the number of balls increases past $n$, then we can determine that the family is not extremal as the wall will go pass $n-k$, and the process is then stopped, also note that, in the last step $k$ we shall not create the new balls, since there is no further step), and there are at most, $k$ iterations, then the time $O(nk)$ follows.

This finishes the proof of Theorem~\ref{thm:decision} and Theorem~\ref{thm:alg}.

\section{On counting extremal families} \label{sec:count}

We give now some comments on counting the number of extremal families.

In this work we have given some results that translate properties from extremal families into information of their hypergraph.
If $H$  is the hypergraph of $S\subseteq \binom{[n]}{k}$ ($S$ not necessarily extremal), and has $(n_1,\ldots,n_k)$ as its sequence of edge sizes, then: the subhypergraph $\overline{H}$ of $H$ 
giving an initial segment of the colex order and whose sequence of number of edges  per size, $(n_1',\ldots,n_k')$, is maximal in the lexicographical order, is said to be the \emph{maximal colex subhypergraph of $H$}.
Using the work previously presented we have the following properties:
\begin{itemize}
	\item If $S$ is extremal, $n_i=n_i'$ $i\leq c\cdot \log\log(n)$, by Theorem~\ref{thm:bound_depth}.
	\item Following the notation of vertices $u_i$ and $v_i$ from Section~\ref{sec:hypergraph_of_colex}, and if $e$ is an edge of size $j$ in $H$ but not in $\overline{H}$, then there is a vertex in $v_1,\ldots,v_{j-1}$ that $e$ does not cover. This implies that $e$ generates a ball at certain position with respect to $\overline{H}$.
	\item From all the subhypergraphs of $H$, $\overline{H}$ is the one that gives the smallest segment in the colexicographical order. Since $\overline{H}$ is a subhypergraph of $H$, the family of $\overline{H}$ contains $S$.
\end{itemize}
If we want $S$ to be an extremal family, then there are some restriction on the $n_i'$ and the possible extensions of $\overline{H}$ into $H$, a hypergraph with $n_i$ being  the number of edges with size $i$, $i\in[k]$. \footnote{The family of $\overline{H}$ is the smallest initial segment of the colex order that contains $S$ and that it comes from a subhypergraph of $H$. In general, $\overline{H}$ is not the hypergraph of the smallest initial segment in the colex order that contains $S$. For that we have to do the following. Let $H'=H'(n_1',\ldots,n_k')$ be the hypergraph with $\{f_i\}$ as its edge set; then $H'$ is said to be \emph{valid for $H$} if, for each $i$, there exists a $j\in [|E(H)|]$ such that $e_j\subset f_i$ (each edge $f_i$ contains an edge in $H$). The \emph{minimal colex containing $S$} is the valid hypergraph $H'$ that is the hypergraph of the colex, and that $(n_1',\ldots,n_k')$ is maximal in the lexicographical order from among all valid hypergraphs giving initial segments of the colex order.}
Assume that the edges are in a comfortable ordering and that the first $n_i'$ edges from $[n_1+\ldots+n_{i-1}+1,n_1+\ldots+n_{i-1}+n_i']$ are precisely those from $\overline{H}$. Then the other edges of size $i$ do not contain any previous edge of size $\leq i$. Further, we should account for the balls induced and their descendants (which grow at a doubly exponential rate), and the total number of edges (actually, of balls) is bounded by $n-k$.

This argument points towards saying that there there many restrictions for a family being extremal, and even more if the depth of the family is high. However, coming with a precise counting seems a difficult task, especially if one considers the possible automorphisms of the family. In the next section we present to some results when the families are considered without isomorphisms (so that there is a unique initial segment of the colex order in $\binom{[n]}{k}$).

\subsection{Counting without isomorphisms}

Theorem~\ref{thm:count_1} in the Introduction, which we now proceed to prove, shows that most of the extremal families are obtained by removing some sets from the initial segment in the colex oder. These are the extremal families constructed using \cite[Proposition~2.5-2.4]{furgri86} when the initial family from which we subtract some sets is the initial segment in the colex order.

\begin{proof}[Proof of Theorem~\ref{thm:count_1}]
Using the arguments in Proposition~\ref{prop.translation} and the relation between the walls, the balls and the $k$-binomial decomposition (which in this case is essentially the fact that there are no hyperedges inside each other), each time we have a hypergraph with $n_i$ edges of size $i$, we conclude that the number of edges of size $j$ that we can choose is
\[
\leq \binom{n-n_1-1}{j}+\binom{n-n_1-n_2-2}{j-1}+\cdots+\binom{n-n_1-\cdots-n_{j-1}-(j-1)}{2}+\binom{n-n_1-\cdots-n_{j-1}-(j-1)}{1}
\]
Further, as the hypergraph of an extremal family has, at most, $n-k$ edges, then the number of hypergraph corresponding to extremal families is, at most,
\begin{align} 
&\sum_{\substack{n_1,\ldots,n_k\in \mathbb{Z}, n_i\geq 0 \\ n_1+\cdots+n_k\leq n-k}} \binom{n}{n_1} \binom{\binom{n-n_1}{2}}{n_2}
\binom{\binom{n-n_1-1}{3}+\binom{n-n_1-n_2-1}{2}}{n_3} \cdots \nonumber\\
&\qquad\qquad\cdots \binom{\binom{n-n_1-1}{k}+\binom{n-n_1-n_2-2}{k-1}+\cdots+\binom{n-n_1-\cdots-n_{k-1}-(k-2)}{2}}{n_k} :=\nonumber \\
&\qquad:=\sum_{\substack{n_1,\ldots,n_k\in \mathbb{Z}, n_i\geq 0 \\ n_1+\cdots+n_k\leq n-k}} f_{1,k}(n_1,\ldots,n_k)
\label{eq:upper_count}
\end{align}
Let us also define:
\begin{align}
f_{3,k}(n_1,\ldots,n_{k-1})&:=\binom{n-n_1-1}{k}+\binom{n-n_1-n_2-2}{k-1}+\cdots\nonumber \\
&\qquad \cdots+\binom{n-n_1-\cdots-n_{k-2}-(k-2)}{3}+\binom{n-n_1-\cdots-n_{k-1}-(k-2)}{2} \nonumber 
\end{align}

Now, compare \eqref{eq:upper_count} with just the instances in which the edges up to size $k-1$ are placed to give the initial segment in the colex order. Since we are assuming that the elements of $[n]$ are labelled, there is only one such configuration, and thus we obtain
{\small\begin{equation} \label{eq:count_rem_colex}
\sum_{\substack{n_1,\ldots,n_k\in \mathbb{Z}, n_i\geq 0 \\ n_1+\cdots+n_k\leq n-k}} \binom{\binom{n-n_1-1}{k}+\binom{n-n_1-n_2-2}{k-1}+\cdots+\binom{n-n_1-\cdots-n_{k-1}-(k-2)}{2}}{n_k}:=\sum_{\substack{n_1,\ldots,n_k\in \mathbb{Z}, n_i\geq 0 \\ n_1+\cdots+n_k\leq n-k}} f_{2,k}(n_1,\ldots,n_k)
\end{equation}}
This corresponds to remove $n_k$ edges from the initial segment of the colex order.
The statement of the proposition is them implied by the following claim. For any fixed $k\geq 0$,
\[
\left[\sum_{\substack{n_1,\ldots,n_k\in \mathbb{Z}, n_i\geq 0 \\ n_1+\cdots+n_k\leq n-k}} f_{1,k}(n_1,\ldots,n_k)\right]\Big/\left[\sum_{\substack{n_1,\ldots,n_k\in \mathbb{Z}, n_i\geq 0 \\ n_1+\cdots+n_k\leq n-k}} f_{2,k}(n_1,\ldots,n_k)\right]\stackrel{n\to \infty}{\to} 1
\]
Indeed, for any $j\leq n-k$, we have
{\small \begin{align}
\sum_{\substack{j>n_i\geq 0 \\ n_1+\cdots+n_k=j}} &f_{1,k}(n_1,\ldots,n_k)=
\sum_{\substack{n_k=t,t\in[0,j)\\ n_1+\cdots+n_k=j}} f_{1,k}(n_1,\ldots,n_{k-1},t) \nonumber \\
&\leq 
\sum_{n_k=t,t\in[0,j)} \binom{j-t+(k-2)}{k-2} f_{1,k}(0,\ldots,0,j-t,t) \nonumber \\
%
&= \sum_{n_k=t,t\in[0,j)} \binom{j-t+(k-2)}{k-2} \binom{\binom{n}{k-1}}{j-t}  \binom{f_{3,k}(j-t)}{j}
\frac{j \cdots (t+1)}
{(f_{3,k}(j-t)-t)(f_{3,k}(j-t)-t-1)\cdots (f_{3,k}(j-t)-j+1)} \nonumber \\
&\qquad \text{with }f_{3,k}(j-t)=\binom{n-1}{k}+\binom{n-2}{k-1}+\cdots+\binom{n-(k-4)}{3}+\binom{n-(k-4)-(j-t)}{2} \nonumber \\
&\stackrel{k \text{ is fixed}, n\to \infty}{\approx} \binom{f_{3,k}(j-t)}{j}
 \sum_{n_k=t,t\in[0,j)}
  \frac{n^{(k-1)(j-t)}}{n^{k(j-t)}}(j-t)^{k-2} \nonumber \\
 &=\binom{f_{3,k}(j-t)}{j} \sum_{n_k=t,t\in[0,j)}   \frac{1}{n^{(j-t)}}(j-t)^{k-2}= \binom{f_{3,k}(j-t)}{j} \sum_{t=0}^{j-1}   \frac{1}{n^{(j-t)}}(j-t)^{k-2} \nonumber 
\end{align}}
which goes to zero as $n\to \infty$, thus showing the result. Indeed, if $\frac{a_{i}(n)}{b_i(n)}\to_{n\to \infty} 1$ with a certain, uniform rate, then $\frac{\sum_i a_i(n)}{\sum_i b_i(n)}\to_{n\to \infty} 1$, and we are observing that the term when $n_k$ is not maximal can be subsummed, multiplicatively, to the term where $n_k$ is maximal.
\end{proof}

If we restrict ourselves to families having the same $k$-binomial decomposition, then we can show the following.
\begin{proposition} \label{prop:count_3}
	Consider the families as labelled. Fix $k>0$ and let $t_0,\ldots,t_{k-1}$ be some positive integers.
	Let $\mathcal{E}(t_0,\ldots,t_{k-1})$ be the set of extremal families in $\binom{[n]}{k}$ that have $k$-binomial decomposition:
	\[
	\binom{n-t_0-1}{k}+\cdots+\binom{n-t_{k-2}-(k-1)}{2}+\binom{n-t_{k-1}-(k-1)}{1}
	\]
	Let $\mathcal{E}_0(t_0,\ldots,t_{k-1})$ be the set of extremal families in $\mathcal{E}(t_0,\ldots,t_{k-1})$ that are obtained from the initial segment in the colex order by removing some edges of size $k$. Then
	\[
	\lim_{n\to\infty} \frac{|\mathcal{E}(t_0,\ldots,t_{k-1})|}{|\mathcal{E}_0(t_0,\ldots,t_{k-1})|}=1
	\]
\end{proposition}

Construction $A$ and $B$ can be seen as taking the role of the initial segment in the colex order when the depth is introduced.
\begin{theorem} \label{thm:count_2}
Consider the families as labelled. Fix $k>i\geq 0$ and let $t_0,\ldots,t_{k-1}$ be some positive integers.
	Let $\mathcal{E}(i,t_0,\ldots,t_{k-1})$ be the set of extremal families in $\binom{[n]}{k}$
	such that $\Delta^i(E)$ is the initial segment of the colex order, but $\Delta^{i-1}(E)$ is not and that have $k$-binomial decomposition:
	\[
	\binom{n-t_0-1}{k}+\cdots+\binom{n-t_{k-2}-(k-1)}{2}+\binom{n-t_{k-1}-(k-1)}{1}
	\]
Let $\mathcal{E}_0(i,t_0,\ldots,t_{k-1})$ be the set of extremal families in $\mathcal{E}(i,t_0,\ldots,t_{k-1})$ that are obtained from Construction $A$ or Construction $B$ (or one giving the same balls as one of them up to delay $k$) by removing some edges of size $k$. Then
\[
\lim_{n\to\infty} \frac{|\mathcal{E}(i,t_0,\ldots,t_{k-1})|}{|\mathcal{E}_0(i,t_0,\ldots,t_{k-1})|}=1
\]
\end{theorem}

\begin{proof}[Proof of Theorem \ref{thm:count_2} and \ref{thm:count_3}.]
	The proofs follow the same lines as in the proof of Theorem~\ref{thm:count_1} but observing that the choices for the edges of length $k$ in the hypergraph are reduced in all the other instances of extremal families the initial segment in the colex order for Theorem~\ref{thm:count_3} and for other instances of Construction $A$ or Construction $B$ (or one giving the same set of ball up to delay $k$). We also use the fact that $t_1,\ldots,t_{k-1}$ are fixed while $n$ grows in both instances.
\end{proof}

\section{Maximal chains} \label{sec:maximal_chains}

Given a family $S$, extremal or not, a natural question is the following. How many elements $|E|$ should we add to $S$ until:
\begin{enumerate}[label*=(Q\arabic*)]
	\item\label{en:q1} $S\cup E$ is an extremal family, or
	\item\label{en:q2} $S\cup E$ is such that there exists a sequence $\{e_1,\ldots,e_m\}$ with 
	\begin{enumerate}[label*=(\arabic*)]
		\item \label{en:cond1} all $\{S\cup E\cup\{e_1,\ldots,e_i\}\}_{i\in [m]}$ are extremal, and
		\item \label{en:cond2} $S\cup E\cup\{e_1,\ldots,e_m\}=\binom{n}{k}$
	\end{enumerate} 
\end{enumerate}
We do not give a comprehensive and precise answer to the \ref{en:q1}.
However, one can give an upper bound on the minimal size of $E$.
If $H$ is the family of sets of $S$, we can find all the $H'\subset H$ such that $H'$ is the hypergraph of an extremal family $S'$, then $S\subset S'$ and we have
\[
|E|\leq \min_{|S'|}\{ S' \text{ is extremal with hypergraph }H', H'\subset H\} -|S|
\]
that is, we consider all the possible hypergraphs that are subhypergraphs of $H$ and that give an extremal family, and then consider the smallest of such families (which is a superfamily of $S$).

For \ref{en:q2} we are asking when do we land in a maximal chain of extremal sets (maximal in the poset of the families ordered by inclusion). We can give a more precise answer, thanks to Theorem~\ref{t.inest} and Theorem~\ref{thm:extending_families}.
In particular, $S\cup E$ satisfies \ref{en:cond1} and \ref{en:cond2} if and only if $H(S\cup E)$ is extremal, and $H^{(k-1)}(S\cup E)$ is the hypergraph of the segment in the colex order. Therefore, the family $S'$ given by the hypergraph $H'=H^{(k-1)}(S\cup E)$ is the minimal colex with no term $\binom{b_{k-1}}{1}$ in the binomial decomposition, and such that contains $S$.
Therefore, $H'$ is given by a set of edges $\{f_1,\ldots,f_m\}$, each $f_i$ satisfying $e_j\subseteq f_i$ for some $e_j\in E(H)$, and $|f_i|\leq k-1$, and such that it gives an initial segment in the colex order, and the colex has minimum size with the edge containtment property. Observe that this hypergraph is not necessarily the maximal hypergraph $\overline{H}$ given in Section~\ref{sec:count};
for instance, the hypergraph with edges 
\[
(1,2),(1,3),(1,4,5),(1,4,6),(1,4,7),(1,4,8,9,10),(11,12)
\]
has the hypergraph with edges
\[
(1,2),(1,3),(1,4,5),(1,4,6),(1,4,7),(1,4,8,9,10),
\]
as its ``minimal colex'' formed by the edges containing it, yet the minimal colex containing it comes from the edges
\[
(1,2),(1,3),(1,4,5),(1,4,6),(1,4,7),(1,4,11,12),
\]
which is not a subhypergraph.

We do not know the complexity of finding such $H'$, but an exhaustive search is not necessary. Indeed, the following procedure that considers the edges by increasing size is possible. For each possible size of edge, increasingly, add the maximum number of edges $f_i$ to the previously found maximum configurations (with the containtment restriction $e_j\subseteq f_i$, and forming an initial segment in the colex order); for the next step, one should consider all the possible configurations  that have added the maximum number of edges to the configurations that have been considered previously.\footnote{In this section we are considering the families up to isomorphisms; it is therefore possible to have several labelled hypergraphs giving initial segments in the colex order, all of which being isomorphic.}


\bibliographystyle{abbrv}
\bibliography{bib_john.bib}

\end{document}